\documentclass[amssymb,amsfonts]{amsart}
\usepackage{amssymb, latexsym}
\usepackage[T1]{fontenc}
\usepackage[english]{babel}
\usepackage[intlimits,sumlimits,noDcommand,narrowiints]{kpfonts}
\usepackage{microtype}
\usepackage{svn}
\usepackage{graphicx}
\usepackage[pdfusetitle]{hyperref}
\usepackage{xspace}
\usepackage[geom,slash,sets,norm,pde,miscops,roman]{www_math_commands}

\newcommand*\graychooser{}
\renewcommand*\graychooser{_gray}

\newcommand*\tocaddline{\addtocontents{toc}{\protect\vspace{1.5ex}}}

\newcommand*\dS{\mathcal{M}_{\mathrm{dS}}}
\newcommand*\rotvf[1]{\Omega_{(#1)}}
\newcommand*\ooo[1]{\mathring{#1}}
\newcommand*\cpmc{\texorpdfstring{\ensuremath{\mathrm{C}_+\mathrm{MC}}}{C+MC}\xspace}
\newcommand*\vmc{\texorpdfstring{\ensuremath{\mathrm{C}_0\mathrm{MC}}}{C0MC}\xspace}
\newcommand*\sff{\mathit{II}}
\newcommand*\mcv{\vec{H}}
\newcommand*\mcs{H}
\newcommand*\shapeop{S}
\newcommand*\lorvf[1]{\Lambda_{(#1)}}
\newcommand*\strtensor{\mathcal{Q}}
\newcommand*\massterm{\mathfrak{M}}
\newcommand*\spaceslice[1]{\Sigma_{#1}} 
\newcommand*\bulkregion[2]{\mathcal{D}_{#1}^{#2}}
\newcommand*\tweight{\mathfrak{T}}
\newcommand*\areasph{\D{\text{Area}}_{\Sphere^d}}
\newcommand*\energy{\mathcal{E}^2}
\newcommand*\allrot{\mathfrak{R}}

\theoremstyle{plain}
\newtheorem{thm}{Theorem}[section]
\newtheorem{prop}[thm]{Proposition}
\newtheorem{lem}[thm]{Lemma}
\newtheorem{cor}[thm]{Corollary}
\theoremstyle{definition}

\newtheorem{quest}{Question}
\theoremstyle{remark}
\newtheorem{rmk}[thm]{Remark}
\newtheorem{exa}[thm]{Example}

\numberwithin{equation}{section}

\SVN $Date: 2014-10-13 13:36:23 +0200 (Mon, 13 Oct 2014) $
\SVN $Rev: 506 $

\begin{document}
\title[Stability and instability for the Lorentzian \cpmc]{Stability
and instability of expanding solutions to the Lorentzian
constant-positive-mean-curvature flow}
\author[W. W.-Y. Wong]{Willie Wai-Yeung Wong}
\address{\'{E}cole Polytechnique F\'{e}d\'{e}rale de Lausanne,
Switzerland}
\thanks{Version \texttt{rev\SVNRev} of \texttt{\SVNRawDate}}
\email{willie.wong@epfl.ch}
\subjclass[2010]{Primary: 53C44; Secondary: 35A02, 35B35, 35B40, 35L45, 35L72, 53A10, 53C50, 58J45}

\begin{abstract}
We study constant mean curvature \emph{Lorentzian} hypersurfaces of
$\Real^{1,d+1}$ from the point of view of its Cauchy problem. We
completely classify the spherically symmetric solutions, which include
among them a manifold isometric to the de Sitter space of general
relativity. We show that the spherically symmetric solutions exhibit
one of three (future) asymptotic behaviours: (i) finite time collapse (ii)
convergence to a time-like cylinder isometric to some
$\Real\times\Sphere^d$ and (iii) infinite expansion to the future
converging asymptotically to a time translation of the de Sitter solution. 
For class (iii) we examine the future stability properties of the
solutions under \emph{arbitrary} (not necessarily spherically
symmetric) perturbations. We show that the usual notions of asymptotic
stability and modulational stability cannot apply, and connect this to
the presence of \emph{cosmological horizons} in these class (iii)
solutions. We can nevertheless show the global existence and future 
stability for small perturbations of class
(iii) solutions under a notion of stability that naturally takes into
account the presence of cosmological horizons. The proof is based on
the vector field method, but requires additional geometric insight. In
particular we introduce two new tools: an \emph{inverse-Gauss-map}
gauge to deal with the problem of cosmological horizon and a quasilinear 
generalisation of Brendle's Bel-Robinson tensor to obtain natural
energy quantities. 
\end{abstract}

\maketitle
{\small
\tableofcontents
}

\section{Introduction}
In this paper we study Lorentzian (i.e.\ time-like) hypersurfaces 
$M\subset \Real^{1,d+1}$ of $d+2$ dimensional Minkowski spaces 
with constant, \emph{positive} mean curvature (``$M$ is \cpmc''). The limiting
case where $M$ has everywhere \emph{vanishing} mean curvature (``$M$
is \vmc'') is
actively studied under names such as \emph{relativistic membranes} and
\emph{extremal} or \emph{time-like minimal/maximal} hypersurfaces.
Mathematically they give rise to natural classes of \emph{quasilinear
wave equations} with clear geometric interpretation, and serve as a
testing ground for development of techniques in geometric analysis and
in the study of nonlinear waves on curved backgrounds; some recent
successes can be found in
\cite{DoKrSW2013,NguTia2013,Lindbl2004,Brendl2002}. On the other hand,
manifolds which are \vmc give one plausible description of a
classical (as opposed to quantum), relativistic, \emph{extended} test
object moving freely in space \cite{AurChr1979}. Understanding such 
objects seems to be
a first step toward the quantization of extended relativistic objects
(see \cite{Hoppe2013} for a recent topical review of the physical 
perspective). 

If \vmc manifolds are ``freely evolving'', then \cpmc manifolds are
those subject to a ``constant normal force''. The analogy is clearest
when we start with dimension $d = 0$. The ambient space-time is then
a $2$-dimensional Lorentzian manifold, and our manifold $M$ is simply a
curve. By assumption $M$ is taken to be time-like, and so we
interpret it as the world-line of a particle. 
Taking an arc-length (i.e.\ proper time) parametrisation, the
mean curvature of $M$ is nothing more than the acceleration of this
particle! Hence in the $d= 0$ case, the \vmc manifolds are geodesics,
and the \cpmc manifolds are those subject to a constant force, once we
appeal to Newton's second law. (See also the discussion in
\cite{AurChr1979}.)

(It is interesting to note that one can alternatively characterise
geodesics in a pseudo-Riemannian manifold as the image of a harmonic
map from $\Real$. Swapping the source space to a higher-dimensional
manifold gives another possible interpretation of what it means to
describe a freely evolving, classical, relativistic, extended test object.)

Just as the equations describing a \emph{Riemannian} hypersurface of
prescribed mean curvature have an \emph{elliptic} nature, the
equations describing our \emph{Lorentzian} hypersurfaces are
hyperbolic partial differential equations, with a locally well-posed 
initial value problem. The easiest way to see this is to fix a point 
$x\in M$ and consider $M$, locally in a neighbourhood of $x$, as a
graph over the tangent plane $\Pi_x$ to $M$. Letting $\phi$ be the
height of the graph (in the direction of the Minkowski normal
direction to $\Pi_x$), the mean curvature (see Appendix
\ref{app:defmc} for a
quick review) of $M$ is given by
\begin{equation}\label{eq:graphmc}
\text{mean curvature} = \frac{\partial}{\partial y_i} \left(
\frac{m^{ij} \frac{\partial}{\partial y_j} \phi}{\sqrt{ 1 +
m^{ij}\partial_i\phi \partial_j\phi}} \right)
\end{equation}
where $\{y_0, \dots, y_d\}$ is a flat (Minkowski) coordinate system
for the hyperplane $\Pi_x$ and $m^{ij}$ is the induced Minkowski
metric with signature $(- + \cdots +)$. That $M$ remains time-like is
captured in the condition $1 + m^{ij}\partial_i\phi \partial_j \phi >
0$. Cast in this form it is evident that \cpmc and \vmc manifolds can
be locally described by quasilinear wave equations, which classically
admit well-posed initial value problems
\cite{CouHil1962,HuKaMa1976,Klaine1980}.
Taking advantage of the finite speed of propagation for such equations, 
these local descriptions can be glued together (a technique common in
geometric wave equations and mathematical relativity, see e.g.\ 
\cite{Four1952,Ringst2009,KlaMac1995}) to get the desired local
existence of evolution. 

\begin{rmk}\label{rmk:defcauchyproblem}
More precisely, the Cauchy problem of the constant mean curvature flow
can be phrased as following.  Let $\Sigma$ be a $d$--dimensional smooth 
manifold, and $H$ the value of the prescribed mean curvature. Our 
initial data is $\Upsilon_0:\Sigma \to \Real^{1,d+1}$ a (sufficiently
regular) embedding such that $\Upsilon_0(\Sigma)$ is a space-like 
submanifold, together with $\Upsilon_1:\Sigma \to \Real^{1,d+1}$ a
family of future-directed time-like vectors. A solution to the Cauchy
problem is an embedding $\Upsilon:\Real\times\Sigma \to \Real^{1,d+1}$
satisfying $\Upsilon(\Real\times\Sigma)$ has the constant mean
curvature $H$, such that $\Upsilon(0,\pholder) = \Upsilon_0(\pholder)$
and that the image of $\D*{\Upsilon}(0,\pholder)$ is spanned by the
image of $\D*{\Upsilon_0}(\pholder)$ and $\Upsilon_1(\pholder)$. Note
that phrased in this way there is considerable gauge freedom in the
diffeomorphism $\Upsilon$ due to diffeomorphism invariance. To get a 
well-posed problem one would need to fix a gauge or coordinate system.
When $\Upsilon_0$ takes value in $\{0\}\times\Real^{d+1}\subset
\Real^{1,d+1}$ a convenient gauge is to require that
$\Upsilon(t,\pholder) \in \{t\}\times\Real^{d+1}$ and that
$\partial_t\Upsilon(t,\pholder)$ be orthogonal to
$\Upsilon(t,\Sigma)$. It is relatively simple to convert between a
solution described in this gauge with the local solution defined by
solving \eqref{eq:graphmc}. 

For obtaining global estimates in the case where $\Sigma$ is a
$d$--dimensional (topological) sphere, and the initial data $\Upsilon_{0,1}$ 
are ``sufficiently small'', it turns out a more convenient gauge
choice is what we will call the \emph{inverse-Gauss-map} gauge, and
which we will discuss in Section \ref{sec:IGMgauge}.
\end{rmk}

When facing an evolution equation with well-posed local dynamics, it
is natural to ask ``for which classes of initial data do we have
\emph{global} existence of solutions?'' When furthermore certain
\emph{explicit} solutions are known, it is also natural to ask ``are
the behaviours exhibited by those explicit solutions \emph{stable}?''
These two questions drive the analysis of the current paper. 

\subsection{Some known results in the \vmc case}
To give examples of the type of answers that one looks for in regards
to the two questions above, let us briefly review the recent progress
concerning the case of \vmc manifolds. 

The first results concerning global stability are that for the
``trivial solution'' of the \vmc equations. One easily sees that the
Minkowski space $\Real^{1,d}$ embeds in $\Real^{1,d+1}$ as a
hyperplane, and this embedding is totally geodesic, and hence has
vanishing mean curvature. Brendle (\cite{Brendl2002} for $d \geq 3$)
and Lindblad (\cite{Lindbl2004} for $d = 2$) were able to show that
starting with initial data ``sufficiently close'' (in a Sobolev sense)
to one of these time-like hyperplanes, the solution to the \vmc
equations exist for all time and converges asymptotically in time back
to said hyperplane. 

As the solution is a perturbation of a hyperplane, the manifold $M$ in
this case can be \emph{globally} represented as a graph. The results
and Brendle and Lindblad can thus be understood, via
\eqref{eq:graphmc}, as a statement about global well-posedness and
scattering for a quasilinear wave equation on $\Real^{1,d}$, and in
fact can be deduced from earlier works of Christodoulou
\cite{Christ1986} and Klainerman \cite{Klaine1986}. The decay
that drives the asymptotic convergence then takes its origins in the
\emph{linear} decay of waves on Minkowski space with $d \geq 2$, and
the crucial observation that allows the nonlinearity to be controlled
by the linear decay is that \eqref{eq:graphmc} obeys both the
quadratic \cite{Christ1986, Klaine1986} and, in Lindblad's case, the cubic
\cite{Alinha2001, Alinha2001a} \emph{null conditions}. 

There are, of course, other known explicit global solutions to the
\vmc equations. In fact, if one starts with any complete \emph{minimal
hypersurface} in $\Real^{d+1}$, extending it trivially in the time
direction leads to a geodesically complete \vmc manifold $M$. One can then ask whether the
same stability property enjoyed by the hyperplane shown by Brendle and
Lindblad (global existence for perturbed initial data, asymptotic
decay of the perturbation) is also shared by such $M$. Exactly this
question was studied recently by the author, together with R.~Donninger, 
J.~Krieger, and J.~Szeftel, for $M$ being the stationary solution 
generated by the \emph{catenoid}, with $d = 2$ \cite{DoKrSW2013}. The
catenoid is \emph{variationally} unstable as a minimal surface
\cite{FisSch1980}, a fact leading directly to linear instability of
the stationary catenoid solution under the \vmc flow. Nevertheless, in
\cite{DoKrSW2013} the authors were able to construct a centre manifold
for the evolution: under some symmetry assumptions (which in
particular allows the authors to avoid some difficulty having to do
with the trapping of null geodesics) they were able to show the
existence of a co-dimension 1 set of small perturbations which evolve
into solutions that converge asymptotically back to the catenoid. The
main decay mechanism here is, again, the dispersive decay of solutions
to the linear wave equation (on a now curved background, and with a 
short-range potential); here they crucially exploited the catenoid's
nature as an asymptotically flat manifold.

On the other hand certain blow-up results are available. It is
expected that for $M$ arising from initial data that is a
\emph{compact} manifold, one should have finite time singularity
formation under the \vmc equations. This is motivated in part by the
non-existence of compact minimal hypersurfaces in $\Real^{d+1}$, which 
implies there are no stationary solutions to the \vmc equations with
compact spatial cross-section. The singularity formation can also be 
easily verified in the
spherically symmetric case\footnote{Even in the case of higher
co-dimensions, and with external forces; see \cite{AurChr1979a}.}. 
Here the manifold can be described as the
set $\{r = f(t)\}$ where $f$ solves the nonlinear ordinary differential 
equation (see also Section \ref{sec:rotsym} below)
\begin{equation}\label{eq:vmcrotsym} 
0 = f f'' + d[ 1 - (f')^2]. 
\end{equation}
That the manifold is time-like requires $|f'| < 1$, and by assumption
$f > 0$ (it is the value of the radial coordinate). From convexity one 
can easily see the finite time collapse of any
initial data. (For $d = 1, 2$ the equation can be explicitly solved in
terms of trigonometric and Jacobi elliptic functions respectively.)
Outside of spherical symmetry, the recent work of Nguyen and Tian
\cite{NguTia2013} verified singularity formation in dimension $d = 1$
for initial data being a closed curve, and provided detailed
information about the behaviour of the solution at the singular point. 

\subsection{Positive mean curvature}
An immediate difference one notices when studying the \cpmc case is
that there exist global-in-time solutions with compact spatial cross
sections. In fact, as the sphere $\Sphere^d\subset \Real^{d+1}$ is a
constant positive mean curvature hypersurface, its trivial extension
in time gives a stationary \cpmc manifold; physically one may think of
this as a soap bubble supported by a pressure differential. As we will
discuss in Section \ref{sec:rotsym} below in the context of spherical
symmetry and time-symmetric initial data, for a \emph{fixed} value of 
the mean curvature, this static
solution forms a barrier between solutions which collapses in finite
time (both in the future and in the past) and solutions which expand
indefinitely. This immediately implies the instability of this
stationary solution (which is isometric to the Einstein cylinder)
under small perturbations, which then leads to an interesting
open question in the direction of \cite{DoKrSW2013}:
\begin{quest}\label{q:hard}
Does there exist some non-trivial set of initial perturbations of the
data generating $\Real\times \Sphere^d$ on which the \cpmc flow (with
mean curvature $d$) is \emph{orbitally} stable?
\end{quest}
A few remarks are in order. Firstly, the question is stated in terms
of orbital stability instead of asymptotic stability as the latter
would essentially require proving certain small data solutions to a
quasilinear wave equation on the Einstein cylinder decay in time. This
seems highly unlikely to the author as even for the \emph{linear} wave
equation on the Einstein cylinder one has no dispersive decay (there
are finite energy mode solutions whose amplitudes are constant in
time). Secondly, once we allow ourselves to consider solutions which
remain bounded asymptotically, there are obvious initial
perturbations, which correspond to the translation symmetries of
$\Real^{d+1}$, leading to orbital stability; hence the requirement
that the initial perturbation is non-trivial. 

We will not address Question \ref{q:hard} in this paper beyond the
spherically symmetric case; see Theorem \ref{thm:unstablecylinder}. 
Instead, the
main focus is the following, slightly easier problem.
\begin{quest}\label{q:easy}
Are the spherically symmetric expanding solutions ``outside'' the
``Einstein cylinder'' \emph{stable} under the \cpmc flow in any sense? 
\end{quest}
That this question may be more tractable comes from the expansion of
the background solution. That the expansion of space-time can drive
the decay of solutions to wave equations, even when the spatial
topology is compact, is a well-studied phenomenon
from the study of space-times with \emph{positive cosmological
constant} in general relativity. In some cases the decay given by this
expansion can be seen as stronger and giving rise to better 
estimates, compared to the dispersion on a flat space-time. For the
linear wave equation, for example, the accelerated expansion of the
space-time leads to exponential (in proper time) decay of solutions to
a constant (see, e.g.\ \cite{MeSaVa2014} and references therein);
dispersion on a flat space-time only gives polynomial decay. For a
nonlinear example one may consider Friedrich's proof of the stability 
of \emph{de Sitter space}\footnote{The original formulation of
Question \ref{q:easy}, as posed to the author by Lars Andersson, is
precisely whether de Sitter space is stable under \cpmc flow. As we
will discuss in Appendix \ref{app:gaussmap}, de Sitter space has a canonical
representation as a \cpmc manifold.} \cite{Friedr1986} compared to the
Christodoulou-Klainerman theorem on stability of \emph{Minkowski
space} \cite{ChrKla1993}. 

It is however easy to see that the answer to Question \ref{q:easy}
must be in the \emph{negative} if one studies the perturbed solution
$M$ as a graph in the normal bundle over the spherically symmetric expanding background. 
A first class of unstable perturbations are easily
understood: again we make use of the symmetries of the ambient
space-time. Isometries of $\Real^{1,d+1}$ send \cpmc manifolds to
other \cpmc manifolds; the spatial and temporal translations in
particular preserves none of the spherically symmetric expanding
solutions. As we shall see in Section \ref{sec:linearinstability}, the
corresponding perturbations grow \emph{exponentially} in proper time.
A second class of perturbations correspond to the purely radial
perturbations. From the analysis of the corresponding ODE system in
Section \ref{sec:rotsym}, we will also see that these give rise to
also exponentially growing perturbations. 

\begin{rmk}\label{rmk:notionsofstability}
It turns out that the rate of growth depends on how one measures time
and how one measures the deviation of a solution from the
spherically symmetric expanding background. For the former it is
convenient to measure with respect to a time function that is
``proper'' (by which we mean has unit length) relative to the induced
Lorentzian geometry of the spherically symmetric expanding background.
The latter is more complicated. If one treats the solution
as a graph in the normal bundle of the background, then we indeed have
linear instability leading to exponential growth (and in particular
neither asymptotic nor orbital stability holds). In terms of the
formulation given in Section \ref{sec:rotsym} where the solutions are
described as graphs $r = f(t)$, the \emph{orbital} stability of
solutions is an easy corollary of the analyses leading up to Theorem
\ref{thm:expansionstability}. In Section
\ref{sec:geometricimplications} we will also see how one can interpret
the spherically symmetric expanding solutions as orbitally stable
under general perturbations; this, however, will be a direct
consequence of the our more detailed control on the asymptotic
behaviour of solutions.  
\end{rmk}

In order to deal with these unstable perturbations, a commonly used
technique is that of \emph{modulation theory}, originally introduced
for proving \emph{orbital} (instead of asymptotic) stability of
certain stationary solutions of semilinear equations
\cite{Weinst1985,Weinst1986}. A key feature to this theory is to
identify a finite dimensional subspace (the \emph{modulation space})
of the solution space which captures the instability (in the
asymptotic sense) of the (linearised) evolution. The partial
differential equation then is decomposed as a coupled system of
\emph{ordinary} differential equations (the \emph{modulation
equations}) describing the trajectory (of the projection) on the
modulation space along with a partial differential equation describing
the dynamics transverse to the modulation space. The choice of the
modulation space and modulation equations are so that the remaining
PDE enjoys better stability or compactness properties, rendering the
problem more tractable. In many (semilinear) cases the modulation 
equations can be tracked ``in the large'', leading to results on orbital 
stability or stable blow-up dynamics (e.g.\ \cite{MeRaSz2010,
RapRod2012}). For quasilinear equations, the dependence of the
linearised operator on the background solution makes the procedure
more delicate; but if one restricts attention to showing the existence
of a centre manifold for the evolution, the basic method of Lyapunov and 
Perron can be viewed as a ``baby'' version of modulation theory, from
which some success can be obtained (for example \cite{DoKrSW2013}).

If one were to try to adapt the idea of modulation theory (or at the
very least, the Lyapunov-Perron method) na\"\i{}vely
to the \cpmc setting to study the stability of spherically symmetric
expanding solutions, one runs into an obstacle tied to the background
geometry. As is well-known in the literature in mathematical
relativity, a feature of expanding solutions such as the de Sitter
geometry or the Friedmann-Lema\^\i{}tre-Robertson-Walker (FLRW)
geometry is the presence of \emph{cosmological horizons}. Roughly
speaking, from the \emph{intrinsic} point of view the space-time may
be expanding faster than the speed of light, leading to regions which
asymptotically \emph{cannot} communicate with each other. (This will
be explained in more detail in Section \ref{sec:cosmohorizon}.) The net effect
of these cosmological horizons is that, asymptotically, one needs to
keep track of an \emph{infinite} dimensional modulation space, which
effectively obviates the advantages usually proffered by the use of
modulation theory. 

Our resolution of this conundrum is through a \emph{geometrically
motivated replacement} of modulation theory, which in practice is
implemented through a good \emph{gauge choice} (see Section
\ref{sec:IGMgauge}). The rough idea is the following: an analysis of
the spherically symmetric expanding solutions shows that they all
share the same asymptotic profile. This suggests that at the
derivative level the perturbations should ``converge to
zero''.\footnote{One notes here that this is commensurate with the
analysis of linear waves on such expanding backgrounds; the improved
decay estimates are only expected to hold for derivatives. The
function itself can converge asymptotically to a constant: unlike the
case with non-compact spatial slices, there are no obvious ways to
rule out the constant solution.} One
should then try to formulate the equation ``at the level of the first
derivative''. (Note that the perturbation equations for the solution
described as a graph over the perturbed background have a scalar 
dependence on the
solution itself, so formulating the equation for the derivatives is
not as simple as just commuting the equation with a differential
operator.) 

An imperfect analogy can be drawn with the various proofs
of the stability of Minkowski space. In harmonic coordinates,
the vacuum Einstein equations can be written as a quasilinear wave
equation for the components of the metric itself. This equation
however does not satisfy the classical null condition and it is not
until the recent work of Lindblad and Rodnianski \cite{LinRod2010,
LinRod2005} that the global behaviour of small-data solutions is
understood in terms of the so-called \emph{weak null condition}. 
Furthermore, asymptotically there is a certain loss of control for 
solutions to equations satisfying the weak null condition compared to
those to equations satisfying the classical null condition
\cite{Alinha2003, Lindbl2008}, a phenomenon having to do with the fact
that the background Minkowskian geometry is a poor approximation of
the true dynamical null geometry of the solution. Morally speaking this corresponds to
the approach studying the \cpmc problem as a quasilinear equation for
the height function of a graph over a background \cpmc manifold. Our
approach, then, is more similar to the proof of Christodoulou and
Klainerman \cite{ChrKla1993}. There the authors studied an system of 
associated equations at the level of the second derivatives of the 
unknown metric (the Weyl curvature), with one family of equations
(the Bianchi identities) arising from an integrability constraint
(morally that the curvature is the ``derivative of something else''),
and another (the dual Bianchi identities) a consequence of the 
original vacuum Einstein equations (note that the vacuum Einstein
equations is ``lower order'' than the dual Bianchi identities). One
exploits the dispersive nature of this system of equations (obtained
by considering the true dynamical null geometry of the solution) to gain
decay estimates, which one can then integrate (null structure
equations) to obtain control on the first derivatives of the metric
(Ricci rotation coefficients). As will be discussed in Section
\ref{sec:IGMgauge}, we will study for the \cpmc system also an
associated system of equations at ``higher order'' than the statement
of constant mean curvature, consisting of an integrability constraint
and an equation derived as a consequence of constant mean curvature.
This will allow us to directly prove the decay on the level of
derivatives without worrying about the possible exponential growth of
the height function itself. 

At this point we should mention that similar results (exponential
growth of the unknown together with decay of derivatives) have also
been obtained recently in the context of nonlinear stability of spatially
homogeneous solutions to coupled systems of Einstein's equation with
positive cosmological constant with various matter fields
\cite{Ringst2008,RodSpe2013,Speck2012,Speck2013,HadSpe2013}. The
positive cosmological constant drives an accelerated expansion and
leads naturally to discussions similar to Question \ref{q:easy}. A typical
feature of the results mentioned here is that the control obtained for
the fundamental unknown, which let us call $u$, takes the form 
\[ \norm[\infty]{e^{-t} u(t,\cdot) - u_0(\cdot)} \lesssim e^{-t} \]
while
\[ \norm[\infty]{u_0} \approx \epsilon \]
where $\epsilon$ is the size of the initial perturbation, while for
higher derivatives of $u$ one gets improved decay. In particular, for
the unknown $u$ itself, one cannot prove that it decays to zero, even
after renormalisation; one can only expect (renormalisable) exponential 
growth. This \emph{freezing-in} of the initial perturbation seems to
be a stable feature of stability problems for background with
accelerated expansion. Compare to this our geometrical approach
provides a small gain: we are in fact able to extract quite precisely 
the asymptotic behaviour of our perturbed solutions (see next
section). 

We remark here also that the methods employed
in \cite{Ringst2008,RodSpe2013,Speck2012,Speck2013,HadSpe2013} study
directly the equations at the level of the metric (similar to
\cite{LinRod2005} and comparable to the case of studying the height
function of the graphical description in our problem), and requires
carefully keeping track the structure of the equation to verify that
the exponential growth of the unknown itself will not cause problems.
In comparison our geometric approach allows us to be much more
schematic when considering the structure of the equations --- this is
attested in the relative simplicity of the proof of Theorem
\ref{thm:mainSmallData} below. Unfortunately it is not clear to the
author whether a similar approach is available to treat the problems
in general relativity.

\subsection{Main results and outline of paper}
We start with some remarks. First, we will use throughout the
\emph{Japanese bracket} notation  $\jb{x} \eqdef \sqrt{1 + x^2}$ for
$x\in \Real$. Secondly, we give a quick review of pseudo-Riemannian
geometry in the Appendix, which includes setting of the convention for
the definition of the mean curvature (in our convention the unit
sphere $\Sphere^d\subset\Real^{d+1}$ has positive mean curvature $d$).
Thirdly, examining the behaviour of the mean curvature
under scaling transformations (see Appendix \ref{app:defmc} and
\eqref{eq:sffscaling}), we see that when studying the \cpmc problem,
we can assume without loss of generality that the mean curvature
scalar is fixed to be $(d+1)$. 

In Section \ref{sec:rotsym}, we study the \cpmc problem in spherical
symmetry. The equations of motion reduce to a single second order
ordinary differential equation, and we completely classify its
asymptotic behaviour (including the blow-up cases), first qualitatively 
in Section \ref{sec:classification} and then quantitatively in Section
\ref{sec:rotsymasymp}. As we have already seen above, \emph{symmetries
of the ambient space}\footnote{In the context of spherical symmetry,
the only nontrivial compatible symmetry in the Poincar\'e group is
time translation.} can generate instabilities for the associated
equations of motion; a fact we will recover from our analysis.
However, our asymptotic profile also implies that this is the
\emph{only} instability in the spherical symmetric case, and we have
indeed a \emph{modulational} stability result. To illustrate, 
we give here a rough version of Theorem
\ref{thm:expansionstability}.
\begin{thm}
Let $M = \set{(t,x)\in\Real_+\times \Real^{d+1}}{\abs{x} = f(t)}$
denote a spherically
symmetric \cpmc manifold, such that the defining function
$f(t):\Real_+\to\Real_+$ satisfies $\lim_{t\to\infty} f(t) = \infty$.
Then $M$ as an individual solution is unstable under small
perturbations. However, the family of all time translations of $M$ is
future asymptotically stable.
\end{thm}

The natural question to ask after the previous theorem is whether it
extends to the case without spherical symmetry. In Section
\ref{sec:cosmohorizon} we show that the answer is \emph{no}, by
exploiting the finite speed of propagation properties of hyperbolic
partial differential equations, and the presence of so-called
``cosmological horizons'' on expanding space-times such as de Sitter
space $\dS$ (see Appendix \ref{app:gaussmap} for the definition). The
main results of this section are Theorems \ref{thm:allmodeunstable}
and \ref{thm:badasymp}. The first theorem applies to the linearised
equation around $\dS$, where the solution is treated as a graph over
the normal bundle of $\dS$: it indicates that the linearised equation
has an infinite dimensional set of unstable directions, making naive
applications of modulation theory unsuitable. The second theorem shows
that the (finite dimensional) family generated by the application of
the Poincar\'e group to $\dS$ cannot exhaust all possible asymptotic
structures, in stark contrast to the spherically symmetric case. 

The remaining sections are devoted to proving that, in spite of the
results obtained in Section \ref{sec:cosmohorizon}, one can still have
a positive answer to Question \ref{q:easy} if one refines the notion
of ``stability''. (We remark here that while the Sections
\ref{sec:rotsym} and \ref{sec:cosmohorizon} have some independent
interest and lays the motivation and intuition for the rest of the
paper, the material presented in the
remaining sections are essentially logically independent.)
Returning to the issue of cosmological horizons, we
see that it forces an asymptotic decoupling of disjoint spatial
regions of the solution. Thus one should expect that, in order to
apply some sort of modulation theory, the modulation parameter should
no longer be just a running function of time. Instead, it should be
given by a function defined over the entire space-time: this nicely 
dovetails with the intuition
that the modulation space is infinite dimensional. The actual
implementation of this idea, however, is geometrical: we find a
mapping from our perturbed manifold to the standard $\dS$ such that
certain geometric quantities (including the difference of the induced
metrics and the difference of induced second fundamental forms and
their derivatives) decay asymptotically. We may interpret our final
result (Theorem \ref{thm:bigmaintheorem}) as
\begin{thm}\label{thm:maintheoremgeometric}
Let $M$ be the (future) \cpmc manifold generated by a small
perturbation of the initial data for a spherically symmetric, future
expanding solution described in Section \ref{sec:rotsym} (which
includes, in particular, the $\dS$ solution). Then as long
as the initial perturbation is sufficiently small, we have that
\begin{itemize}
\item $M$ is future global;
\item $M$ converges in time, \emph{spatially locally}, to a (space-time)
translation of the original expanding solution. By spatially locally
one should think a notion such as ``along tubular neighbourhoods `of
fixed spatial size $\epsilon\ll 1$' of time-like curves''. 
\end{itemize}
\end{thm}

In order to obtain the above results, we introduce two
new\footnote{Both tools have appeared before in the mathematical
literature in the large. But
their use in this context is new.} tools, which, to the specialists,
would be the main contribution of this paper. The first, as already
mentioned, is the \emph{inverse-Gauss-map gauge}, our geometric
replacement for modulation theory. This is developed in Section
\ref{sec:IGMgauge}. Under the inverse-Gauss-map gauge, the equations
of motions reduce to a relatively simple form \eqref{eq:mainsystem}
which is a quasilinear divergence-curl system. To establish the
suitable \emph{a priori} energy estimates for demonstrating
\emph{decay}, we first refine Brendle's Bel-Robinson tensor
\cite{Brendl2002} in Section \ref{sec:strtensor} to a very general
setting in order to apply to our
quasilinear situation. This allows us to prove $L^2$-based energy
estimates in Section \ref{sec:energyest}; these estimates are somewhat
unintuitively \emph{weighted in time} (the unweighted $L^2$ norms are allowed to
grow exponentially in proper time).  The favourable geometry 
of $\dS$ allows us to dwarf this growth by the exponential growth of the
spatial volume, which, via a Sobolev embedding, gives that the
$L^\infty$ norm will in fact decay exponentially (in proper time), assuming 
boundedness
of the weighted $L^2$ energy. Small data global existence and
asymptotic stability then follows by a standard bootstrap argument. 

In writing up this paper, concision is sacrificed for motivation and
for a desire for the manuscript to be reasonably self-contained. The
author wishes the readers grant him this indulgence. 

\subsection{Acknowledgements}
This paper has its genesis in a question posed to the author by Lars
Andersson at the 2014 OXPDE workshop in \emph{Nonlinear Wave Equations
and General Relativity}; as such the author must thank Lars for the
interesting question, and also OXPDE, especially Gui-Qiang Chen and
Qian Wang, for their hospitality. The research for this paper can thus
be said to have its seeds sown at Oxford; the author is otherwise
supported by the Swiss National Science Foundation through a grant to
Joachim Krieger, who has on many occasions also lent the author his
sympathetic ear. The author would also like to thank Joules Nahas for
several profitable discussions; Jared Speck for clarifying
some details of his work on the stability of expanding FLRW solutions;
and Demetrios Christodoulou for some insightful comments on the
manuscript.

\tocaddline
\section{Rotationally symmetric solutions}\label{sec:rotsym}
Under rotational symmetry\footnote{In the slightly different setting
of axially symmetric co-dimensional 2 surfaces in $\Real^{1,3}$, 
which correspond
physically to the case of circular \emph{strings}, evolving under a 
constant electromagnetic field, similar results have been obtained by
Aurilia and Christodoulou \cite{AurChr1979a}. The presence of the
external field generates some different dynamics, depending on whether
the field is electric or magnetic.}, the equation for constant mean 
curvature reduces to an ordinary
differential equation in the extrinsic time variable $t$: let $r$ be the
radial coordinate, the inward unit normal to the rotationally symmetric
surface given by the graph of $r = f(t)$ is 
\[ \vec{n} = -\frac{1}{\sqrt{1 - (f')^2}}(\partial_r + (f') \partial_t).\]
A direct computation yields that the nonlinear ODE for constant mean
curvature $c$ is 
\begin{equation}\label{eq:cmcsphsym}
[1-(f')^2] d + f f'' = c [1-(f')^2]^{\frac32} f,
\end{equation}
as indicated before, by rescaling we can fix $c = d+ 1$ for
convenience. We can equivalently write \eqref{eq:cmcsphsym}, with the
choice of $c$ fixed, as
\begin{equation}\label{eq:cmcsphsym2}
\left( \frac{f'}{\sqrt{1 - (f')^2}}\right)' = (d+1) -
\frac{d}{f\sqrt{1-(f')^2}}.
\end{equation}

The equation \eqref{eq:cmcsphsym} admits two explicit solutions.
The pseudo-sphere $\dS = \Sphere^{1,d+1,1}$ as described in Appendix
\ref{app:gaussmap} corresponds to the solution $f(t) = \jb{t}$. 
Another explicit solution is given by the cylinder 
$f(t) \equiv \frac{d}{d+1}$. Note that as \eqref{eq:cmcsphsym} is
autonomous, time translations of solutions are also solutions. 

In this
section we will analyse the ODE \eqref{eq:cmcsphsym} and describe the
asymptotic behaviours of the solutions. Observe that from the
fundamental theorem of existence and uniqueness of ordinary
differential equations, if $f(t_0) \neq 0$ and $\abs{f'(t_0)} < 1$,
the equation \eqref{eq:cmcsphsym} has an unique local solution also
satisfying $f \neq 0$ and $\abs{f'} < 1$. These two conditions are
\emph{geometric} in nature: when $f = 0$ the solution manifold $\{r
= f(t)\}$
collapses to a point and fails to be regular, while when $\abs{f'} =
1$ the induced pseudo-Riemannian structure on the 
solution manifold $\{r = f(t)\}$ becomes degenerate. 
We first prove a blow-up criterion.

\begin{prop}\label{prop:odeblowupcrit}
Let $\abs{t_1},\abs{t_2} < \infty$, and let $f:(t_1,t_2)\to\Real_+$ be
a $C^2$ solution of \eqref{eq:cmcsphsym}. If $\sup_{(t_1,t_2)}\abs{\log f} 
< \infty$, and $\abs{f'(t_0)} < 1$ for some $t_0\in (t_1,t_2)$, 
then $\sup_{(t_1,t_2)} \abs{f'} < 1$. 
\end{prop}
\begin{proof}
Consider the quantity $u = 1 - (f')^2$. A direct computation from
\eqref{eq:cmcsphsym} gives
\begin{equation}\label{eq:interimuprimebound0}
(\log u)' = \frac{u'}{u} = 2\frac{f'}{f} \left[ d - (d+1) f u^\frac12\right].
\end{equation}
Observe that $u > 0 \implies \abs{f'} < 1$, and that by construction
$u \leq 1$. Thus the right-hand side of \eqref{eq:interimuprimebound0} 
is bounded whenever $u> 0$ and $\log f$ is bounded. Let 
$U$ be the connected component containing $t_0$ of the 
open subset $\set{t \in (t_1,t_2)}{u >0}$.
Integrating \eqref{eq:interimuprimebound0} from $t_0$, using the
boundedness of $t_1,t_2$, gives that $\sup_U \abs{\log u} < \infty$,
and hence $U$ is closed. Therefore $U = (t_1,t_2)$ and 
$\sup_{(t_1,t_2)} \abs{f'} < 1$.
\end{proof}

The implied bound on $f'$ in Proposition \ref{prop:odeblowupcrit} also
shows that starting from initial data $f(t_0) > 0$ and $\abs{f'(t_0)}
< 1$, the solution $f$ cannot blow-up to $\infty$ in finite time.
Hence we have the continuation criterion
\begin{cor}\label{cor:odecontcrit}
With initial data $f(t_0) \in \Real_+$ and $f'(t_0)\in (-1,1)$, the
solution can be extended as long as $f$ is bounded away from $0$. 
\end{cor}

\subsection{Classification}\label{sec:classification}

Next we make precise the notion of the cylindrical solution $f \equiv
\frac{d}{d+1}$ being a barrier between global existence and finite
time extinction. 
\begin{prop}\label{prop:globalodeex}
If $f(t_0) > \frac{d}{d+1}$, and $f'(t_0) \geq 0$ then $f$ can be
extended to a solution on the whole ray $[t_0,\infty)$ with $0 \leq f'
< 1$, and such that $f$ grows unboundedly as $t\to\infty$. Similarly, 
if $f(t_0) > \frac{d}{d+1}$ and $f'(t_0) \leq 0$ then
$f$ can be extended to a solution on the whole ray $(-\infty,t_0]$
with $-1 < f' \leq 0$, and such that $f$ grows unboundedly as $t \to
-\infty$. 
\end{prop}
\begin{proof}
By time reversal it suffices to consider the case $f'(t_0) \geq 0$. 
For the existence proof we need to show that $f$ remains bounded
below. Rearranging \eqref{eq:cmcsphsym} we get
\begin{equation}\label{eq:rearrangedsphsym} f'' = \frac{1 - (f')^2}{f} \left[ (d+1) f \sqrt{1 - (f')^2} - d
\right]\end{equation}
which implies that whenever $f(t)\sqrt{1 - (f'(t))^2} > \frac{d}{d+1}$, we
must have $f''(t) > 0$. In view of the initial conditions this implies
$f'(t) > 0$ for all $t > t_0$ when the solution exists. This further
implies that $f(t) \geq f(t_0) > \frac{d}{d+1}$ and by Corollary
\ref{cor:odecontcrit} the solution can be extended for all future
time. 

To show that the solution cannot remain bounded, we argue by
contradiction. We have shown that $f'(t) > 0$ for
all $t > t_0$. Were $f$ to remain bounded, necessarily $\lim_{t \to
\infty} f'(t) = 0$. But since we know that $f(t) \geq f(t_0) >
\frac{d}{d+1}$, for all sufficiently large $s$ this gives $f(s) \sqrt{1
- (f'(s))^2} > \frac{d}{d+1}$, and hence by
  \eqref{eq:rearrangedsphsym} again $f''(s) > 0$, which then gives a
contradiction with the assumed decay of $f'$.
\end{proof}

\begin{prop}\label{prop:odeblowupcollapse}
If $f(t_0)\sqrt{1 - (f'(t_0))^2} < \frac{d}{d+1}$, then the solution
extinguishes in finite time. More precisely, under the above
assumption
\begin{itemize}
\item if $f'(t_0) \leq 0$ then there exists $t_1 > t_0$ such that the
solution exists on $[t_0,t_1)$, and $\lim_{t\nearrow t_1} f(t) = 0$. 
\item if $f'(t_0) \geq 0$ then there exists $t_1 < t_0$ such that the
solution exists on $(t_1,t_0]$ and $\lim_{t \searrow t_1} f(t) = 0$. 
\end{itemize}
\end{prop}
\begin{proof}
For convenience write $\gamma = \frac{f'}{\sqrt{1 - (f')^2}}$ and
$\eta = f \sqrt{1 - (f')^2}$. From \eqref{eq:cmcsphsym2} we see 
\begin{equation}\label{eq:gammaprime} \eta \gamma' = (d+1) \eta - d.\end{equation}
A direct computation shows
\begin{equation}\label{eq:etaprime1} \eta' = f' \sqrt{1- (f')^2} \left( 1- \eta
\gamma'\right).\end{equation}
Thus whenever $\eta < \frac{d}{d+1}$ we have $\eta\gamma' < 0$ and
$\eta' f' \geq 0$. Hence if $f'(t_0) \leq 0$ (or $\geq 0$) we must 
have that for all $t > t_0$ (or $< t_0$) where the solution exists,
$\eta(t) \leq \eta(t_0) < \frac{d}{d+1}$. Now, we have that
\[ \gamma' = \frac{f''}{[ 1 - (f')^2]^{\frac32}}\]
and hence our control on $\eta(t)$ implies that $f''(t) < 0$ in the
relevant intervals. Hence in finite time $f$ must become zero. 
\end{proof}
\begin{rmk}
Suppose that $f(t_0') \sqrt{1 - (f'(t_0'))} = \frac{d}{d+1}$. Then in
the proof above we see that $\gamma'(t_0') = 0$, which implies that
$\eta'(t_0') = f'(t_0') \sqrt{1 - (f'(t_0'))^2}$. Hence if $f'(t_0') < 0$ (or
$< 0$), at time $t_0 = t_0' + \epsilon$ (or $- \epsilon$) for some 
$\epsilon > 0$ sufficiently small, the hypotheses of Proposition
\ref{prop:odeblowupcollapse} are satisfied, and we also have finite
time collapse. The remaining case is when
$\eta(t_0') = \frac{d}{d+1}$ and $\gamma(t_0') = 0$: this corresponds to
the static cylindrical solution $f \equiv \frac{d}{d+1}$. 
\end{rmk}
\begin{rmk}
Combining \eqref{eq:gammaprime} and \eqref{eq:etaprime1} we get
\begin{equation}\label{eq:etaprime2}
\eta' = f' \sqrt{1 - (f')^2} (d+1) (1 - \eta).
\end{equation}
The corresponding stationary solution $\eta \equiv 1$ is given by
precisely the pseudo-sphere $f(t) = \jb{t}$. 
\end{rmk}

Propositions \ref{prop:odeblowupcollapse} and 
\ref{prop:globalodeex} completely characterises solutions of
\eqref{eq:cmcsphsym} when $f' = 0$ somewhere. They fall into three
classes:
\begin{description}
\item[Expanding solutions] The derivative $f'$ vanishes at exactly one
point $t_0$, the solution exists globally, with $f(t) > \frac{d}{d+1}$
always. Furthermore $\lim_{t \to \pm\infty} f(t) = \infty$. 
\item[Static cylinder] $f \equiv \frac{d}{d+1}$, $f' \equiv 0$. 
\item[Big bang and big crunch] The solution exists on a bounded
interval $(t_1,t_2)$ with $\abs{t_1} + \abs{t_2} < \infty$. The
derivative $f'$ vanishes at exactly one point $t_0 \in (t_1,t_2)$.
$f(t) < \frac{d}{d+1}$ always, and $\lim_{t\to t_1,t_2} f(t) = 0$. 
\end{description}
From Cauchy stability the class of expanding solutions and the class
of ``big bang and big crunch'' solutions are stable under small
perturbations, in the sense that sufficiently small perturbations of a
solution in one of the two above classes will be another solution in
the same class. 

To categorise the remaining solutions for which $f'$ never vanishes,
we need the following lemma.
\begin{lem}\label{lem:asymptotics}
Let $f$ be a positive $C^2$ solution of \eqref{eq:cmcsphsym} on 
$(t_0,\infty)$ and $f'\neq 0$. Then if  $\lim_{t \to\infty}f(t) < 
\infty$ we must have $\lim_{t\to\infty} f(t) = \frac{d}{d+1}$. 
\end{lem}
\begin{proof}
If $f$ is monotonic and bounded, then $\lim_{t\to\infty} f'(t) = 0$.
Then by \eqref{eq:rearrangedsphsym} we have that for all sufficiently
large $s$, $f''(s)$ is signed and bounded away from zero if
$\lim_{t\to\infty}f(t) \neq \frac{d}{d+1}$. This gives a
contradiction with the decay of $f'$. 
\end{proof}
\begin{rmk}\label{rmk:otherclasses}
Lemma \ref{lem:asymptotics} implies that when $f'$ never vanishes, the
solution belongs to one of the six classes given by
\begin{enumerate}
\item $f' > 0$: $f$ collapses to $0$ in finite time in the past, and grows
unboundedly in the future.
\item $f' > 0$: $f$ collapses to $0$ in finite time in the past, and
asymptotically approaches $\frac{d}{d+1}$ from below. 
\item $f' > 0$: $f$ exists globally; it approaches to $\frac{d}{d+1}$
from above in the past, and it grows unboundedly in the future. 
\end{enumerate}
and their time reversals.
\end{rmk}

\begin{lem}\label{lem:otherclassesnonempty}
All six classes in Remark \ref{rmk:otherclasses} are non-empty. 
\end{lem}
\begin{proof}
That the first class in Remark \ref{rmk:otherclasses} and its time-reversal
are non-empty follows by applying Propositions \ref{prop:globalodeex}
and \ref{prop:odeblowupcollapse} to initial data with $f(t_0) =
\frac{d}{d+1}$ and $f'(t_0) \neq 0$. This further implies that the
other classes are also non-empty: we give the proof for the third
class; the proof for the remaining classes are similar and omitted. 

 Let $f_0$ be a solution that
collapses in finite time in the past, and expands indefinitely in the
future; then at some value $t_0$ we can satisfy $f_0(t_0) >
\frac{d}{d+1}$, and $f_0'(t_0) > 0$. 

Now let $f_{(\lambda)}$ be the solution given by $f_{(\lambda)}(t_0) =
f_0(t_0)$ and $f'_{(\lambda)}(t_0) = \lambda$, where $\lambda \in
(-1,1)$. Define the sets
\[ C = \set{ \lambda \in (-1,1)}{f_{(\lambda)}\text{ collapses in finite time in the
past}} \]
and
\[ E = \set{ \lambda \in (-1,1)}{f_{(\lambda)}\text{ expands indefinitely in the 
past}}, \]
neither is empty since $f_0\in C$ and $f_{(\lambda)}\in E$ for every
$\lambda \leq 0$ by Proposition \ref{prop:globalodeex}. From
Propositions \ref{prop:globalodeex} and \ref{prop:odeblowupcollapse},
together with Cauchy stability for the initial value problem, we have
that both $C$ and $E$ are open sets. As $(-1,1)$ is connected, there
must then exist a $\lambda'  > 0$ such that
$f_{(\lambda')}$ neither expands indefinitely in the past nor
collapses in finite time. Hence it must be in the third class of Remark
\ref{rmk:otherclasses}.
\end{proof}

The construction given in the proof above in fact shows that for each
$r_0$, there exists some $\lambda_0$ such that the solution corresponds
to $f(t_0) = r_0$ and $f'(t_0) = \lambda_0$ converges to
$\frac{d}{d+1}$ in the future (past). To understand better the
dependence of $\lambda_0$ on $r_0$, we observe the following maximum
principle. 

\begin{lem}\label{lem:maximumprinciplegamma}
Let $f_1$ and $f_2$ be two distinct solutions to \eqref{eq:cmcsphsym}, then
$(f_2 - f_1)^2$ has at most one critical point, and it must be a local
minimum. 
\end{lem}
\begin{proof}
Assume $t_0$ is a critical point of $(f_2 - f_1)^2$. By the
fundamental uniqueness theorem of ODEs, since $f_1$ and $f_2$ are
distinct solutions, either $f_2(t_0) = f_1(t_0)$ or $f_2'(t_0) =
f_1'(t_0)$. In the first case since $(f_2 - f_1)^2$ is non-negative,
the critical point must be a local minimum. In the second case
\eqref{eq:cmcsphsym} implies that at the point $t_0$ 
\[ (f_2 - f_1)'' = \frac{d [1 - (f_1')^2] (f_2 - f_1)}{f_1
f_2} \]
holds, which implies that
\[ [(f_2 -f_1)^2]'' = 2(f_2 - f_1)(f_2 - f_1)'' + 2 [(f_2 - f_1)']^2 >
0.\]
Thus any critical point of $(f_2 - f_1)^2$ must be a local minimum,
which rules out the possibility of more than one critical point, since
between any two local minimum there must be a local maximum. 
\end{proof}

\begin{cor}\label{cor:odeminimum}
Let $f_1$ and $f_2$ be two distinct solutions to \eqref{eq:cmcsphsym}.
\begin{enumerate}
\item $f_1$ and $f_2$ intersect at most once. If they do intersect,
then $f_2 - f_1$ is strictly monotonic. 
\item $f_1$ and $f_2$ are parallel at most once. When they are
parallel, it is when $f_2 - f_1$ is at a strict minimum. 
\end{enumerate}
\end{cor}

\begin{cor}\label{cor:odecomparisonprinciple}
\begin{enumerate}
\item For every $r_0\in\Real_+$, there exists exactly one $\lambda_0\in (-1,1)$ such that the
solution with data $f(t_0) = r_0$ and $f'(t_0) = \lambda_0$ satisfies
$\lim_{t\to \infty} f(t) = \frac{d}{d+1}$; solutions with $f(t_0) =
r_0$ and $f'(t_0) > \lambda_0$ (or $< \lambda_0$) will expand
indefinitely (or collapse in finite time) to the future. 
\item For every $\lambda_0\in (-1,1)$, there exists exactly one
$r_0\in\Real^+$
such that the solution with data $f(t_0) = r_0$ and $f'(t_0) =
\lambda_0$ satisfies $\lim_{t\to\infty} f(t) = \frac{d}{d+1}$.
Solutions with $f'(t_0) = \lambda_0$ and $f(t_0) > r_0$ (or $< r_0$)
will expand indefinitely (or collapse in finite time) to the future. 
\end{enumerate}
\end{cor}

\begin{prop}\label{prop:smoothdependencelambda}
Let $\lambda_+:\Real_+ \to (-1,1)$ be the assignment given by
Corollary \ref{cor:odecomparisonprinciple}, and $\lambda_-$ be the one
of the time-reversed version. Then $\lambda_\pm$ are smooth, strictly
monotonic functions
on $\Real_+\setminus \{d/(d+1)\}$, and continuous at $d/(d+1)$. 
\end{prop}
\begin{proof}
Let $f_{1,+}$ be a solution that collapses in finite-time in the past
and converges to $\frac{d}{d+1}$ in the future. Since $f'_{1,+} > 0$
always, we have that the function $f'_{1,+}\circ f^{-1}_{1,+}:
(0,d/(d+1)) \to (0,1)$ is a smooth function, and clearly it agrees
with $\lambda_+$. Similarly using $f_{2,+}$ the solution that expands
indefinitely in the past and converges to $\frac{d}{d+1}$ in the
future, we show that $\lambda_+$ is smooth on $(d/(d+1),\infty)$. By
their definitions it is also clear that 
\[ \lim_{r\nearrow\frac{d}{d+1}} f'_{1,+}\circ f^{-1}_{1,+}(r) = 0
= \lim_{r\searrow \frac{d}{d+1}} f'_{2,+}\circ f^{-1}_{2,+}(r) \]
establishing continuity. Monotonicity then follows from the continuity
and the fact that by Corollary \ref{cor:odecomparisonprinciple} that
$\lambda_\pm$ are invertible. 
\end{proof}

\subsection{Asymptotics}\label{sec:rotsymasymp}
For the non-static solutions, it is clear that due to the freedom of
time translation, the solutions cannot be asymptotically stable in the
direction where the solution expands or collapses. To understand their
behaviour, we examine in more detail the asymptotic behaviour of
solutions. 

\subsubsection{Convergence to $\frac{d}{d+1}$}
One can converge from above, or from below. From below, it is clear
that the quantity $f\sqrt{1-(f')^2} < \frac{d}{d+1}$ throughout, and
hence by \eqref{eq:rearrangedsphsym} we have $f'' < 0$ throughout. For
the decay of $f'$ to zero, we must have that $f''$ is integrable.
Using that $(1-(f')^2) / f > 1$ in the limit, this implies that 
\[ (d+1) f\sqrt{1 - (f')^2} - d \]
(which is strictly increasing since $\abs{f'}$ is decreasing and $f$
is increasing) must be integrable.

In the case of convergence from above, we note that if $f\sqrt{1 -
(f')^2}$ ever falls below $\frac{d}{d+1}$, then Proposition
\ref{prop:odeblowupcollapse} kicks in and we have finite time
collapse. This implies that necessarily we must have $f\sqrt{1 -
(f')^2} > \frac{d}{d+1}$ throughout. Thus $f''$ is positive throughout
and, as above, must remain integrable. Hence in this case we also have
that $\abs{f \sqrt{1 - (f')^2} - \frac{d}{d+1}}$ is integrable. 

On the other hand, since $f$ is monotonic and converges, we must also
have $\abs{f'}$ be integrable. Using that $1 - x^2 > (1-\abs{x})^2$ we have
that $1 - \sqrt{1 - (f')^2}$ is integrable, and hence
\begin{prop}
If $\lim_{t\to\infty} f(t) = \frac{d}{d+1}$ for a (semi-global)
solution, we must have that $\abs{f(t) - \frac{d}{d+_1}}$ is
integrable. Analogously for the case $t\to-\infty$. 
\end{prop}

\subsubsection{Expansion}
Assume now that $f(t)$ expands indefinitely as $t\to\infty$; the
$t\to-\infty$ case can be dealt with analogously. In the
following analysis, we assume that $t$ is sufficiently large so that
from our previous analysis $f'(t) > 0$. Recall the quantity $\eta =
f\sqrt{1-(f')^2}$. Going back to
\eqref{eq:etaprime2} we see that the stationary solution $\eta = 1$ is
attractive, in the sense that if $\eta < 1$ then $\eta' > 0$ and if $\eta
> 1$ then $\eta' < 0$. In particular, $\eta-1$ cannot change sign. 

\begin{lem}\label{lem:expansionderivativelimit}
Under our expansion assumption, $\lim_{t\to\infty} f'(t) = 1$. 
\end{lem}
\begin{proof}
From the discussion above $\eta$ is bounded and monotonic, and hence
must converge as $t\to \infty$. This requires $\eta' \to 0$. From
\eqref{eq:etaprime2} we see that this requires either $\eta \to 1$,
$f'\to 0$, or $\sqrt{1 - (f')^2} \to 0$. The middle option is
impossible in the expansion case in view of
\eqref{eq:rearrangedsphsym}. As $f$ increases unboundedly by
assumption, if $\eta\to 1$ we must have $\sqrt{1 - (f')^2} \to 0$.
Since $f' > 0$ we have that the limit must be $f' \to 1$. 
\end{proof}

\begin{lem}
Under the above assumptions, $1 - f'(t)$ is integrable.
\end{lem}
\begin{proof}
In the case $\lim \eta \neq 1$, the fact that $\eta'$ is integrable
implies that $\sqrt{1 - (f')^2}$ is integrable by
\eqref{eq:etaprime2}. As pointwise for $x\in (0,1)$ we have
$\sqrt{1-x^2} \geq 1 - x$, we have that $1-f'$ is also integrable. 

In the case $\lim \eta = 1$ (in fact this argument works as long as
$\lim \eta > \frac{d}{d+1}$), we note that asymptotically, by
\eqref{eq:gammaprime} we have $\gamma' \approx 1$. Thus for some
sufficiently large $T$ we have that, for every $t > T$
\[ \gamma(t) - \gamma(T) \geq \frac12 (t-T).\]
This implies that, using the definition $\gamma =
\frac{f'}{\sqrt{1-(f')^2}} < \frac{1}{\sqrt{1 - (f')^2}}$, that
\[ \frac{1}{\frac12 (t-T) + \gamma(T)} \geq \sqrt{1-(f')^2}.\]
So asymptotically we have that
\[ 1 - f' = \frac{1 - (f')^2}{1 + f'} \lesssim \frac{1}{t^2}\]
giving also integrability. 
\end{proof}

\begin{cor}\label{cor:expansiontolightcone}
There exists a constant $\tau_0$ such that 
\[ \lim_{t \to\infty} \abs{f(t) - (t - \tau_0)} = 0.\]
\end{cor}

In terms of the geometric picture, \emph{every expanding \cpmc
manifold is asymptotic to a light-cone}. 

\begin{rmk}\label{rmk:expansionstability}
As the pseudo-sphere $\dS$ is also an expanding solution, and
asymptotes to a light-cone, equivalently we can say that every
expanding \cpmc manifold is asymptotic to a time-translation of $\dS$.
This fact is what will drive our \emph{stability} analysis later:
one can hope that the $\dS$ gives a suitable asymptotic profile once
we factor in the Euclidean symmetries. Note also that in the case
where the solution expands both in the future and the past, the
parameter $\tau_0$ in the previously corollary can be different at the
two ends, and similarly the past and future expansions need not
be asymptotic to the same $\dS$ solution. 
\end{rmk}

\subsubsection{Collapse}
We complete the analysis by examining the asymptotic behaviour at the
collapse points $f \to 0$. This follows by examining the equation
\eqref{eq:gammaprime} for the quantity
\[ \gamma = \frac{f'}{\sqrt{1 - (f')^2}}\]
which we rewrite in integral form as
\begin{equation}\label{eq:gammaprime2}
\gamma(t_2) - \gamma(t_1) = (d+1) (t_2 - t_2) - \int_{t_1}^{t_2}
\frac{d}{\eta(s)} \D{s}.
\end{equation}
Now, let $f\to 0$ as $t\nearrow T$ (the collapse in the past can be
treated analogously). By Proposition \ref{prop:odeblowupcrit} we have
that $\abs{f'} \leq 1$ for the duration of the evolution, and hence
$\lim_{t\nearrow T} \eta(t) = 0$, where we recall that $\eta = f
\sqrt{1 - (f')^2}$. Revisiting \eqref{eq:etaprime2} tells us that
$\eta'$ remains bounded, hence we have the blow-up rate
\[ \frac{1}{\eta(t)} \gtrsim \frac{1}{T - t}.\]
This in particular implies that $\frac{1}{\eta}$ is not integrable. So
\eqref{eq:gammaprime2} implies that 
\[ \lim_{t\nearrow T} \gamma(t) = - \infty.\] 
Using again that $f'$ remains bounded on the interval of existence, we
see that this requires 
\[ \lim_{t\nearrow T} 1 - (f')^2 = 0.\]
Hence we have proven
\begin{lem}
The derivative $\abs{f'}$ converges to $1$ when $f$ collapses to $0$. 
\end{lem}
This can be strengthened a little to a rate of convergence. Revisiting
\eqref{eq:etaprime2} we see that this means $\eta(t) \lesssim
(T-t)^2$ in a small neighbourhood. This implies that $\abs{\gamma(t)}
\gtrsim \frac{1}{T-t}$, and hence 
\begin{prop}\label{prop:collapsetangentcone}
If $\lim_{t\nearrow T} f(t) = 0$, then in a small neighbourhood
$(T-\epsilon,T)$ the following estimate holds:
\[ 1 - \abs{f'(t)} \leq \sqrt{1 - (f'(t))^2} \lesssim T - t.\]
\end{prop}

\subsection{Stability and instability}
We now summarise the stability and instability properties of solutions
to \eqref{eq:cmcsphsym} in view of the analyses given above. This
answers exactly Questions \ref{q:hard} and \ref{q:easy} posed in
the introduction for the \emph{spherically symmetric} case. We will
phrase our statements in terms of \emph{future} stability, but the
time-reversed case is analogous.  

\begin{thm}\label{thm:unstablecylinder}
Let $f$ be a semi-global solution to \eqref{eq:cmcsphsym} such that
$\lim_{t\to\infty}f(t) = \frac{d}{d+1}$. Then $f$ is future unstable:
generic perturbations of $f$ will either collapse in finite time or
expand indefinitely in the future. There exists however a co-dimension
1 set of stable perturbations. 
\end{thm}
\begin{proof}
Follows immediately from Corollary \ref{cor:odecomparisonprinciple}
and Proposition \ref{prop:smoothdependencelambda}.  
\end{proof}

\begin{thm}\label{thm:expansionstability}
Let $f$ be a semi-global solution to \eqref{eq:cmcsphsym} such that
$f$ expands indefinitely in the future. Then $f$ is future
asymptotically unstable. However, writing $f_\tau(t) = f(t+ \tau)$,
the family of time-translates $\{f_\tau\}_{\tau\in\Real}$ is future
asymptotically stable, in the sense that for every initial data
sufficiently close to that of $f$, one can find $\tau_0$ such that the
perturbed solution converges to $f_{\tau_0}$ as $t\nearrow \infty$. 
\end{thm}
\begin{proof}
That for generic perturbations $f$ is future asymptotically unstable
follows from the time-translation symmetry of \eqref{eq:cmcsphsym}.
The stability of the family $\{f_\tau\}$ follows from 
Corollary \ref{cor:expansiontolightcone} and Remark
\ref{rmk:expansionstability}.
\end{proof}

\begin{thm}\label{thm:collapsestability}
Let $f$ be a solution to \eqref{eq:cmcsphsym} that collapses in finite
time in the future. Then $f$ is future unstable, in the sense that a
generic perturbation of $f$ collapses at a different finite time in
the future. However, the family of time translations
$\{f_\tau\}_{\tau\in\Real}$ as defined in the previous theorem is
stable, in the sense that for every initial data sufficiently close to
that of $f$, one can find $\tau_0$ such that the perturbed solution
collapses at the same time as $f_{\tau_0}$, and the first derivative
converges to that of $f_{\tau_0}$.  
\end{thm}
\begin{proof}
The generic instability follows again from the time-translation
symmetry of the equation. 
The stability statement is an immediate consequence of the asymptotic
profile given by Proposition \ref{prop:collapsetangentcone}. 
\end{proof}

\tocaddline
\section{Cosmological horizon as stability
obstacle}\label{sec:cosmohorizon}
From here on we will focus on the future stability of a spherically
symmetric solution that expands indefinitely in the future. In general
the Minkowski space $\Real^{1,1+d}$ has the full Poincar\'e group of
symmetries, which consists of spatial and temporal translations,
spatial rotations, Lorentz boosts, and their compositions. Under the
assumption of spherical symmetry, the only relevant symmetry is that
of time translation. And we have see in Theorem
\ref{thm:expansionstability} that \emph{modulo} the symmetry of time
translations, the expanding solutions can be regarded as
asymptotically stable. 

One may then ask na\"\i{}vely whether a similar result holds outside
spherical symmetry: are spherically symmetric future-expanding
solutions asymptotically stable if we allow ourselves the full
Poincar\'e group of symmetries? The answer, as it turns out, is
\emph{no}\footnote{This is for $d\geq 1$, a condition which we will
implicitly assume for the rest of this section. The $d = 0$ case can be
fully recovered from the ODE result: the manifold $\Sphere^0$ consist
of two points which evolve independently following the appropriate
ODE. One can easily show that allowing the
full $(1+1)$-dimensional Poincar\'e group of symmetries (in
particular, space-like translations in addition to time-like
translations) we have modulational stability: both particles converge
to the reference background solution up to global space-time translations.}. 
We first discuss the difficulty by analysing the
\emph{linear stability} of the pseudo-sphere $\dS$, treating the
perturbed solution as a graph over the pseudo-sphere background. Next
we will describe the geometric origins of this difficulty (namely, the
presence of cosmological horizons in de Sitter space) and show that
the na\"\i{}ve statement above must be false. 

\subsection{Geometry of the pseudo-sphere}
By the pseudo-sphere we refer to the isometric image of de Sitter
space embedded in a higher dimensional Minkowski space. As described
in Appendix \ref{app:gaussmap}, the set 
\[ \dS = \Sphere^{1,d+1,1} = \set{(x^0, \dots,x^{d+1}) \in
\Real^{1,d+1}}{-(x^0)^2 + \sum_{i = 1}^{d+1} (x^i)^2 = 1} \]
is a \cpmc manifold with mean curvature $d+1$ and unit inward normal
vector $\vec{n} = - \sum_{i = 0}^{1+d} x^i \partial_{x^i}$. 

Since the indefinite orthogonal group $O(1,d+1)$ preserves the
Minkowski form on $\Real^{1,d+1}$, we see that $\dS$ is invariant
under its action. In particular, this induces a family of
$(d+1)(d+2)/2$ Killing
vector fields on $\dS$ exhibiting its maximally symmetric nature. More
precisely, the Lorentz boosts 
\begin{equation}\label{eq:deflorvf}
\lorvf{i} = x^0 \partial_{x^i} + x^i \partial_{x^0}, \quad i \in
\{1,\dots, d+1\}
\end{equation}
and the spatial rotations
\begin{equation}\label{eq:defrotvf}
\rotvf{ij} = x^i\partial_{x^j} - x^j\partial_{x^i}, \quad i,j\in
\{1,\dots,d+1\}
\end{equation}
generate the symmetries of $\dS$. 

From their definitions it is clear that $\rotvf{ij}$ are always
space-like vector fields. The Lorentz boosts, however, can change
type:
\[ g(\lorvf{i},\lorvf{i}) = (x^0)^2 - (x^i)^2 = -1 + \sum_{j\in
\{1,\dots,d+1\}, j\neq i} (x^j)^2.\]
The sets $\{x^0 = \pm x^i\}$ divide $\dS$ into regions where
$\lorvf{i}$ has fixed type; see Figure \ref{fig1} below. Each of the 
connected components where $\lorvf{i}$ is time-like is \emph{globally
hyperbolic}, and on each such region $\lorvf{i}$ is in fact a
\emph{static} Killing vector field, i.e.\ it is hypersurface
orthogonal.  

\begin{figure}[t]
\centering
\includegraphics[width=0.7\textwidth]{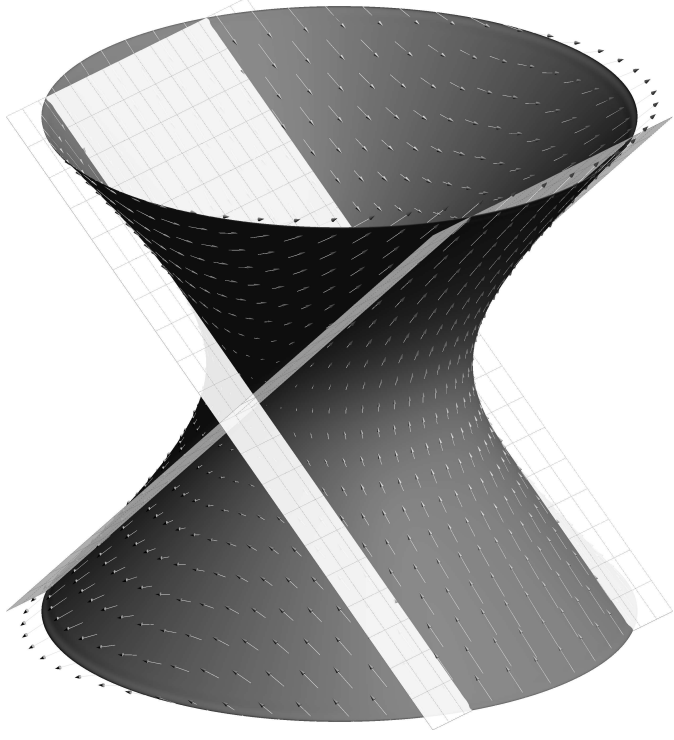}
\caption{The surface $\dS$ in the case $d = 1$. Also shown is the
Lorentz boost vector field $\lorvf{2}$ and the hyperplanes $x^0 = \pm
x^2$. As one sees that hyperplanes divide $\dS$ into six open regions,
inside four of which $\lorvf{2}$ is space-like, and in the other two
$\lorvf{2}$ is time-like. Along the hyperplanes $\lorvf{2}$ has
vanishing Minkowskian length, and the vector field vanishes exactly at
the intersection of the two hyperplanes. }
\label{fig1}
\end{figure}

This hypersurface orthogonality translates into a static decomposition
of the metric. Fix now our attention to the vector field $\lorvf{d+1}$
and a corresponding set on which it is time-like. On this set
($\{x^{d+1} > \abs{x^0}\}\cap \dS$) define the coordinates
$\zeta,\rho,z^i$ for $i\in \{1,\dots,d\}$ and $\sum_{i = 1}^d (z^i)^2 =
1$ (so $z^i$ describes the unit sphere in $\Real^{d}$) by
\begin{equation}
x^\mu = \begin{cases}
\sqrt{1 - \rho^2} \sinh(\zeta) & \mu = 0\\
\sqrt{1 - \rho^2} \cosh(\zeta) & \mu = d+1\\
\rho z^\mu & \mu\in \{1,\dots,d\}
\end{cases}.
\end{equation}
In this coordinate system the induced metric on $\dS$ takes the form
\begin{equation}\label{eq:staticslicing}
 -(1 - \rho^2) \D*{\zeta}^2 + \frac{1}{1 - \rho^2} \D*{\rho}^2 +
\rho^2 \D*{\omega}^2_{\Sphere^{d-1}}\end{equation}
which is spherically symmetric and explicitly independent of $\zeta$. 

The boundary of the region $\{x^{d+1} > \abs{x^0}\}$ corresponds to
$\rho \to \pm 1$, at which point our coordinate system becomes
degenerate. This is a manifestation of the \emph{cosmological horizon}
that is present in $\dS$, and is related to the fact that this region
is globally hyperbolic. This endows $\dS$ with a rather different
asymptotic causal structure when compared to Minkowski space
$\Real^{1,d}$. 

On Minkowski space, let $\gamma_{1,2}:\Real\to\Real^{1,d}$ represent
the worldline of two inertial observers; in other words $\gamma_{1,2}$
represent two time-like straight lines parametrised by arc-length. We
have the following nice property: for every $s_1\in \Real$, there
exists $s_2\in \Real$ such that $\gamma_1(s_1)$ is in the causal past
of $\gamma_2(s)$ for every $s > s_2$. This property is no longer true
on $\dS$. In particular, if we let $\omega$ be a unit vector in
$\Real^{d+1}$, we can consider the geodesic
$\gamma_\omega: t\mapsto(t,\sqrt{1+t^2}\omega)$ curve along $\dS$. For
every distinct pair $\omega_1,\omega_2$, there exists
$\tilde{t}_1,\tilde{t}_2$ such
that for all $t_1 > \tilde{t}_1$ and $t_2> \tilde{t}_2$,
$\gamma_{\omega_2}(t_2)$ is not in the causal past of
$\gamma_{\omega_1}(t_1)$ and vice versa.

The presence of this cosmological horizon has important consequences
for the solutions of wave equations on a $\dS$ background. Most
notably is the fact that ``structures'' on a fixed ``scale'' tend to
be frozen in place after finite time. One sees this already in Figure
\ref{fig1}. Fix $\tilde{x}^0 \geq 0$. Consider the wave equation on
$\dS$ to the future of $\tilde{x}^0$ with initial data prescribed on
the sphere $\{x^0 = \tilde{x}^0\}$. As already evident in Figure 
\ref{fig1}, by a domain of dependence argument the portion of the 
solution inside the set $\{x^{d+1} >
\abs{x^0}\}$ is entirely independent of the portion of the solution
inside the set $\{x^{d+1} < -\abs{x^0}\}$, as each of the two sets are
future globally hyperbolic.  Noting that $\tilde{x}^0$ gets larger and
larger,
the set $\{x^{d+1} > \abs{x^0}\}$ takes up smaller and smaller angular
size of the constant $\tilde{x}^0$ spheres, we see that we can divide
$\dS \cap \{x^0 \geq \tilde{x}^0\}$ into more and more of these
``mutually independent regions''. As we shall see in the remainder of
this section, this feature
of $\dS$ background introduces an obstacle to parametrizing the
asymptotic behaviour of \cpmc manifolds using a finite dimensional
modulation space. 

We conclude this subsection with a small computation that will be
useful later. Let $\tau$ be the unit future time-like vector field
along $\dS$ that is orthogonal to its constant $x^0$ slices; in terms
of the coordinate system of $\Real^{1,1+d}$ we easily verify that 
\begin{equation}\label{eq:taudef} \tau = \jb{x^0} \partial_{x^0} + \sum_{i = 1}^d \frac{x^0 x^i}{\jb{x^0}} \partial_{x^i}
 = \sum_{i = 1}^d \frac{x^i}{\jb{x^0}} \lorvf{i}.\end{equation}
Indeed $\tau$ is in the span of the radial vector $\sum_1^d
x^i\partial_{x^i}$ and the time-like vector $\partial_{x^0}$ so is
orthogonal to the constant $x^0$ slices, which are round spheres. It
is tangent to $\dS$ as it is a linear combination of $\lorvf{i}$ which
are tangent vector fields. And a direct computation shows its
Minkowskian length is $-1$. A further computation shows that $\tau$ is
a geodesic vector field along $\dS$. 

As $\tau$ is unit and orthogonal to the constant $x^0$ hypersurfaces,
its covariant derivative along said hypersurfaces is the shape tensor 
(see Appendix \ref{app:defmc}), and so is related to the second
fundamental form of the constant $x^0$ slices. 
Examining the definition of the second fundamental form, and
considering the nested embedding of the constant $x^0$ slices into
$\dS$ and $\dS$ into $\Real^{1,d+1}$,  we have that the second 
fundamental form of the constant $x^0$ slices inside $\dS$ is just the
\emph{projection of the second fundamental form of the corresponding
round sphere (which has radius $\jb{x^0}$) in $\Real^{1,d+1}$ onto
$\dS$}. Thus we have derived
\begin{prop}\label{prop:nablatau}
Let $\ooo\covD$ denote the induced Levi-Civita connection along $\dS$,
and let $\tau$ as above. Then 
\[ \ooo\covD_b \tau^c = \frac{x^0}{\jb{x^0}} ( \delta^c_b +
\tau_b\tau^c ) \]
where $\tau_b = \ooo{g}_{ba}\tau^a$ is the metric dual of $\tau$ via
the induced metric $\ooo{g}$ on $\dS$. 
\end{prop}

\subsection{Linear ``instability'' of \texorpdfstring{$\dS$}{de
Sitter} without modulation}\label{sec:linearinstability}
Going back to the \cpmc problem, let us consider a small perturbation
of $\dS$ as a graph over it. More precisely, we consider the manifold
$M_\phi = \set{\vec{x} + \phi(\vec{x})\vec{n}}{\vec{x}\in \dS}$ where
$\phi:\dS \to \Real$ is some smooth function. 
In Appendix \ref{app:linearisationmc} we compute the mean curvature of
$M_\phi$ and its formal linearisation. The linearised equation for
$\phi$ such that $M_\phi$ has constant mean curvature $d+1$ is given
by \eqref{eq:linearMCdS}, which we reproduce here
\[ \Box_{\dS} \phi + (d+1)\phi = 0. \tag{\ref*{eq:linearMCdS}}\]
In this subsection we study the asymptotic behaviour of solutions to
this linearised equation. 

Equation \eqref{eq:linearMCdS} takes the form of a Klein-Gordon
equation but with \emph{negative mass}. Experience with static
space-times (such as Minkowski space) tells us that wave equations
with sufficiently negative potentials will exhibit generically
exponential growth of the solution. This is clear when we write the
equations in the form 
\[ \partial_\zeta^2 \phi = - L \phi \]
where $L$ is a time-independent Schr\"odinger operator whose spectrum 
protrudes into the negative real axis. As we have seen in the previous
subsection, the induced metric on $\dS$ in 
fact admits such a static decomposition, if we restrict to a region
where a given Lorentz boost vector field is time-like. Furthermore,
the spherical symmetry of the static decomposition
\eqref{eq:staticslicing} implies that the exponential growth of the
solution to the equation also induces exponential growth of (some of)
the derivatives. 

One may however argue that the $(\zeta,\rho,z^i)$ coordinate system of
\eqref{eq:staticslicing}, in addition to not covering the entirety of
$\dS$ without degeneration, is also not representative of the true
asymptotic behaviour of a \cpmc manifold, due to the fact that every
constant $\zeta$ slice passes through the sphere $x^0 = x^{d+1} = 0$,
and so the behaviour of the solution as $\zeta\to\infty$ may not be
reflective of what we physically think of as asymptotic behaviour,
where $x^0\to \infty$. With regards to the ``true'' asymptotic
behaviour, one may expect something better. This is in view of known
results concerning the wave and (positive-mass) Klein-Gordon equations
on de Sitter backgrounds (see e.g.\ \cite{MeSaVa2014} and references
therein) that suggest one expects the solution itself to converge to a
(possibly non-zero) constant, with decaying derivatives. One may hope
that even in the case of the negative-mass Klein-Gordon term, the 
derivatives obey certain improved decay or boundedness properties
compared to the unboundedly growing solution. 

To understand the more physically relevant asymptotics, we first write
down explicitly the operator $\Box_{\dS}$ in coordinates. Let $t =
x^0$ and $\omega$ be some coordinate system for the sphere
$\Sphere^{d}$, the metric for $\dS$ in this cylindrical coordinate
system can be expressed as
\begin{equation}
- \frac{1}{\jb{t}^2} \D*{t}^2 + \jb{t}^2\D*{\omega}^2, 
\end{equation}
from which we can write down the wave operator as
\begin{equation}\label{eq:boxdscoord}
\Box_{\dS}\phi = - \frac{1}{\jb{t}^{d-1}} \partial_t \left(
\jb{t}^{d+1} \partial_t \phi\right) + \frac{1}{\jb{t}^2} \subDD \phi
\end{equation}
where $\subDD$ is the spherical Laplacian on $\Sphere^{d}$. The
spherical symmetry allows us to decompose a solution based on angular
momentum $\ell$, which is a non-negative integer. For a solution with
angular momentum $\ell$, which we denote by $\psi_\ell$, we have that
\[ \subDD\psi_\ell = - \ell(\ell + d + 1) \psi_\ell.\]

\subsubsection{Spherically symmetric case}
For $\ell = 0$, the linearised \cpmc equation \eqref{eq:linearMCdS} reduces via 
\eqref{eq:boxdscoord} to the ODE
\[ \jb{t}^2\psi_0'' + (d+1)t \psi_0' = (d+1)\psi_0.\]
Substituting $\psi_0 = t \hat{\psi}_0$ we get
\[
\jb{t}^2( t \hat{\psi}_0'' + 2 \hat{\psi}_0') + (d+1) t(t
\hat{\psi}_0' + \hat{\psi}_0) = (d+1) t \hat{\psi}_0
\]
which gives us the following equation for $f = \hat{\psi}_0'$:
\begin{equation}
\jb{t}^2 tf' + 2\jb{t}^2 f + (d+1) t^2 f = 0.
\end{equation}
This equation we can explicitly integrate
\[
\frac{f'}{f} = - \frac{2}{t} - \frac{(d+1)t}{\jb{t}^2} \implies \log f
+ C = -2 \log t - (d+1) \log \jb{t}
\]
or
\begin{equation}\label{eq:renormalderive0}
\hat{\psi}_0' = \frac{C}{t^2\jb{t}^{d+1}}.
\end{equation}
The case $C = 0$ corresponds to $\hat{\psi}_0 = C'$ and hence $\psi_0
= C' t$. This solution corresponds to the \emph{temporal translation
symmetry} of the background $\Real^{1,d+1}$. Note that in terms of the
``proper time'' for the constant $\omega$ observers in $\dS$, the
linear in $t$ growth translates to an exponential growth\footnote{This
follows by noting that with
$\tau$ being the unit future time-like vector orthogonal to constant
$x^0$ slices with coordinate expression \eqref{eq:taudef}, we have
$\tau(x^0) = \jb{x^0}$ which shows the $x^0$ coordinate is the
hyperbolic cosine of the elapsed proper time since $x^0 = 0$ for
observers described by $\tau$.}.

The equation \eqref{eq:renormalderive0} shows that $\hat{\psi}_0'$ is
integrable as $t\to\infty$ and so we have that $\hat{\psi}_0$ converges
to a finite constant generically, and signals the generic linear in
$t$ growth of a solution, agreeing\footnote{Some change of coordinates
is involved to see this; see Remark \ref{rmk:notionsofstability}.}
 with our analysis in Section
\ref{sec:rotsym}. We summarise the results as

\begin{prop}
For spherically symmetric perturbations, the solutions $\psi_0$ to the
linearised equation grows linearly in $t$ generically. The
renormalised quantity $\frac{1}{t}\psi_0$ is bounded, and its first
derivative decays; this is while the derivative $\psi'_0$ generically
remains bounded but does not decay. 
\end{prop}

\subsubsection{Higher angular momentum case}
For $\ell > 0$, instead of commuting with the $t$ weight, we commute
with a $\jb{t}$ weight. Writing $\varphi = \jb{t} \breve{\varphi}$ we have
\[ \Box_{\dS} \varphi + (d+1)\varphi = \jb{t} \left[ \Box_{\dS}
\breve{\varphi}
- 2t \partial_t \breve{\varphi} + \frac{d}{\jb{t}^2} \breve{\varphi}
  \right]. \]
Multiplying the equation with $\jb{t}^{2-d} \partial_t\breve{\varphi}$
and integrating against the space-time volume form $\jb{t}^{d-1} \D{t}
\D{\omega}$ we have
\begin{align*}
0 &= \iint \jb{t}^2 \left[ \Box_{\dS} \breve{\varphi} - 2 t
\breve{\varphi}_t  + \frac{d}{\jb{t}^2} \breve{\varphi} \right]
\breve{\varphi}_t \D{t} \D{\omega} \\
& = \iint - \frac{1}{2\jb{t}^{2d-2}} \partial_t \left( \jb{t}^{d+1}
\breve{\varphi}_t\right)^2 - \frac12 \partial_t \left(
\breve{\varphi}_\omega^2 - d \breve{\varphi}^2\right) - 2t \jb{t}^2
\breve{\varphi}_t^2 \D{t} \D{\omega} \\
& = \iint - \frac12 \partial_t \left( \jb{t}^4 \breve{\varphi}_t^2 +
\breve{\varphi}_\omega^2 - d \breve{\varphi}^2 \right) - (d+1) t
\jb{t}^2 \breve{\varphi}_t^2 \D{t} \D{\omega}
\end{align*}
which implies
\begin{equation}\label{eq:linearenergyestimate}
\int \jb{t}^4 \breve{\varphi}_t^2 +
\breve{\varphi}_\omega^2 - d \breve{\varphi}^2 \D{\omega} \Big]^{t =
T}_0  = - 2(d+1)
\iint t \jb{t}^2 \breve{\varphi}_t^2 \D{t}\D{\omega} < 0
\end{equation}
provided we restrict to the forward region $t > 0$. 
For the $\ell \geq 1$ spherical harmonics the energy quantity 
controlled is positive semi-definite; for $\ell > 1$ the energy
quantity is in fact coercive. The monotonicity of
\eqref{eq:linearenergyestimate} already implies that, as long as we project
out the spherically symmetric mode first, that $\int \jb{t}^4
\breve{\varphi}_t^2 \D{\omega}$ remains bounded. This can be strengthened
slightly: since $\int \jb{t}^4 \breve{\varphi}_t^2\D{\omega}$ is monotonically
decreasing in $t$, the limit as $t\to\infty$ must exist. Suppose
$\lim_{t\to\infty} \int \jb{t}^4 \breve{\varphi}_t^2\D{\omega} > 0$, 
this gives asymptotically a lower bound 
\[ \int \jb{t}^2 \breve{\varphi}_t^2 \D{\omega} \gtrsim \frac{1}{\jb{t}^2}\]
which would imply that the right hand side of
\eqref{eq:linearenergyestimate} is not integrable in time, giving a
contradiction. Therefore we have that

\begin{prop}\label{prop:lineargrowthhighermode}
When $\ell \geq 1$, we have that 
\[ \lim_{t\to\infty} \jb{t}^2 \partial_t[ \jb{t}^{-1} \psi_\ell(t)]  =
0,\]
which implies that the $\jb{t}^{-1}\psi_\ell(t)$ converges to a finite
value for each fixed $\omega\in \Sphere^d$, and that $\psi_\ell$ grows
at most linearly. 
\end{prop}

\begin{rmk}\label{rmk:lineargrowthmodeest}
For the special case $\ell = 1$, for each fixed $\omega$, the
renormalised quantity $\jb{t}^{-1}\psi_1(t,\omega)$ is such that its
time derivative is spanned by $\{0, \jb{t}^{-d-3}\}$. The former
corresponds to $\psi_1(t,\omega) = \psi_1(0,\omega) \jb{t}$, which
is the linear instability associated with the spatial translation
symmetries of $\Real^{1,d+1}$. Of the symmetries in the Poincar\'e
group, only the translations do \emph{not} fix $\dS$; the indefinite
orthogonal group $O(1,d+1)$ generates no additional linear
instabilities. 

This, however, should not be interpreted as the non-existence of
additional linear instabilities. In fact, as we will see immediately
below, \eqref{eq:linearenergyestimate} and the linear upper bound for
$\psi_\ell$ in Proposition \ref{prop:lineargrowthhighermode} are as
good as we can get. 
\end{rmk}

\subsection{Cosmological horizon and a no-go result}
Playing with a domain of dependence argument using the presence of
cosmological horizons on $\dS$ and the instability in the spherically
symmetric cases gives us immediately 
\begin{enumerate}
\item linear mode instability for infinitely many angular momenta
(Theorem \ref{thm:allmodeunstable}), and
\item a no-go theorem for approximating the asymptotic profile of
perturbed solutions using directly $\dS$ and its images under the
Poincar\'e group (Theorem \ref{thm:badasymp}). 
\end{enumerate}
The first result indicates that were one to try to na\"\i{}vely
approach the \cpmc stability problem from the point of view of
modulation theory applied to the set-up where the solution manifold
$M_\phi$ is regarded as a graph in the normal bundle of $\dS$, one
will necessarily have to contend with an infinite dimensional family
of modulations. The second result shows that in the
non-spherically-symmetric case, the asymptotic behaviour of solutions
should be more complicated than suggested by Corollary
\ref{cor:expansiontolightcone} and Theorem
\ref{thm:expansionstability}. 

\begin{thm}\label{thm:allmodeunstable}
Let $Y_{m,\ell}$ denote spherical harmonics for $\Sphere^{d}$. Let
$\{(m_\alpha,\ell_\alpha)\}$ be a fixed, finite set of parameters.
Then there exists a solution to \eqref{eq:linearMCdS} satisfying 
\[ \norm[L^\infty(\Sphere^d)]{\phi(t)} \gtrsim t \]
for all sufficiently large $t$, and that 
\[ \int_{\Sphere^d} \phi(t,\omega) Y_{m_{\alpha},\ell_\alpha}(\omega) \D{\omega} = 0
\]
for parameters belonging to our fixed finite set. 
\end{thm}
\begin{proof}
Let $N > \#\{(m_\alpha,\ell_\alpha)\}$ and choose $N$ distinct points
on $\Sphere^d$ labelled $\{\omega_1,\dots,\omega_N\}$. Let
$\delta_0 = \min_{i\neq j} \abs{\omega_i - \omega_j}$. Fix
$t_0 \geq 8/\delta_0$. Consider the sets 
\[ B_i(r) = \set{(t_0,\omega)\in\dS}{\abs{\omega - \omega_i} < r},\]
by our discussion of the cosmological
horizon above, we see that there exists $\rho_0> 0$ such that,
\begin{itemize}
\item $B_i(2\rho_0)$ are mutually disjoint;
\item the domain of dependence for $B_i(\rho_0)$ contains the ray
$\set{(t,\omega_i)}{t \geq t_0}$. 
\end{itemize}
Now let $\phi_i^{(0)}: \{t_0\}\times\Sphere^d \to\Real$ be a smooth 
bump function supported on
$B_i(2\rho_0)$ and such that $\phi_i^{(0)} = 1$ on $B_i(\rho_0)$.
Consider the Cauchy problem for \eqref{eq:linearMCdS} with initial
data prescribed on $\{t_0\}\times\Sphere^d$ such that
\[
\phi|_{t_0}  = \sum_{i = 1}^N \epsilon_i t_0 \phi_i^{(0)}, \qquad
\text{and} \qquad 
\partial_t\phi|_{t_0} = \sum_{i = 1}^N \epsilon_i \phi_i^{(0)}.
\]

We have the inside the domain of dependence of $B_i(\rho_0)$, and in
particular along the ray $[t_0,\infty)\times\{\omega_i\}$, 
solution coincides with the spherically symmetric solution
$\phi(t,\omega) = \epsilon_i t$, and hence as long as for some
$\epsilon_i$ we have that $\epsilon_i \neq 0$, we have the lower bound
on the $L^\infty$ norm as claimed. 

It remains to verify the assumption on the spherical harmonics. But
the condition on the spherical harmonics reduce to a family of
$\#\{(m_\alpha,\ell_\alpha)\}$ linear equations which can be solved
for generic choices of $\omega_i$, and with the prescription
$\epsilon_1 = 1$. 
\end{proof}

\begin{thm}\label{thm:badasymp}
Suppose the full \cpmc problem, as a perturbation of $\dS$, has small
data global existence. Then there exist \cpmc manifolds which are not
eventually tangent to a light cone. More precisely, fix a initial time
$t_0 > 0$, then there exist a \cpmc manifold $M$, which can be chosen to
be arbitrarily close to $\dS$ at time $t_0$, such that for every
$(\tau_0,\xi_0) \in \Real^{1,1+d}$ we have
\[ \lim_{t\to \infty} \sup_{(t,x)\in M} \abs{ \abs{x - \xi_0} - t +
\tau_0} > 0.\]
\end{thm}
\begin{proof}
Consider the initial data as a graph over the $t_0$ slice of $\dS$. 
Since $t_0 > 0$, we can choose initial data at $t_0$ such that
restricted to the set $\{x^{d+1} \geq \sqrt{1 + t_0^2}\}$ agree with
the spatial translation of $\dS$ in the positive $x^{d+1}$ direction
by distance $\epsilon$, and such that restricted to the set $\{x^{d+1}
\leq - \sqrt{1 + t_0^2}\}$ agree with the spatial translation of $\dS$
in the negative $x^{d+1}$ direction by distance $\epsilon$, and
smoothly joined in between. Our assumption of small data global
existence implies that for sufficiently small $\epsilon$ the Cauchy
problem given this initial data can be solved globally. In the
relevant domains of dependence, which in particular will contain a
small neighbourhood of the curves with $x^1 = \cdots = x^d = 0$ and $t
> t_0$ the solution will agree with the two spatially translated
solutions. From this we see that for any $t > t_0$ the quantity 
\[ \sup_{(t,x)\in M} \abs{\abs{x - \xi_0} - t + \tau_0} > \epsilon /
2.\qedhere\]
\end{proof}

\begin{rmk}
That $t_0 > 0$ is only for convenience. The data can be chosen
to be arbitrarily close to that of $\dS$ at any one fixed time by
a Cauchy stability argument. By choosing larger $t_0$ we can also
include more ``pieces'' of translated solutions, as is done in the
proof to Theorem \ref{thm:allmodeunstable}. This shows that the
asymptotic profiles for the full problem is also, in some sense,
infinite dimensional, in accordance to the idea that the cosmological
horizon freezes in perturbations after finite time. 
\end{rmk}

\begin{rmk}
For the last inequality in the proof of Theorem \ref{thm:badasymp} 
to hold we used that $\dS$ is ruled by null
geodesics. Thus as $t \to\infty$, while the \emph{angular size} of the
set $\{x^{d+1} \geq \sqrt{1 + t_0^2}\}\cap \dS$ gets smaller and
smaller, the physical ``diameter'' of this set at a fixed time remains
roughly constant. 
\end{rmk}

Observe in the statement of Theorem \ref{thm:badasymp} that
\emph{small data global existence is assumed}. One may ask whether it
is possible to prove this fact independently of asymptotic stability
(which is false in the graphical setting in view of Theorem
\ref{thm:badasymp}). In the graphical setting Theorem
\ref{thm:allmodeunstable} suggests that this would be difficult. In
fact, in view of Theorem \ref{thm:allmodeunstable}, the asymptotics
given by Proposition \ref{prop:lineargrowthhighermode} implies that
$\partial_t\phi(t)$ for a generic solution of \eqref{eq:linearMCdS} 
converges to a non-zero constant (for each fixed $\omega$). This poses
a severe obstacle to studying the \cpmc problem in the formulation of
a quasilinear wave equation for a graph over the normal bundle of
$\dS$. Nevertheless, as we shall see in the remainder of this paper,
the desired small data global existence is in fact true, and we can
also obtain some control over the asymptotic behaviour of the
solutions. 

\tocaddline
\section{Inverse-Gauss-map gauge}\label{sec:IGMgauge}

The linear analysis of Section \ref{sec:linearinstability} leaves
still a ray of hope: while the analysis shows that treating a \cpmc
manifold which is a perturbation of $\dS$ as a graph over its normal
bundle leads to great difficulties in studying the associated
quasilinear wave equations, it also shows that certain
\emph{renormalised} quantities behave better. In particular,
\eqref{eq:linearenergyestimate} gives derivative control over 
the rescaled solution $\jb{t}^{-1} \phi$ to \eqref{eq:linearMCdS}. One
approach to the \cpmc stability problem would be to derive the
equation for this rescaled quantity, and hope that its equation of
motion can be studied perturbatively under this derivative decay. Here
we take an alternative approach. Instead of treating a \cpmc manifold
as a graph over the normal bundle of $\dS$, we introduce in this
section the \emph{inverse-Gauss-map gauge}. In this gauge the
evolution equation for the \cpmc problem is most naturally expressed
in terms of the Codazzi equations for the embedding, which is
automatically at the \emph{derivative level} compared to the graphical
formulation. In other words, by reformulating the question in a
geometric manner, we will be able to avoid some of the difficulties
that manifest in the graphical treatment of the problem.

For a time-like hypersurface $M$ of $\Real^{1,d+1}$, the Gauss map is
a (smooth) map $G: M\to \dS$, where the value of $G$ corresponds to the
out-ward unit-normal of the hypersurface (see Appendix
\ref{app:gaussmap}). We will consider the regime in which $G$ is a
diffeomorphism onto its image; this will be the case, for example, 
if we have a sufficiently small perturbation of $\dS$ (see Remark
\ref{rmk:invertible} below). We let $\Phi$ denote the 
inverse map to $G$ in this case. 

Now, using the \emph{ambient geometry} of $\Real^{1,d+1}$ we can
naturally identify the tangent spaces $T_pM$ and $T_{G(p)} \dS$. Note
that this identification is \emph{different} from the identification
afforded by the derivative map $\D*{G}$ or $\D*\Phi$. Under
this identification we observe that $\D*{G}$ as well as $\D*\Phi$ can
each be interpreted as a $(1,1)$-tensor on $TM$ and $T\dS$
respectively. By the \emph{inverse-Gauss-map gauge} we mean that $M$
is studied as the image of $\dS$ under the mapping $\Phi$, and its
geometry studied through the ``tensor field'' $\D*\Phi$. 

In classical differential geometry (see also Appendices
\ref{app:defmc} and \ref{app:gaussmap}) it is well-known that
the differential of the Gauss map is the shape operator. Using the
identification above, as well as the fact that $\D*{G}\circ \D*{\Phi} =
\Id$, we see that the first and second fundamental forms, which we now
write as $g$ and $k$, of $M$ as a
hypersurface in $\Real^{1,d+1}$ can be expressed as tensor fields
\emph{over} $\dS$, where $\ooo{g}$ is the metric of $\dS$. For
convenience we also write $A_a^b = (\D*\Phi)_a^b$ the $(1,1)$-tensor
interpretation of the deformation. 
\begin{align}
g_{ab} &= \ooo{g}_{cd} A_a^c A_b^d \\
k_{ab} &= \ooo{g}_{ac} A_b^c.
\end{align} 
The second fundamental form $k_{ab}$ is symmetric. The equation of
motion that we will be studying then arises from the \emph{Codazzi}
equations of the embedding $M\to \Real^{1,d+1}$, which reads
\begin{align}
\label{eq:codazzi} \covD_a k_{bc} - \covD_b k_{ac} &= 0,\\
\label{eq:dualcodazzi} g^{ac} \covD_a k_{bc} &= 0.
\end{align}
The equation \eqref{eq:codazzi} is a direction consequence of Codazzi
equation using that $\Real^{1,d+1}$ is flat, and contains within it
the integrability condition that ``$\D*{\D*\Phi} = 0$''; on the other
hand
\eqref{eq:dualcodazzi} is obtained from taking the $g$ trace of
\eqref{eq:codazzi} and using the condition that $\trace_g k = d+1$ is
constant for our solution. 

The system \eqref{eq:codazzi} and \eqref{eq:dualcodazzi} clearly forms
a \emph{quasilinear divergence-curl} system for the unknown $A_a^b$,
and in this form already exhibits the hyperbolic nature of the
equations; it also has some formal similarity with the systems of
nonlinear electrodynamics. As we are interested in the case of
perturbations of $\dS$, we observe that our $A_a^c$ can be written as
the perturbation $A_a^c = \delta_a^c + \phi_a^c$, with $\phi \equiv 0$
being the specific $\dS$ solution.

For convenience we will also write $\phi_{ab} = \ooo{g}_{ac}\phi^c_b$,
we have that by assumption $k_{ab} = \ooo{g}_{ab} + \phi_{ab}$, and
hence $\phi_{ab}$ is a symmetric two-tensor. 

We will write $\Gamma_{ab}^c$ for the Christoffel symbols of $g_{ab}$
relative to the background metric $\ooo{g}_{ab}$, that is to say
\begin{equation}
\Gamma_{ab}^c \eqdef \frac12 g^{cd} \left( \ooo\covD_a g_{bd} +
\ooo\covD_b g_{ad} - \ooo\covD_d g_{ab}\right) 
\end{equation} 
where $\ooo\covD$ is the Levi-Civita connection relative to the metric
$\ooo{g}$. Observe that $\Gamma$ is at the level of one derivative of
$\phi$: this could potentially cause a bit of problem with expansion
of the
equations. Fortunately we have some nice cancellations that comes in
to play. In the remaining part of this section we will re-write the
equations of motion \eqref{eq:codazzi} and \eqref{eq:dualcodazzi} in
terms of $\phi$ and the fixed background $\ooo{g}$. 

First we expand $\Gamma_{ab}^c$, using that $\ooo\covD
\ooo{g} = 0$, and that
\begin{equation}\label{eq:defgfromphi}
g_{ab} = \ooo{g}_{ab} + 2\phi_{ab}  + \phi_{ad}\phi^d_b.
\end{equation}
We have
\[
\Gamma_{ab}^c = \frac12 g^{cd} \big[ \ooo\covD_a (2\phi_{bd} +
\phi_{bf} \phi_d^f)
+ \ooo\covD_b (2 \phi_{ad} + \phi_{af} \phi_d^f)
 - \ooo\covD_d (2 \phi_{ab}  + \phi_{bf}
\phi_a^f)\big].
\]
So \eqref{eq:codazzi} becomes
\begin{align*}
0 &= \covD_a k_{bc} - \covD_b k_{ac} \\
&= \ooo\covD_a \phi_{bc} - \ooo\covD_{b} \phi_{ac} - \Gamma_{ab}^e
k_{ec} - \Gamma_{ac}^e k_{be} + \Gamma_{ba}^e k_{ec} + \Gamma_{bc}^e
k_{ae} \\
& = \ooo\covD_a \phi_{bc} - \ooo\covD_{b} \phi_{ac} - \Gamma_{ac}^e
k_{be} + \Gamma_{bc}^e k_{ae} \\
&= \ooo\covD_a \phi_{bc} - \ooo\covD_{b} \phi_{ac} \\
& \qquad - \frac12 g^{ef} \left[ \ooo\covD_a (2\phi_{cf} +
\phi_{ch} \phi_f^h)
+ \ooo\covD_c (2 \phi_{af} + \phi_{ah} \phi_f^h)
 - \ooo\covD_f (2 \phi_{ac}  + \phi_{ch}\phi_a^h) \right] k_{be} \\
& \qquad + \frac12 g^{ef} \left[ \ooo\covD_b (2\phi_{cf} +
\phi_{ch} \phi_f^h)
+ \ooo\covD_c (2 \phi_{bf} + \phi_{bh} \phi_f^h)
 - \ooo\covD_f (2 \phi_{bc}  + \phi_{ch}\phi_b^h) \right] k_{ae}
\end{align*}
Now we use the following: let $\psi_a^c$ be such that 
\begin{equation}\label{eq:defpsi} 
(\delta_a^c + \psi_a^c)(\delta_c^b + \phi_c^b) = \delta_a^b
\end{equation}
For $\phi$ sufficiently small this exists as the linear mapping $A$ is
invertible. This in particular implies that $\psi\phi = \phi\psi$
(they commute). We have that by definition 
\begin{equation} g^{ef} = (\delta_a^f + \psi_a^f) \ooo{g}^{ab}
(\delta_b^e +
\psi_b^e), \end{equation}
so that
\begin{equation}\label{eq:inversedecomp}
g^{ef}k_{ae} = \delta_a^f + \psi_a^f.
\end{equation}
Note that \eqref{eq:inversedecomp} implies, via our assumption that
$\trace_g k = d+1$, that 
\begin{equation}\label{eq:inversetracefree}
\trace\psi = \psi_a^a = 0,
\end{equation}
this in turn implies that
\begin{equation}\label{eq:traceid}
\phi_a^a + \phi_a^b \psi_b^a = 0.
\end{equation}

\begin{rmk}
The equations \eqref{eq:inversetracefree} and \eqref{eq:traceid}
forces certain constraints on the prescribed initial data. More
precisely, we should think that the free data to be prescribed would
be $\psi$, which satisfies \eqref{eq:inversetracefree} and from which
we can think about $\phi$. One can equivalently write down the
equation for $\psi$; however it seems that the equations for $\phi$ is
somewhat simpler than that of $\psi$. 
\end{rmk}

\begin{rmk}
Observe that in \eqref{eq:traceid}, since $\psi$ is trace-free, we
can decompose $\phi = \hat{\phi} + \tilde{\phi}$, and similarly $\psi
= \hat{\psi} + \tilde{\psi}$, into pure-trace and traceless parts.
\eqref{eq:inversetracefree} says that $\hat{\psi} = 0$. The identity
\eqref{eq:traceid} becomes $ \trace \hat{\phi} +
\tilde{\phi}:\tilde{\psi} = 0$, indicating that the trace part of
$\phi$ should be ``quadratic'' in size compared to the trace-free
part.
\end{rmk}

Using this decomposition we expand \eqref{eq:codazzi} to get
\begin{align*}
0 &= \ooo\covD_a \phi_{bc} - \ooo\covD_{b} \phi_{ac} \\
& \qquad - \frac12 \left[ \ooo\covD_a (2\phi_{cf} +
\phi_{ch} \phi_f^h)
+ \ooo\covD_c (2 \phi_{af} + \phi_{ah} \phi_f^h)
 - \ooo\covD_f (2 \phi_{ac}  + \phi_{ch}\phi_a^h) \right] (\delta_b^f
   + \psi_b^f) \\
& \qquad + \frac12 \left[ \ooo\covD_b (2\phi_{cf} +
\phi_{ch} \phi_f^h)
+ \ooo\covD_c (2 \phi_{bf} + \phi_{bh} \phi_f^h)
 - \ooo\covD_f (2 \phi_{bc}  + \phi_{ch}\phi_b^h) \right] (\delta_a^f
   + \psi_a^f)\\
&= -\ooo\covD_a \phi_{bc} + \ooo\covD_{b} \phi_{ac} -
\ooo\covD_a (\phi_{ch}\phi_b^h)  + \ooo\covD_b (\phi_{ch}\phi_a^h)\\
& \qquad - \frac12 \left[ \ooo\covD_a (2\phi_{cf} +
\phi_{ch} \phi_f^h)
+ \ooo\covD_c (2 \phi_{af} + \phi_{ah} \phi_f^h)
 - \ooo\covD_f (2 \phi_{ac}  + \phi_{ch}\phi_a^h) \right] \psi_b^f \\
& \qquad + \frac12 \left[ \ooo\covD_b (2\phi_{cf} +
\phi_{ch} \phi_f^h)
+ \ooo\covD_c (2 \phi_{bf} + \phi_{bh} \phi_f^h)
 - \ooo\covD_f (2 \phi_{bc}  + \phi_{ch}\phi_b^h) \right]\psi_a^f\\
&= -\ooo\covD_a \phi_{bc} + \ooo\covD_{b} \phi_{ac} -
\ooo\covD_a (\phi_{ch}\phi_b^h)  + \ooo\covD_b (\phi_{ch}\phi_a^h)\\
& \qquad - \left[ \ooo\covD_a \phi_{cf} + \ooo\covD_c \phi_{af} 
 - \ooo\covD_f \phi_{ac}  \right] \psi_b^f + \left[ \ooo\covD_b
   \phi_{cf} + \ooo\covD_c  \phi_{bf} 
 - \ooo\covD_f \phi_{bc}  \right]\psi_a^f\\
& \qquad - \frac12 \left[ \ooo\covD_a (\phi_{ch} \phi_f^h)
+ \ooo\covD_c (\phi_{ah} \phi_f^h)
 - \ooo\covD_f (\phi_{ch}\phi_a^h) \right] \psi_b^f \\
& \qquad + \frac12 \left[ \ooo\covD_b (\phi_{ch} \phi_f^h)
+ \ooo\covD_c (\phi_{bh} \phi_f^h)
 - \ooo\covD_f (\phi_{ch}\phi_b^h) \right]\psi_a^f \\
&= -\ooo\covD_a \phi_{bc} + \ooo\covD_{b} \phi_{ac} -
\ooo\covD_a (\phi_{ch}\phi_b^h)  + \ooo\covD_b (\phi_{ch}\phi_a^h)\\
& \qquad - \left[ \ooo\covD_a \phi_{cf} + \ooo\covD_c \phi_{af} 
 - \ooo\covD_f \phi_{ac}  \right] \psi_b^f + \left[ \ooo\covD_b
   \phi_{cf} + \ooo\covD_c  \phi_{bf} 
 - \ooo\covD_f \phi_{bc}  \right]\psi_a^f\\
& \qquad + \frac12 \left[ \ooo\covD_a \phi_{ch} 
+ \ooo\covD_c \phi_{ah} \right] (\phi_b^h + \psi_b^h) - \frac12 \left[
\ooo\covD_b \phi_{ch} 
+ \ooo\covD_c \phi_{bh}\right](\phi_a^h + \psi_a^h) \\
& \qquad - \frac12 \left[ \phi_{ch}\ooo\covD_a \phi_f^h +
\phi_{ah}\ooo\covD_c \phi_f^h 
 - \ooo\covD_f (\phi_{ch}\phi_a^h) \right] \psi_b^f \\
& \qquad + \frac12 \left[ \phi_{ch} \ooo\covD_b \phi_f^h
+ \phi_{bh}\ooo\covD_c\phi_f^h
 - \ooo\covD_f (\phi_{ch}\phi_b^h) \right]\psi_a^f
\end{align*}
where we used that $\phi\psi = \psi \phi = - \psi - \phi$. This allows
us to further simplify 
\begin{align*}
0 &= -\ooo\covD_a \phi_{bc} + \ooo\covD_{b} \phi_{ac} \\
& \qquad + \frac12\left(
\ooo\covD_c\phi_{ah} - \ooo\covD_a \phi_{ch}\right)\phi_b^h + \frac12
\left( \ooo\covD_b\phi_{ch} - \ooo\covD_c\phi_{bh}\right) \phi_a^h +
\left( \ooo\covD_b \phi_{ah} - \ooo\covD_a\phi_{bh}\right) \phi_c^h\\
& \qquad - \frac12 \left[ \ooo\covD_a \phi_{cf} + \ooo\covD_c
\phi_{af} 
 - 2\ooo\covD_f \phi_{ac}  \right] \psi_b^f + \frac12 \left[
   \ooo\covD_b
   \phi_{cf} + \ooo\covD_c  \phi_{bf} 
 - 2\ooo\covD_f \phi_{bc}  \right]\psi_a^f\\
& \qquad - \frac12 \left[ \phi_{ch}(\ooo\covD_a \phi_f^h- \ooo\covD_f
\phi_a^h) +
\phi_{ah}(\ooo\covD_c \phi_f^h - \ooo\covD_f\phi_c^h) \right] \psi_b^f
\\
& \qquad + \frac12 \left[ \phi_{ch} (\ooo\covD_b \phi_f^h -
\ooo\covD_f\phi_b^h ) + \phi_{bh}(\ooo\covD_c\phi_f^h -
\ooo\covD_f\phi_c^h)\right]\psi_a^f.
\end{align*}
From this we can extract an exterior-derivative structure by writing
the equation as some coefficients times $\ooo\covD_{[i}\phi_{j]k}$.
More precisely, we have that the above expression can be further
simplified to be
\begin{multline}\label{eq:coda2}
0 = \ooo\covD_{[i}\phi_{j]k} \cdot \Big[
(\delta^i_b+ \psi^i_b)\delta^j_a(\delta^k_c+ \phi^k_c) - (\delta^i_a +
\psi^i_a)\delta^j_b(\delta^k_c + \phi^k_c) + \delta^i_c \delta^j_a
\phi^k_b -
\delta^i_c\delta^j_b \phi^k_a  \\
+ \psi^i_b \delta^j_c(\delta^k_a + \phi^k_a)  - \psi^i_a \delta^j_c
  (\delta^k_b+ \phi^k_b)\Big].
\end{multline}
The most important feature of \eqref{eq:coda2} is that the term in the
bracket can be written as 
\[ 2\delta_b^i\delta_a^j\delta_c^k + O(|\phi|,|\psi|) \]
which implies immediately that as a linear mapping on $T^{0,3}\dS$ to itself, it
has no null space when $\phi,\psi$ are small. Furthermore, we can
rewrite \eqref{eq:coda2} as
\begin{multline}\label{eq:coda3}
0 = \ooo\covD_{[i}\phi_{j]k} \cdot \Big[ (A^{-1})^i_b \delta^j_b A^k_c
- (A^{-1})^i_a\delta^j_b A^k_c \\ + (A^{-1})^i_b \delta^j_c A^k_a -
  (A^{-1})^i_a\delta^j_c A^k_b + \delta^i_c\delta^j_b\delta^k_a -
\delta^i_c\delta^j_a\delta^k_b\Big],
\end{multline}
where $A^a_c = \delta^a_c + \phi^a_c$ is as defined in the beginning
of this section. \emph{Provided that the term inside the brackets has no
non-trivial kernel}, we can more conveniently write \eqref{eq:coda3} as 
\begin{equation}\label{eq:codafinal}
0 = \ooo\covD_{[a}\phi_{b]c}.
\end{equation}
As we will see in Proposition \ref{prop:coeffinvert} below, in the 
situations that we will be interested in (small
perturbations of the spherically symmetric, expanding solutions), this
condition is satisfied. For now let us assume \eqref{eq:codafinal} and
continue. 

An immediate consequence of \eqref{eq:codafinal} is a simplified
expression for the Christoffel symbol. Indeed, we can simplify to
\begin{equation}\label{eq:christre}
\Gamma^c_{ab} = g^{cd} ( \ooo\covD_a \phi_{bf})(\delta^f_d + \phi^f_d)
= (\delta^c_e + \psi^c_e) (\ooo\covD_a \phi_b^e)
\end{equation}
which we note is \emph{symmetric} in the indices $a,b$. 

Next we treat the divergence equation \eqref{eq:dualcodazzi}, which
gives us, in view of \eqref{eq:christre},
\begin{align*}
0 &= g^{ab}\covD_ak_{bc} = \covD_a(g^{ab}k_{bc}) = \covD_a \psi_c^a \\
&= \ooo\covD_a \psi_c^a + \Gamma_{ae}^a \psi_c^e - \Gamma_{ac}^e
\psi_e^a \\
&= \ooo\covD_a\psi_c^a + (\ooo\covD_a \phi_e^a) \psi_c^e + \psi^a_e
(\ooo\covD_a\phi_f^e) \psi_c^f - (\psi^a_f +
\psi^a_e\psi^e_f)(\ooo\covD_a
\phi_c^f).
\end{align*}
Taking a derivative of \eqref{eq:defpsi} we obtain
\begin{equation}\label{eq:derpsi}
\ooo\covD_a \psi^b_c = -(\delta_c^s + \psi_c^s) \ooo\covD_a \phi_s^t
(\delta_t^b + \psi_t^b).
\end{equation}
Using \eqref{eq:derpsi} and \eqref{eq:codafinal} we then have
\begin{align*}
0 &= -(\delta_c^s + \psi_c^s) \ooo\covD_a \phi^a_s - (\delta_c^s +
\psi_c^s) (\ooo\covD_a \phi_s^t)\psi_t^a \\
& \qquad + (\ooo\covD_a \phi_e^a)
\psi_c^e + \psi^a_e (\ooo\covD_a\phi_f^e) \psi_c^f - (\psi^a_f +
\psi^a_e\psi^e_f)(\ooo\covD_a \phi_c^f) \\
& = - \ooo\covD_a\phi^a_c - 2 \psi_t^a \ooo\covD_a\phi_c^t - \psi^a_e
\psi^e_f \ooo\covD_a \phi_c^f. 
\end{align*}
We can write more compactly 
\begin{equation}\label{eq:dcodafinal}
0 = (\delta_f^a + 2 \psi_f^a + \psi^a_e \psi^e_f)(\ooo\covD_a
\phi_c^f).
\end{equation}
Now, we note that since $A$ is self-adjoint relative to $\ooo{g}$, 
from Proposition \ref{prop:linalg} below we have that both $\phi$ and
$\psi$ are self-adjoint relative to both the background metric and its
inverse. In particular, lowering and raising an index in
\eqref{eq:dcodafinal} gives us
\[ 0 = (\delta_e^a + \psi^a_e)(\delta^e_f + \psi^e_f) \ooo{g}^{fd}
\ooo\covD_a \phi_{cd} = g^{ad} \ooo\covD_{a}\phi_{cd}, \]
an even more compact notation for the same equation. 

To summarise, in the inverse-Gauss-map gauge, the CMC equations reduce
to the system 
\begin{equation}\label{eq:mainsystem}
\begin{split}
\ooo\covD_a \phi_b^c - \ooo\covD_b \phi_a^c &= 0, \\
(\delta^a_b + 2 \psi^a_b + \psi^a_d \psi^d_b) \ooo\covD_a \phi^b_c &=
0.
\end{split}
\end{equation}
In the sequel we will study the evolution of this system. 

\begin{rmk}
In terms of the Cauchy problem, the system \eqref{eq:mainsystem} is
equivalent to the full \cpmc problem. Observe that we can reconstruct
the map $\Phi$, the inverse of the Gauss map $G$, by integrating
$\D*\Phi$ along constant $\omega$ lines of $\dS$, given initial data
prescribed on a hypersurface transverse to the constant $\omega$
lines. 

More precisely, using the notation of Remark
\ref{rmk:defcauchyproblem}, observe that knowledge of $\D*{\Upsilon}_0$ and
$\Upsilon_1$ is sufficient to give us the value of the Gauss map $G$
along $\{0\}\times\Sigma$, which by assumption embeds as a space-like
$d$-sphere in $\dS$, and hence is transverse to the constant $\omega$
lines. The initial value for the inverse map $\Phi$ is then
given by $\Upsilon_0$. Once we solve for $\phi$, we can reconstruct
$\D*\Phi$ and integrate to get $\Phi$. 

The Cauchy problem satisfied by $\phi$, however, has its initial value
given by the value of $\D*\Phi$ along the initial slice, which depends
on the $2$-jet of the solution embedding $\Upsilon$ along $\{t = 0\}$.
That we can recover the $2$-jet pointwise from the $1$-jet is due to
the local wellposedness of the Cauchy problem as described in Remark
\ref{rmk:defcauchyproblem}, and is equivalent to the real-analytic 
local wellposedness of the Cauchy problem for $\Upsilon$ in the sense of
Cauchy-Kowalevski.

The point of view we prefer to take, however, is that issue of local
well-posedness of the \cpmc problem is already solved (see the
discussion surrounding Remark \ref{rmk:defcauchyproblem}). The system 
\eqref{eq:mainsystem} is an associated system of PDEs that allows us
to more easily derive good \emph{a priori} estimates which, for
a large class of suitable data, leads to global-in-time existence and
good controls on the asymptotics. 
\end{rmk}

\subsection{Spherical symmetry revisited}\label{sec:sphsymrevisit}
Before we launch into the study of the full system
\eqref{eq:mainsystem}, however, let us first re-investigate the 
spherically symmetric case, which
we previously treated in Section \ref{sec:rotsym}, now using the
language introduced above. 

Let $\tau^a$ denote the unit time-like vector field orthogonal to the
constant $t$ slices of $\dS$. Under spherical symmetry, it is easy to
see that the $\ooo{g}$-self-adjoint map $\phi_a^c$ is determined by 
two scalar functions $\zeta$ and $\eta$ as
\begin{equation}
\phi_a^c = \eta \tau_a\tau^c + \zeta \delta_a^c;
\end{equation}
furthermore, $\eta$ and $\zeta$ are constant on the constant $t$
slices of $\dS$. Using \eqref{eq:defpsi} and \eqref{eq:inversedecomp}
we have that the constant mean curvature condition is equivalent to 
\begin{equation}\label{eq:etazeta} \frac{d}{1+\zeta} + \frac{1}{1+\zeta - \eta} = d+1 \iff \eta =
\frac{\zeta (1+\zeta)}{\zeta + \frac{1}{d+1}}.\end{equation}
This implies that 
\begin{equation}\label{eq:expressionA1}
 A^{c}_a = (1 + \zeta)\left[\left(\delta^c_a + \tau_a\tau^c\right) -
\frac{1}{(d+1)\zeta + 1} \tau_a\tau^c\right] \end{equation}
and
\begin{equation}\label{eq:expressionA2} (A^{-1})^c_a = \frac{1}{1+\zeta} \left[ (\delta^c_a + \tau_a\tau^c)
- ((d+1)\zeta + 1)\tau_a\tau^c\right] \end{equation}
so we have that $A^c_a$ is positive definite provided $\zeta > -
\frac{1}{d+1}$. We now verify that in this regime, the bracketed terms
in \eqref{eq:coda3} have no non-trivial kernel. 
\begin{prop}\label{prop:coeffinvert}
In the spherically symmetric case, with $A^c_a$ as in
\eqref{eq:expressionA1} and with $\zeta > - \frac{1}{d+1}$, the
equations \eqref{eq:coda3} and \eqref{eq:codafinal} are equivalent. 
\end{prop}
\begin{proof}
The proof is by direction computation using some elementary
(multi)linear algebra, and we sketch the computations here. We
consider here the action of the bracketed term in \eqref{eq:coda3}, 
which for the purpose below we denote by $-B^{ijk}_{abc}$, on
elements of $T^{0,3}\dS$ with the \emph{same symmetry type as
$\ooo\covD_{[i}\phi_{j]k}$}, while recalling that $\phi$ is symmetric
in its indices. In particular, letting $f^{(0)},\dots,f^{(d)}$ be an
orthonormal basis of $T^*\dS$ with $f^{(0)} = \tau$, we see that 
$\ooo\covD_{[i}\phi_{j]k}$ can be decomposed in terms of tensors of the
form
\[ F^{(\alpha\beta\gamma)}_{ijk} = f^{(\alpha)}_i f^{(\beta)}_j f^{(\gamma)}_k + f^{(\alpha)}_i
f^{(\gamma)}_j f^{(\beta)}_k  - f^{(\beta)}_i f^{(\alpha)}_j
f^{(\gamma)}_k  - f^{(\gamma)}_i f^{(\beta)}_j f^{(\alpha)}_k .\]
Observe that $F^{(\alpha\beta\gamma)}_{ijk} =
F^{(\alpha\gamma\beta)}_{ijk}$. 
For ease of notation, we write $\nu = (d+1) \zeta + 1$ (which is
positive by assumption), then
$B^{ijk}_{abc}$ can be re-written as
\begin{align*}
B^{ijk}_{abc} &= \delta^i_a \delta^j_b\delta^k_c -
\delta^i_b\delta^j_a\delta^k_c \\
& \quad - (\nu^{-1} - 1) \left[ \tau^i\tau_a \delta^j_b \delta^k_c -
\tau^i\tau_b \delta^j_a \delta^k_c + \tau^i\tau_a \delta^j_c\delta^k_b
- \tau^i\tau_b \delta^j_c\delta^k_a \right] \\
& \quad - (\nu - 1) \left[ \delta^i_a\delta^j_b \tau^k\tau_c -
\delta^i_b \delta^j_a \tau^k\tau)c + \delta^i_a \delta^j_c
\tau^k\tau_b - \delta^i_b \delta^j_c \tau^k\tau_a\right] \\
& \quad + (\nu - 1)(\nu^{-1} - 1) \left[ \tau^i\tau_a \delta^j_b
\tau^k\tau_c - \tau^i\tau_b\delta^j_a \tau^k\tau_c + \tau^i\tau_a
\delta^j_c \tau^k\tau_b - \tau^i\tau_b \delta^j_c \tau^k\tau_a \right]
\end{align*}
by way of \eqref{eq:expressionA1} and \eqref{eq:expressionA2}.

We consider three cases, depending on how many of
$\alpha,\beta,\gamma$ is $0$. 
\begin{enumerate}
\item $\alpha,\beta,\gamma \in \{1,\dots,d\}$: then we see that
\[ B^{ijk}_{abc} F^{(\alpha\beta\gamma)}_{ijk} = 2
F^{(\alpha\beta\gamma)}_{abc}.\] 
\item $\alpha = 0$, $\beta,\gamma \in \{1,\dots,d\}$: then we see that
(recalling that $\tau_a\tau^a = -1$) 
\[ B^{ijk}_{abc} F^{(0\beta\gamma)}_{ijk} = (1 + \nu^{-1})
F^{(0\beta\gamma)}_{abc}\]
and
\[ B^{ijk}_{abc} F^{(\beta 0\gamma)}_{ijk} = (1 + \nu) F^{(\beta 0
\gamma)}_{abc} - (\nu^{-1} - 1) F^{(0\beta\gamma)}_{abc} - (\nu - 1)
F^{(\gamma 0 \beta)}_{abc}.\]
\item $\beta = \gamma = 0$, $\alpha \in \{1,\dots,d\}$: note first
that we have $F^{(\alpha 0 0 )}_{ijk} = -2 F^{(0\alpha 0)}_{ijk}$, 
then we see that
\[ B^{ijk}_{abc} F^{(\alpha 0 0 )}_{ijk} = 2 \nu F^{(\alpha 0 0
)}_{ijk}.\]
\end{enumerate}
Thus we see that expressed in terms of the
$F^{(\alpha\beta\gamma)}_{ijk}$ tensors the operator $B^{ijk}_{abc}$
is almost diagonalised, and its invertibility clearly follows when
$\nu \neq 0 $. 
\end{proof}

The Codazzi equation \eqref{eq:codafinal} then gives that
\[ (\ooo\covD_a \eta)\tau_b \tau^c - (\ooo\covD_b\eta)\tau_a\tau^c +
(\ooo\covD_a\zeta)\delta_b^c - (\ooo\covD_b\zeta)\delta_a^c + \eta
[\tau_b\ooo\covD_a\tau^c - \tau_a\ooo\covD_b\tau^c] = 0.\]
To get an evolution equation we need to contract against $\tau^a$. If
we contract also against $\tau^b$ the expression vanishes by
anti-symmetry. So let $\sigma^a$ be a spatial unit vector and we have,
contracting against $\tau^a\sigma^b\sigma_c$ that
\begin{equation}\label{eq:zetatauder}
\tau(\zeta) = - \eta (\ooo{\covD}_b \tau^c)\sigma^b\sigma_c.
\end{equation}
Applying now Proposition \ref{prop:nablatau} for the value of
$\ooo{\covD}_b\tau^c$, we can write the equation of motion for $\xi$
relative to $t = x^0$ as
\begin{equation}\label{eq:igmrotsym}
\frac{\D*}{\D*{t}} \zeta = - \frac{t}{1+t^2} \frac{1 +
\zeta}{\frac{1}{d+1}+\zeta} \zeta.
\end{equation}
The fixed point $\zeta \equiv 0$ for \eqref{eq:igmrotsym} is an attractor
for all initial data $\zeta(t_0) > - \frac{1}{d+1}$, where $t_0 > 0$. 
Using that in this regime 
\[ \frac{1 +\zeta}{\frac{1}{d+1}+\zeta} > 1\]
we have that 
\[ \abs{\zeta(t)} \lesssim \frac{1}{\sqrt{1+t^2}}\]
for $t > t_0 > 0$. (In fact, in spherical symmetry the decay is
stronger due to monotonicity: for negative initial data, this argument
shows that $\abs{\zeta(t)} \lesssim \jb{t}^{-(d+1)}$. While for
positive initial data this argument shows that $\abs{\zeta(t)}
\lesssim_\epsilon \jb{t}^{-(d+1)+\epsilon}$ for every $\epsilon > 0$,
where $\lesssim_\epsilon$ indicates that the implicit constant of
proportionality depends on the choice of $\epsilon$.) The
algebraic relation \eqref{eq:etazeta} then shows that $\eta$ must also
decay at the same rate. So we have in fact shown that

\begin{thm}\label{thm:sphsymstableigm}
In the inverse-Gauss-map gauge, any future in time, spherically
symmetric solution generated by initial data prescribed at $t_0 >
0$ with $\zeta(t_0) > - \frac{1}{d+1}$ is stable under small
perturbations. 
\end{thm}

Theorem \ref{thm:sphsymstableigm} should be compared with Theorem
\ref{thm:expansionstability}: the main difference is that in the
inverse-Gauss-map gauge, it is no longer necessary to perform the
modulation by allowing for convergence to the family of 
time-translations. This is of course related to the fact that through
the inverse-Gauss-map gauge, the equation of motion
\eqref{eq:igmrotsym} contains only terms at the ``derivative level''
of the map $\Phi = G^{-1}$, and not on $\Phi$ itself (which sees the
instability associated to the Poincar\'e symmetries of the ambient
$\Real^{1,d+1}$). In the remainder of the paper, using the
inverse-Gauss-map gauge as a non-linear, localised replacement for
modulations, we will extend Theorem \ref{thm:sphsymstableigm} to the
general case without symmetry assumptions. 

\begin{rmk}\label{rmk:invertible}
The boundary $-\frac{1}{d+1}$ has a natural interpretation. Observe
that when $\zeta = - \frac{1}{d+1}$, the trace-free condition
\eqref{eq:etazeta} requires
that $\eta = - \infty$, which implies that the Gauss map is \emph{not
invertible}. Going back to our classification of spherically symmetric
solutions in Section \ref{sec:classification}, we see that of the
solutions for which $f'$ vanishes somewhere, for both the expanding
solutions and the big bang and big crunch scenarios the Gauss map
give diffeomorphisms to $\dS$. In the cases where $f'$ is signed, when
$f$ converges to the static cylinder in either future or past, the
Gauss map gives a diffeomorphism to ``half'' of $\dS$ (either the
future half of past half). In the remaining cases $f \sqrt{1 -
(f')^2}$ can be seen to equal $\frac{d}{d+1}$ somewhere, at which
point $f'' = 0$. It is easy to see that this is equivalent to there
being a critical point for the Gauss map, and to $\zeta = -
\frac{1}{d+1}$. 

Aside from the case of the static cylinder solution, for spherically
symmetric solutions to the \cpmc problem, there is at most one time
$t_0$ at which the Gauss map is critical. By placing our initial data
strictly to the future of this time, it thus makes sense to ask about
the \emph{future} stability of any \emph{future-expanding} spherically 
symmetric solution.
\end{rmk}

We record here some further computations regarding these spherically
symmetric solutions. We rewrite \eqref{eq:zetatauder} as
\begin{equation}
\ooo\covD_a \zeta = \frac{\eta}{d} (\ooo\covD_f\tau^f) \tau_a
\end{equation}
and note that from Proposition \ref{prop:nablatau} that 
\[ \ooo\covD_b\tau^c = \frac{1}{d} \ooo\covD_f\tau^f (\delta_b^c +
\tau_b\tau^c).\]
Next, from \eqref{eq:etazeta} we get that 
\[ \ooo\covD_a \eta = \eta \left( \frac{1}{\zeta} + \frac{1}{1+\zeta} -
\frac{1}{\zeta + \frac{1}{d+1}}\right) \ooo\covD_a \zeta\]
and we observe that
\begin{multline*}
\eta \left( \frac{1}{\zeta} + \frac{1}{1+\zeta} -
\frac{1}{\zeta + \frac{1}{d+1}}\right) = \frac{1}{(\zeta +
\frac{1}{d+1})^2}\left[ (\zeta + 1)(\zeta + \frac{1}{d+1}) +
\zeta(\zeta + \frac{1}{d+1}) - \zeta(\zeta + 1) \right]\\
= 1 + \frac{1}{(\zeta +
\frac{1}{d+1})^2}\left[ \frac{\zeta + 1}{d+1} +
\zeta(\zeta + \frac{1}{d+1})  - (\zeta +
\frac{1}{d+1})^2\right] 
= 1 + \frac{d}{[(d+1)\zeta + 1]^2}.
\end{multline*}
Putting it altogether we get that for $\phi_a^c = \eta \tau_a\tau^c +
\zeta\delta_a^c$ we have
\begin{align*}
\ooo\covD_a\phi_b^c &= \ooo\covD_a\eta \tau_b\tau^c + \eta \ooo\covD_a
\tau_b \tau^c + \eta \tau_b \ooo\covD_a \tau^c + \ooo\covD_a\zeta
\delta_b^c \\
& = \frac{1}{d}\eta (\ooo\covD_f \tau^f) [ (\ooo{g}_{ab} +
\tau_a\tau_b)\tau^c + \tau_b(\delta_a^c + \tau_a\tau^c) ] \\
& \quad + (\delta_b^c + \tau_b\tau^c)\ooo\covD_a \zeta +
\frac{d}{[(1+d)\zeta + 1]^2}\tau_b\tau^c \ooo\covD_a \zeta\\
& = \frac{1}{d}\eta (\ooo\covD_f \tau^f) [ (\ooo{g}_{ab} +
\tau_a\tau_b)\tau^c + \tau_b(\delta_a^c + \tau_a\tau^c) +
\tau_a(\delta_b^c + \tau_b\tau^c)] \\
& \quad +
\frac{\eta}{[(1+d)\zeta + 1]^2} (\ooo\covD_f\tau^f) \tau_b\tau^c
\tau_a
\end{align*}
which we write conveniently as
\begin{equation}
\massterm_{abc} \eqdef \eta(\ooo\covD_f\tau^f)\left[ \frac{3}{d}
\ooo{g}_{(ab}\tau_{c)} + \frac{3}{d} \tau_a\tau_b\tau_c +
\frac{1}{[(1+d)\zeta + 1]^2} \tau_a\tau_b\tau_c\right].
\end{equation}
Observe from Proposition \ref{prop:nablatau} that the divergence
$\ooo\covD_f\tau^f$ is an order 1 quantity, and so we see that using
the decay of $\eta$ implied by the decay of $\zeta$ established above,
we have that $\massterm_{abc}$ decays as roughly $\jb{x^0}^{-(d+1)}$.
Note also that $\massterm$ is spherically symmetric by definition. 

Now let us re-write \eqref{eq:mainsystem} as a perturbation around one
of these spherical symmetric solutions. More precisely, we will
replace 
\begin{equation}
\begin{split}
\phi_a^b &\to \breve{\phi}_a^b + \hat{\phi}_a^b = \eta \tau_a\tau^b +
\zeta\delta_a^b + \hat\phi_a^b,\\
\delta_a^b + \psi_a^b &\to \delta_a^b + \breve{\psi}_a^b +\hat\psi_a^b =  - \frac{1}{1 + \zeta - \eta} \tau_a\tau^b
+ \frac{1}{1+\zeta}(\delta_a^b + \tau_a\tau^b) + \hat\psi_a^b,
\end{split}
\end{equation}
where $\breve\phi$ and $\breve\psi$ are the values corresponding to a
fixed background spherically symmetric solution. 
We let also 
\begin{equation}
\begin{split}
g^{ab} & = (\delta^a_c + \breve{\psi}^a_c + \hat\psi^a_c) \ooo{g}^{cd}
(\delta^b_d + \breve{\psi}^b_d + \hat\psi^b_d), \\
\breve{g}^{ab} & = (\delta^a_c+ \breve{\psi}^a_c) \ooo{g}^{cd}
(\delta^b_d + \breve{\psi}^b_d).
\end{split}
\end{equation}
(By definition we have that $\breve{g}^{ab}\ooo\covD_a\breve{\phi}_{bc}
= 0$.) Then \eqref{eq:mainsystem} becomes
\begin{align*}
\ooo\covD_{[a}\hat\phi_{b]c} & = 0,\\
g^{ab} \ooo\covD_a \hat\phi_{bc} &= -(g^{ab} - \breve{g}^{ab}) \ooo\covD_a
\breve{\phi}_{bc},
\end{align*}
which we simplify as
\begin{equation}\label{eq:largedatasystem}
\begin{split}
\ooo\covD_{[a}\hat\phi_{b]c} & = 0,\\
g^{ab} \ooo\covD_a \hat\phi_{bc} + 2 (\delta^a_e +
\breve{\psi}^a_e)\hat\psi^{be} \massterm_{abc} &= \massterm_{abc} \hat\psi^b_e
\hat\psi^{ae}.
\end{split}
\end{equation}

\tocaddline
\section{Stress-energy tensor}
\label{sec:strtensor}
A by-now standard method of obtaining \emph{a priori} estimates for
wave-like equations is through $L^2$-based \emph{energy inequalities}.
For \emph{second order} partial differential equations arising from
a Lagrangian formulation, a systematic treatment of energy
inequalities based on the construction of an associated \emph{canonical
stress-energy tensor} has been considered in \cite{Christ2000} (see
also \cite[\S 4.2]{Wong2011}). Our system\footnote{We will, for the
remainder of the paper, consider only \emph{small perturbations} of
expanding spherically symmetric solutions. From the discussion in the
previous section, in particular Proposition \ref{prop:coeffinvert},
we see in this regime the equations of motion are equivalent to
the divergence-curl system \eqref{eq:mainsystem}.} \eqref{eq:mainsystem}
however is \emph{first order} and, in our formulation, is not
obviously the Euler-Lagrange equations of a variational functional.
Yet as we shall see below, we can nevertheless write down a 
stress-energy tensor with suitable properties for deriving energy
estimates. 

For the linear system 
\begin{subequations}
\begin{align}
\ooo\covD_a \phi_b^c - \ooo\covD_b\phi_a^c &= 0,\\
\ooo\covD_a\phi_b^a &= 0,
\end{align}
\end{subequations}
where $\phi$ is $\ooo{g}$-self-adjoint, S. Brendle obtained
\cite{Brendl2002} a Bel-Robinson type energy tensor\footnote{Though
interestingly, Brendle in fact did \emph{not} use his Bel-Robinson
tensor in his stability proof. To the best of the author's knowledge,
the present paper is the first time this construction is applied to
obtain concrete energy estimates.}. That such a
tensor is available is not so unexpected:
the relationship between the tensor Brendle wrote down, and the
stress-energy tensor for the linear scalar field (which we interpret
as a div-curl relation for a 1-form), is identical to the
relationship between the classical Bel-Robinson tensor and the
stress-energy tensor for the linear Maxwell field.\footnote{To expand
at little bit: for a solution $u$ to the linear wave equation
$\Box_{\ooo{g}} u = 0$, the one-form $\D*{u}$ satisfies a linear 
divergence-curl
system, and thus the linear scalar field can be viewed as (up to
topological obstructions) identical to a $\ooo{g}$-harmonic one-form field.
Similarly, the linear Maxwell field is a $\ooo{g}$-harmonic two-form field.
The $\phi$ for the linearised equation is a $\ooo{g}$-self-adjoint
mapping on the space of one-form fields, and satisfies a
divergence-curl relation similar to that of the one-form field (in
terms of the number of indices involved). The Weyl-like tensors are
$\ooo{g}$-self-adjoint mappings on the space of two-form fields, and
satisfy divergence-curl relations similar to that of the two-form
field (again in terms of the number of indices involved). Thus at the
algebraic level the relationship between Maxwell and Weyl fields are
very similar to that between one-form fields and our tensor $\phi$.}
To study the system
\eqref{eq:mainsystem} which is quasilinear in $\phi$, one cannot
simply treat the nonlinearities as a ``source'' term, for that will
introduce a loss of derivatives in the corresponding energy estimates.
Instead, we need to consider the ``variable coefficient''
analogue of Brendle's Bel-Robinson tensor. This we can do somewhat
systematically in view of the well-developed theory for stress-energy
tensors of variable coefficient wave equations, and our analogy
comparing relationship between scalar fields and $\phi$ and the
relationship between Maxwell and Weyl fields. 

To set notations, we consider a linear system of equations
\begin{subequations}
\begin{align}
\label{eq:gensyscurl} \ooo\covD_a \phi_{bc} - \ooo\covD_b \phi_{ac} & =
0,\\
\label{eq:gensysdiv} g^{ab} \ooo\covD_a \phi_{bc} &= F_c.
\end{align}
\end{subequations}
The connection $\ooo\covD$ is the Levi-Civita connection for some
background metric $\ooo{g}$, while the coefficients $g^{ab}$ are
\emph{not} $\ooo\covD$-parallel (hence ``variable''), but for now
should be considered ``frozen'' (and not quasilinear). In view of
\eqref{eq:gensyscurl}, we can assume, without loss of generality, that
$g^{ab}$ is symmetric. The term $F_c$
is some fixed source term. We also assume that $\phi_{bc}$ is
symmetric in its indices. The lowering and raising of indices in
computations below will be with the background metric $\ooo{g}$.

By combining equation (16) of \cite{Wong2011} and the computations in
\cite[\S 3]{Brendl2002}, we will first define the tensor $Z^{mn}|^a_b$
by
\[ Z^{mn}|^a_b = g^{mn} \delta^a_b - g^{an}\delta^m_b -
g^{ma}\delta^n_b. \]
Then we define $\strtensor^{ab}_{cd}$ by
\begin{equation}\label{eq:strtendef1}
\strtensor^{ab}_{cd} = \phi_{mo}\phi_{np} Z^{mn}|^a_c Z^{op}|^b_d
\end{equation}
which we expand fully as
\begin{multline}\label{eq:strtendef2}
\strtensor^{ab}_{cd} = \phi_{mo}\phi_{np}g^{mn}g^{op} \delta^a_c\delta^b_d - 2
\phi_{md}\phi_{np} g^{mn} g^{bp} \delta^a_c \\
 - 2 \phi_{co}\phi_{np} g^{op} g^{an}\delta^b_d + 2
   (\phi_{cd} \phi_{np}+ \phi_{cp}\phi_{nd}) g^{na} g^{pb}.
\end{multline}
Observe that from the definition $\strtensor^{ab}_{cd} = \strtensor^{ba}_{dc}$. We
compute the divergence $\ooo\covD_a \strtensor^{ab}_{cd}$ (the divergence
$\ooo\covD_b \strtensor^{ab}_{cd}$ can be obtained by symmetry)
\begin{align*}
\ooo\covD_a \strtensor^{ab}_{cd} &= \delta^b_d \ooo\covD_c(\phi_{mo}\phi_{np}
g^{mn} g^{op}) - 2 \ooo\covD_c(\phi_{md}\phi_{np} g^{mn} g^{bp}) \\
& \quad - 2 \delta^b_d \ooo\covD_a(\phi_{co}\phi_{np} g^{op}g^{an}) + 
2 \ooo\covD_a[ (\phi_{cd} \phi_{np}+ \phi_{cp}\phi_{nd}) g^{na} g^{pb}
] \\
&= \delta^b_d \phi_{mo}\phi_{np} \ooo\covD_c(g^{mn}g^{op}) 
- 2 \phi_{md}\phi_{np}\ooo\covD_c(g^{mn} g^{bp}) \\
& \quad - 2 \delta^b_d \phi_{co} \ooo\covD_a(\phi_{np} g^{op}g^{an}) + 
2 \phi_{cd} \ooo\covD_a(\phi_{np} g^{na} g^{pb}) + 2 \phi_{cp}
\ooo\covD_a (\phi_{nd} g^{na} g^{pb})
\end{align*}
where in the second equality we used \eqref{eq:gensyscurl}. Applying
next \eqref{eq:gensysdiv} we get
\begin{multline}\label{eq:keydivergenceprop}
\ooo\covD_a \strtensor^{ab}_{cd} = \delta^b_d \phi_{mo}\phi_{np} \ooo\covD_c(g^{mn}g^{op}) 
- 2 \phi_{md}\phi_{np}\ooo\covD_c(g^{mn} g^{bp}) 
- 2 \delta^b_d \phi_{co} \phi_{np} \ooo\covD_a(g^{op}g^{an})\\
- 2 \delta^b_d \phi_{co} g^{op} F_p
+ 2 (\phi_{cd}\phi_{np}+ \phi_{cp}\phi_{nd}) \ooo\covD_a(g^{na} g^{pb}) 
+ 2 \phi_{cd} g^{pb}F_p + 2 \phi_{cp} g^{pb}F_d.
\end{multline}
The key feature of \eqref{eq:keydivergenceprop} to note is that the 
$\ooo\covD_a \strtensor^{ab}_{cd}$ does not depend on
derivatives of the solution $\phi$, and hence we can use this to write
down an energy identity \emph{without derivative loss}. 

\begin{rmk}
Observe that in the case $\ooo{\covD}_c g^{ab} = 0$ (which is
satisfied in the ``constant coefficient'' case $g^{ab} =
\ooo{g}^{ab}$) and $F_c = 0$, we have that $\strtensor^{ab}_{cd}$ is divergence
free: this is the case of Brendle's Bel-Robinson tensor. One can check
that the tensor
$Q(h)$ given in \cite[\S 3]{Brendl2002}, under the above assumptions
of homogeneity and constant coefficient, can be expressed as
\[ Q_{abcd} = \frac{3}{4} \ooo{g}_{e(a}\strtensor^{ef}_{cd}\ooo{g}_{b)f} \]
where the parentheses in the indices denote full symmetrisation in 
$a,b,c,d$. We will not forcibly symmetrise in the indices, as that
property is not necessary for the derivation of energy identities, and
the natural form of $\strtensor$ is as a $(2,2)$-tensor (in analogy
with the canonical stress tensor which has type $(1,1)$).  
\end{rmk}

For $\strtensor$ to be useful in energy estimates, it also needs to
satisfy good coercivity properties. We state the pointwise inequality 
in the following ``perturbative'' form. 

\begin{prop}\label{prop:fundposstress}
Let $(f^{(i)})_{i \in \{0,\dots,d\}}$ be an orthonormal co-frame, and
$(e_{(i)})$ its dual frame, for the background metric $\ooo{g}$, which 
we assume to be Lorentzian. Denote also $\tau^a = e_{(0)}^a$, and
hence $\tau_a = - f^{(0)}_a$. Suppose there exists $A,B>0$ such that 
\[ \abs{g^{ab}f^{(0)}_a f^{(0)}_b + A} < \frac{\min(A,B)}{4(d+1)} \]
and for every $i \in \{1,\dots,d\}$,
\[ \abs{g^{ab}f^{(i)}_a f^{(i)}_b - B} < \frac{\min(A,B)}{4(d+1)}, \]
and for every $\mu,\nu \in \{0,\dots,d\}$ such that $\mu \neq \nu$,
\[ \abs{g^{ab}f^{(\mu)}_a f^{(\nu)}_b} < \frac{\min(A,B)}{4(d+1)}, \]
then we have
\begin{equation}\label{eq:fundposstress}
\strtensor^{ab}_{cd} \tau^c\tau^d \tau_a\tau_b \geq
\frac{\min(A,B)^2}{2}
\sum_{i,j = 0}^d \abs{\phi_{(i)(j)}}^2 
\end{equation}
where
\[ \phi_{(i)(j)} = \phi_{ab} (e_{(i)})^a (e_{(j)})^b. \]
\end{prop}
\begin{proof}
Write $\breve{g}^{mn} = - A e_{(0)}^m e_{(0)}^n + \sum_{i = 1}^d B
e_{(i)}^m e_{(i)}^n$. We consider
\[ Z^{mn}|^a_c \tau^c\tau_a = - \breve{g}^{mn}  - 2
\breve{g}^{an}\tau_a \tau^m +
\tilde{Z}^{mn} \]
where 
\[ \tilde{Z}^{mn} = - (g - \breve{g})^{mn} - (g - \breve{g})^{ma}\tau^n
\tau_a- (g-\breve{g})^{na}\tau^m \tau_a.\]
Note that 
\[ \breve{g}^{mn} + 2 \breve{g}^{an}\tau_a \tau^m = A e_{(0)}^m
e_{(0)}^n + \sum_{i = 1}^d B
e_{(i)}^m e_{(i)}^n.\]
By our assumption 
\[ \abs{\tilde{Z}^{mn} (f^{(i)})_m (f^{(j)})_n } < \frac{\min(A,B)}{4(d+1)}. \]
Then our desired inequality follows from the definition
\eqref{eq:strtendef1} and Cauchy's inequality. 
\end{proof}

As we shall see in the next section, to derive the fundamental energy
estimate through the divergence theorem, we will consider the divergence 
\begin{equation} \ooo\covD_a(\strtensor^{ab}_{cd}\tau^c\tau^d\tau_b) 
= \ooo\covD_a(\strtensor^{ab}_{cd})\tau^c\tau^d\tau_b +
\strtensor^{ab}_{cd}\ooo\nabla_a(\tau^c\tau^d\tau_b)
\end{equation}
where the multiplier $\tau$ is the unit future time-like normal orthogonal to the
constant $x^0$ slices of $\dS$. Therefore in addition to the
divergence of the stress-energy tensor, we also need to consider its
contractions with tensors of the form $\ooo\covD_a\tau^c$, for which
we have an explicit expression in Proposition \ref{prop:nablatau}. For
convenience we record the relevant computations here. 

By Proposition \ref{prop:nablatau}, what we need to compute is (up to
a scalar weight in $x^0$)
\[
\strtensor^{ab}_{cd}\left[ (\delta^c_a + \tau^c\tau_a)\tau^d\tau_b +
(\delta^d_a + \tau^d\tau_a)\tau^c\tau_b + (\ooo{g}_{ab} +
\tau_a\tau_b)\tau^c\tau^d\right],
\]
and we do so term by term; the second and third terms inside brackets
are essentially identical due to the symmetry properties of
$\strtensor$. As we have done in the proof to Proposition
\ref{prop:fundposstress}, we proceed by first computing the
coefficients given by the tensor $Z$ appearing in
\eqref{eq:strtendef1}. 

First we see easily that
\[ Z^{mn}|^a_c \delta^c_a = (d-1) g^{mn}, \]
while
\[ Z^{mn}|^a_c \tau_a\tau^c = - g^{mn} - g^{am}\tau_a\tau^n -
g^{an}\tau_a\tau^m, \]
so we can write
\begin{equation}
Z^{mn}|^a_c \delta^c_a = (1-d) Z^{mn}|^a_c \tau_a\tau^c + (1-d)
(g^{am}\tau_a\tau^n + g^{an}\tau_a\tau^m).
\end{equation}

In the sequel, for the main decay estimate we will be working under
the assumption where $g -\ooo{g}$ is a small error term. So defining
\[
\ooo{Z}^{mn}|^a_b  = \ooo{g}^{mn}\delta^a_b - \ooo{g}^{an}\delta^m_b
- \ooo{g}^{am}\delta^n_b,\]
we have
\begin{equation}\label{eq:Zoncetau}
\begin{split}
\ooo{Z}^{mn}|^a_b \tau^b &= \ooo{g}^{mn}\tau^a - \ooo{g}^{an}\tau^m -
\ooo{g}^{am}\tau^n \\
\ooo{Z}^{mn}|^a_b \tau_a &= \ooo{g}^{mn}\tau_b - \tau^n \delta^m_b -
\tau^m \delta^n_b
\end{split}
\end{equation}
which implies that
\begin{multline*}
\ooo{Z}^{mn}|^a_c \ooo{Z}^{op}|^b_d \delta^d_a \tau_b\tau^c =
\ooo{Z}^{mn}|^a_c \ooo{Z}^{op}|^b_d \ooo{g}_{ab}\tau^c\tau^d = \\
- \ooo{g}^{mn}\ooo{g}^{op} - 2\ooo{g}^{mn}\tau^o\tau^p - 2
  \ooo{g}^{op}\tau^m\tau^n + \ooo{g}^{mo}\tau^n\tau^p +
\ooo{g}^{mp}\tau^n\tau^o + \ooo{g}^{no} \tau^m\tau^p +
\ooo{g}^{np}\tau^m\tau^o
\end{multline*}
which we simplify to 
\begin{multline*}
\ooo{Z}^{mn}|^a_c \ooo{Z}^{op}|^b_d \delta^d_a \tau_b\tau^c =
\ooo{Z}^{mn}|^a_c \ooo{Z}^{op}|^b_d \ooo{g}_{ab}\tau^c\tau^d = 
- (\ooo{g}^{mn}+ 2 \tau^m\tau^n)(\ooo{g}^{op}  + 2\tau^o\tau^p) 
\\ + (\ooo{g}^{mo}+ \tau^m\tau^o)\tau^n\tau^p  +
(\ooo{g}^{mp}+ \tau^m\tau^p)\tau^n\tau^o + (\ooo{g}^{no}+\tau^n\tau^o) \tau^m\tau^p +
(\ooo{g}^{np}+\tau^n\tau^p) \tau^m\tau^o.
\end{multline*}
From these we conclude that, writing only schematically (and hence
dropping factors that are fixed at ``size one'') terms that are at least linear
in the error $g - \ooo{g}$, and re-inserting the scalar weight in
$x^0$,
\begin{multline}\label{eq:strwdeform}
\frac{\jb{x^0}}{x^0} \strtensor^{ab}_{cd} \nabla_a(\tau_b\tau^c\tau^d) + (d-2)
\strtensor^{ab}_{cd} \tau_a\tau_b\tau^c\tau^d \\
\approx
 4\phi_{mo}\phi_{np} \left[ (\ooo{g}^{mp} +
\tau^m\tau^p)\tau^n\tau^o + (\ooo{g}^{mo} + \tau^m\tau^o)
\tau^n\tau^p\right] \\
+ 2(d-1) \phi_{mo}\phi_{np}  \tau^m\tau^n(\ooo{g}^{op} + 2 \tau^o\tau^p)
+O(\phi\cdot\phi \cdot g\cdot [g - \ooo{g}]),
\end{multline}
away from where $x^0 = 0$.
There are two troublesome terms in the expression above: when $d > 2$
the term $(d-2) \strtensor^{ab}_{cd} \tau_a\tau_b\tau^c\tau^d$ gives a
\emph{negative} contribution to the divergence; and 
the term $(\trace_{\ooo{g}} \phi) \phi_{np}\tau^n\tau^p$ appearing
in the third line which can, in principle, have indeterminate sign.
Without those two terms, and up to errors of the form $g- \ooo{g}$, 
the remaining terms contribute positively and 
give us in principle a monotonicity formula for an $L^2$ energy
quantity (see Section \ref{sec:energyest} for more details). 

How do we deal with the terms with the bad sign? For the term with the
coefficient $d-2$, we make the following observation. Returning to our
study of the main system \eqref{eq:mainsystem} in spherical symmetry
at the end of the previous section, we showed that the expected
\emph{uniform} decay rate 
of $\phi$ in spatial $L^\infty$ to be $\jb{x^0}^{-1}$. At this rate of
decay, the square of the spatial $L^2$ norm will in fact \emph{grow} at a rate
$\jb{x^0}^{d-2}$, due to the volume of the constant $x^0$ spheres
being $\jb{x^0}^{d}$. So the $\jb{x^0}^{d-2}$ growth rate implied by
the $(d-2)$ term in \eqref{eq:strwdeform} is expected, and can (and
should) be
renormalised away: instead considering the time evolution of the
energy $\strtensor^{ab}_{cd}\tau^c\tau^d\tau_a\tau_b$, we consider the
evolution of the time-weighted energy $\jb{x_0}^{2-d}
\strtensor^{ab}_{cd}\tau^c\tau^d\tau_a\tau_b$. This energy will be
almost conserved, which then by Sobolev embedding will give us that
the $L^\infty$ size of $\phi$ will decay. 

The treatment of the trace term $\trace_{\ooo{g}} \phi$ uses our
assumption of constant mean curvature. Recall that a consequence for
$\phi$ being a solution \eqref{eq:traceid} (or equivalently
\eqref{eq:inversetracefree}) will hold, provided that it holds
initially. This tells us that $\trace_{\ooo{g}}\phi$ has hidden
cancellations and should be treated as a \emph{nonlinear} quantity,
with smallness controlled from the expected $L^\infty$ decay of
$\phi$. 

\tocaddline
\section{Interlude: notations}
We record here the notational conventions that will be in force for
the remainder of the paper. 

The notation $A\lesssim B$ indicates that there exists a universal
constant $C > 0$ such that $A \leq CB$; when $C$ is not universal but
depends on parameter $\alpha$ we write $A\lesssim_\alpha B$. By
$A \approx B$ we intend $A\lesssim B$ and $B\lesssim A$; similarly we
have versions with subscripts.

The Japanese bracket, we recall, is defined by $\jb{s} = \sqrt{1 +
s^2}$.  

We will always be working over (subsets of) the manifold $\dS$ as
defined in Appendix \ref{app:gaussmap}. We will use $t$ and $x^0$
interchangeably for the same coordinate function along $\dS$, and we
will use $\omega\in\Sphere^d$ to parametrise the spatial directions in
the obvious way. The spatial dimension $d$ is always assumed to be at
least $1$: in the $d = 0$ case the result that a space-time
modulational stability holds follow immediately from the
ODE analysis of Section \ref{sec:rotsym}. The background metric $\ooo{g}$ is the induced
Lorentzian metric on $\dS$ with the coordinate expression
\[ - \frac{1}{\jb{t}^2} \D{t}^2 + \jb{t}^2 \D{\omega}^2_{\Sphere^d}.\]
All index-raising and -lowering will be done with respect to
$\ooo{g}$.  We use $\ooo\covD$ for the Levi-Civita connection of 
$\ooo{g}$. 

The constant $t$ subsets of $\dS$ will be denoted $\spaceslice{t}$; and
the space-time region satisfying $t\in (t_1,t_2)$ will be denoted
$\bulkregion{t_1}{t_2} = \cup_{t\in(t_1,t_2)} \spaceslice{t}$. We will
only be interested in $t_2 > t_1 \geq 0$. The
vector field $\tau$ is the future directed unit normal to
$\spaceslice{t}$, and in the coordinate $(t,\omega)$ is given by
$\jb{t}\partial_t$. The vector fields $\rotvf{ij}$ are the rotation
vector fields defined in \eqref{eq:defrotvf}. We will denote by
$\allrot$ the collection of all $\rotvf{ij}$, $i,j\in
\{1,\dots,d+1\}$. 

We write $\areasph$ for the standard area element of $\Sphere^d$. The
induced area element on $\spaceslice{t}$ is $\D*{\ooo{A}} = \jb{t}^d
\areasph$ and the space-time volume element of $\dS$ is $\D*{\ooo{V}} =
\jb{t}^{d-1} \D{t}\areasph$.  Observe that $\D*{\ooo{A}} =
\intprod_\tau \D*{\ooo{V}}$. For convenience we will introduce the
notation $\tweight$ 
\begin{equation}\label{eq:deftweight}
\tweight = \jb{t}^{2-d}.
\end{equation} 
Observe that we have $\tau(\tweight) = (2-d) \frac{t}{\jb{t}}
\tweight$ due to $\tau(t) = \jb{t}$. This implies that
\begin{equation}
\ooo\covD_a \tweight = (d-2) \frac{t}{\jb{t}} \tweight \tau_a.
\end{equation}
(As one can see we will use $\tweight$ to normalise the $(2-d)$ term
appearing in \eqref{eq:strwdeform}.)
 
The tensor $\strtensor[g,\phi]^{ab}_{cd}$ is as defined in
\eqref{eq:strtendef2}, relative to the coefficients $g^{ab}$ and the
unknown symmetric two-tensor $\phi_{ab}$. Using this we define the
$L^2$-based \emph{weighted energy} as 
\begin{equation}\label{eq:defenergy}
\energy[g,\phi](t) \eqdef \int_{\spaceslice{t}} \tweight
\strtensor[g,\phi]^{ab}_{cd} \tau_a\tau_b\tau^c\tau^d \D{\ooo{A}} =
\jb{t}^{2} \int_{\spaceslice{t}} \strtensor[g,\phi]^{ab}_{cd}
\tau_a\tau_b\tau^c \tau^d \areasph.
\end{equation} 

We will also define some conventions for norms using the vector field
$\tau$. The bilinear form $\hat{g}_{ab} = \ooo{g}_{ab} + 2\tau_a\tau_b$ 
is positive definite, as is its counterpart with raised indices
$\hat{g}^{ab} = \ooo{g}^{ab} + 2\tau^a\tau^b$. For a tensor field 
$V^{a \dots b}_{c\dots d}$, we define its \emph{pointwise norm} by
\begin{equation}
\abs{V}_{\ooo{g},\tau}^2 \eqdef V^{a_1\dots b_1}_{c_1\dots d_1} V^{a_2\dots
b_2}_{c_2\dots d_2} \hat{g}_{a_1a_2}\cdots
\hat{g}_{b_1b_2} \hat{g}^{c_1c_2}\cdots\hat{g}^{d_1d_2}.
\end{equation}
In particular,
we can rewrite the conclusion of Proposition \ref{prop:fundposstress}
as
\[\label{eq:fundposstressp}
\tag{\ref*{eq:fundposstress}'} 
\strtensor[g,\phi]^{ab}_{cd} \tau_a\tau_b \tau^c\tau^d \geq \frac12
\min(A,B)^2 \abs{\phi}_{\ooo{g},\tau}^2.
\]
We can analogously define the $L^p$ norms of the tensor $V$ over 
$\spaceslice{t}$ and $\bulkregion{t_1}{t_2}$ by considering the 
$L^p$ norms of the scalar $\abs{V}_{\ooo{g},\tau}$ with respect to the area
and volume measures $\D*{\ooo{A}}$ and $\D*{\ooo{V}}$. To simplify
notations we will write
\begin{align}
\norm[2]{V(t)}^2 &\eqdef \int_{\spaceslice{t}}
\abs{V}_{\ooo{g},\tau}^2
\D{\ooo{A}}, \\
\norm[\infty]{V(t)} &\eqdef \operatorname{ess\,sup}_{\spaceslice{t}}
\abs{V}_{\ooo{g},\tau}.
\end{align}

We also introduce the following notations for higher
order norms and energies. First for the Levi-Civita connection of $\ooo{g}$ we
write
\begin{equation}
\abs{\jb{\ooo\covD}^k V}_{\ooo{g},\tau} \eqdef \abs{V}_{\ooo{g},\tau} +
\abs{\ooo\covD V}_{\ooo{g},\tau} + \cdots + \abs{\ooo\covD^{k}
V}_{\ooo{g},\tau}.
\end{equation}
For a family $\mathfrak{F}$ of vector fields we write
\begin{equation}
\begin{split}
\abs{\mathfrak{F}^k V}_{\ooo{g},\tau} &\eqdef \sum_{\Omega\in
\mathfrak{F}^k} \abs{\lieD_{\Omega_1}\cdots\lieD_{\Omega_k}
V}_{\ooo{g},\tau} \\
\abs{\jb{\mathfrak{F}}^k V}_{\ooo{g},\tau} &\eqdef
\abs{V}_{\ooo{g},\tau} +
\sum_{j = 1}^{k} \abs{\mathfrak{F}^k V}_{\ooo{g},\tau}
\end{split}
\end{equation}
where $\lieD_X$ is the Lie derivative along the vector field $X$. 
The extension of this notation to the $L^p$ and $L^\infty$ cases are
clear. For the energy, we write
\begin{equation}
\begin{split}
\energy[g,\mathfrak{F}^k\phi](t) &\eqdef \sum_{\Omega\in
\mathfrak{F}^k} \energy[g, \lieD_{\Omega_1} \cdots\lieD_{\Omega_k}
\phi](t), \\
\energy[g,\jb{\mathfrak{F}}^k \phi](t) &\eqdef \energy[g,\phi](t) +
\sum_{j = 1}^k \energy[g,\mathfrak{F}^j\phi](t).
\end{split}
\end{equation}

\tocaddline
\section{Linear theory}
In this section we continue our study of the inhomogeneous, variable
coefficient linear system of equations \eqref{eq:gensyscurl} and
\eqref{eq:gensysdiv} which we reproduce here
\begin{align*}
\ooo\covD_{a}\phi_{bc} - \ooo\covD_{b} \phi_{ac} &= 0,
\tag{\ref*{eq:gensyscurl}}\\
g^{ab} \ooo\covD_{a} \phi_{bc} &= F_c . \tag{\ref*{eq:gensysdiv}}
\end{align*}
The main goal is to obtain estimates on $L^2$-based higher Sobolev
norms for the solution $\phi_{ab}$ which depends on properties of the
coefficients $g^{ab}$ and the source term $F_c$. 

\subsection{The fundamental energy estimate}\label{sec:energyest}
Suppose now that $\phi_{ab}$ is a symmetric two-tensor satisfying the
system of equations \eqref{eq:gensyscurl} and \eqref{eq:gensysdiv}, we
apply the divergence theorem to 
\[ \ooo\covD_a ( \tweight \strtensor[g,\phi]^{ab}_{cd}
\tau_b\tau^c\tau^d ) \]
over\footnote{While we fix our attention on domains of the form
$\bulkregion{t_1}{t_2}$, it will be clear from the argument that it
suffices that the ``top'' boundary of our region is of the form
$\spaceslice{t_2}$ for us to get good energy control. The choice of
``bottom'' boundary being $\spaceslice{t_1}$ is one of notational
convenience. It should be clear that the arguments given here can all
be carried through with the bottom boundary being any space-like slice
to the past of $\spaceslice{t_2}$, with minimum modifications.} 
the domain $\bulkregion{t_1}{t_2}$, with $t_2 > t_1 \geq 0$, and we get
\begin{multline*}
\energy[g,\phi](t_1) - \energy[g,\phi](t_2) =
\int_{\bulkregion{t_1}{t_2}} (d-2) \frac{t}{\jb{t}} \tweight
\strtensor[g,\phi]^{ab}_{cd} \tau_a\tau_b\tau^c\tau^d \\
+
\tweight (\ooo\covD_a\strtensor[g,\phi]^{ab}_{cd})\tau_b\tau^c\tau^d + 
\tweight \strtensor[g,\phi]^{ab}_{cd} \covD_a(\tau_b\tau^c\tau^d)
\D{\ooo{V}}.
\end{multline*}
Going back to \eqref{eq:keydivergenceprop}, we note that $\ooo\covD_a
g^{mn} = \ooo\covD_a(g^{mn} -\ooo{g}^{mn})$, so that we have
schematically 
\begin{equation}\label{eq:keydivschematic}
\ooo\covD_a\strtensor^{ab}_{cd} \approx 2 (\phi_{cd} F^b + \phi_c^b
F_d - \delta_d^b \phi_c^a F_a) + O(\phi\cdot\phi\cdot g \cdot
\ooo\covD[g - \ooo{g}]) + O(\phi\cdot F\cdot [g - \ooo{g}]).
\end{equation}
So using \eqref{eq:strwdeform} and moving the terms with the good sign
to the left hand side we have the schematic expression
\begin{multline*}
\energy[g,\phi](t_2) + \int_{\bulkregion{t_1}{t_2}}
2(d+1)\frac{t}{\jb{t}} \tweight
\abs{\phi\cdot\tau}_{\ooo{g},\tau}^2 \D{\ooo{V}}  - \energy[g,\phi](t_1) \\
\leq
\int_{\bulkregion{t_1}{t_2}} 2\tweight\abs{ (\phi_{cd} F^b + \phi_c^b
F_d - \delta_d^b \phi_c^a F_a)\tau_b\tau^c\tau^d}  + 4\tweight \phi_{ab}\tau^a\tau^b (\trace_{\ooo{g}} \phi)\\
+ O(\tweight\cdot \phi^2 \cdot g\cdot \jb{\ooo\covD}[g - \ooo{g}])
+ O(\tweight\cdot \phi\cdot F\cdot [g -
\ooo{g}])
\D{\ooo{V}}
\end{multline*}
where the implicit constant in the big-Oh notation is independent of
$\phi$, $g$, and $F$, but may depend on the dimension $d$. So we can
simply write
\begin{multline}\label{eq:basicenergy}
\energy[g,\phi](t_2) + 2(d+1) \int_{\bulkregion{t_1}{t_2}}
\frac{t}{\jb{t}}\tweight
\abs{\phi\cdot\tau}_{\ooo{g},\tau}^2 \D{\ooo{V}} - \energy[g,\phi](t_1) \\
\lesssim_{d}
\int_{t_1}^{t_2} \tweight \norm[2]{\phi(t)}^2
\norm[\infty]{g(t)}\norm[\infty]{\jb{\ooo\covD}[g - \ooo{g}](t)} 
\jb{t}^{-1} \D{t}\\
+ \int_{\bulkregion{t_1}{t_2}} \tweight
\abs{\phi F}_{\ooo{g},\tau} \left(1 + \abs{g -
\ooo{g}}_{\ooo{g},\tau}\right) 
+ \tweight \abs{\phi}_{\ooo{g},\tau}
\abs{\trace_{\ooo{g}} \phi} \D{\ooo{V}}
\end{multline}
where in the second line we used the decomposition $\D*{\ooo{V}} =
\jb{t}^{-1} \D{t}\D{\ooo{A}}$ using the notation described in the
previous section. 

\subsection{Commutators and higher order energies}
To obtain higher order derivative control (so we can eventually use
Sobolev's inequality to regain $L^\infty$ control from the $L^2$ based
energy quantities), we commute the equations with the rotational
symmetry vector fields $\rotvf{ij}$ of $\dS$. That $\rotvf{ij}$ is
Killing implies that it commutes with covariant derivatives as well as
$\ooo{g}$; a consequence being that index-raising and -lowering, and
tracing with respect to $\ooo{g}$ also
commute with the Lie derivation relative to $\rotvf{ij}$. (This allows
us to work with ordinary Lie derivatives of our unknown field $\phi$,
instead of \emph{modified} Lie derivatives such as those used in 
\cite{ChrKla1993, Christ2009}.) 
We furthermore observe that the commutator of two
rotational vector fields is given by a linear combination of other
rotational vector fields. 

We have from \eqref{eq:gensyscurl} and \eqref{eq:gensysdiv} that
\begin{gather*}
\ooo\covD_a \lieD_{\rotvf{ij}}\phi_{bc}
-\ooo\covD_b\lieD_{\rotvf{ij}}\phi_{ac} = 0,\\
g^{ab}\ooo\covD_a \lieD_{\rotvf{ij}} \phi_{bc} = \lieD_{\rotvf{ij}}
F_c - \lieD_{\rotvf{ij}} (g^{ab} - \ooo{g}^{ab}) \ooo\covD_a
\phi_{bc},
\end{gather*}
and by induction, we have the schematic equations
\begin{gather}
\ooo\covD_a (\lieD_{\Omega})^k \phi_{bc} -
\ooo\covD_b (\lieD_{\Omega})^k \phi_{ac} = 0, \label{eq:gensyscurlh}\\
g^{ab} \ooo\covD_a (\lieD_{\Omega})^k
\phi_{bc} = (\lieD_{\Omega})^k F_c + \sum_{j =
1}^{k} O( (\lieD_{\rotvf{}})^j [g - \ooo{g}] \cdot \ooo\covD_a
(\lieD_{\rotvf{}})^{k-j} \phi). \label{eq:gensysdivh}
\end{gather}
Applying our basic energy estimate \eqref{eq:basicenergy} we get the
following higher order energy estimate.
\begin{prop}\label{prop:higherbasicenergy}
Let $\phi_{ab}$ be a symmetric two tensor solving
\eqref{eq:gensyscurl} and \eqref{eq:gensysdiv}, then
\begin{align*} \energy[g,\allrot^k\phi](t_2) &+ 2(d+1)
\int_{\bulkregion{t_1}{t_2}} \frac{t}{\jb{t}} \tweight
\abs{\allrot^k\phi\cdot\tau}_{\ooo{g},\tau}^2 \D{\ooo{V}} -
\energy[g,\allrot^k\phi](t_1) \\
& \lesssim_{k,d}
\int_{t_1}^{t_2} \tweight \norm[2]{\allrot^k\phi(t)}^2 \norm[\infty]{g(t)}
\norm[\infty]{\jb{\ooo\covD}[g - \ooo{g}](t)} \jb{t}^{-1} \D{t} \\
&\quad + \int_{\bulkregion{t_1}{t_2}} \tweight
\abs{\allrot^k\phi}_{\ooo{g},\tau} \abs{\allrot^k
\trace_{\ooo{g}}\phi} \D{\ooo{V}}\\
&\quad + \int_{\bulkregion{t_1}{t_2}} \tweight \abs{\allrot^k\phi}_{\ooo{g},\tau}
\abs{\allrot^k F}_{\ooo{g},\tau} (1 + \abs{g - \ooo{g}}_{\ooo{g},\tau})
\D{\ooo{V}} \\ & \quad+ \int_{\bulkregion{t_1}{t_2}} \tweight
\abs{\allrot^k\phi}_{\ooo{g},\tau} \abs{\jb{\allrot}^{k-1}\ooo\covD\phi}_{\ooo{g},\tau} 
\abs{\jb{\allrot}^{\lceil k/2\rceil} [g - \ooo{g}]}_{\ooo{g},\tau}
\D{\ooo{V}}\\
& \quad + \int_{\bulkregion{t_1}{t_2}} \tweight
\abs{\allrot^k\phi}_{\ooo{g},\tau}
\abs{\jb{\allrot}^{\lceil k/2\rceil}\ooo\covD\phi}_{\ooo{g},\tau} 
\abs{\jb{\allrot}^{k} [g - \ooo{g}]}_{\ooo{g},\tau}
\D{\ooo{V}},
\end{align*}
where we recall that $\allrot$ denotes the collection of all
rotational vector fields. 
\end{prop}

\subsection{``Elliptic'' estimate}
In Proposition \ref{prop:higherbasicenergy} we see a term of the form
$\allrot^{k-1} \ooo\covD\phi$. We wish to control this term in terms of
$\allrot^k \phi$: we see that for the derivatives tangential to
$\spaceslice{t}$ the control is (more or less) built-in. But we have to worry about
transversal (time) derivatives. For this we will use the
equation.\footnote{This procedure of solving the equations of motion
for the time derivatives comes from the fact that $\spaceslice{t}$ is
non-characteristic, and is one of the basic ingredients for the
Cauchy-Kowalevski theorem.}

We start by establishing some facts concerning the rotation vector
fields $\rotvf{ij}$. An easy consequence of \eqref{eq:defrotvf} is that
\begin{equation}\label{eq:sizeestgradrot}
\abs{\ooo{\covD}\rotvf{ij}}_{\ooo{g},\tau} \approx_{d} 1.
\end{equation}
This implies that, for a $(p,q)$ tensor field $V$, 
\begin{equation}\label{eq:rotliecovexch}
\begin{split}
\abs{\lieD_{\rotvf{ij}}V}_{\ooo{g},\tau} & \lesssim_{p,q,d}
\abs{\ooo\covD_{\rotvf{ij}}V}_{\ooo{g},\tau} + \abs{V}_{\ooo{g},\tau},
\\
\abs{\ooo\covD_{\rotvf{ij}}V}_{\ooo{g},\tau} & \lesssim_{p,q,d}
\abs{\lieD_{\rotvf{ij}}V}_{\ooo{g},\tau} + \abs{V}_{\ooo{g},\tau}.
\end{split}
\end{equation}

Next we recall a well-known fact of the geometry of Euclidean spaces:
\begin{lem}
We have that along $\dS$
\[ \sum_{i,j = 1}^{d} \rotvf{ij}^a\rotvf{ij}^b = 2\jb{t}^2
(\ooo{g}^{ab}+ \tau^a\tau^b).\]
\end{lem}
\begin{proof}
Let $e_{1},\dots,e_{d+1}$ denote the standard unit vectors in
$\Real^{d+1}$. We have $\rotvf{ij} = x^i e_j - x^j e_i$, and the
Euclidean (inverse) metric is $\sum_{i} e_i \otimes e_i$. We have,
writing $r^2 = \sum_i (x^i)^2$,
\begin{align*}
\sum_{i,j} \rotvf{ij}\otimes\rotvf{ij} &= \sum_{i,j} (x^i)^2
e_j\otimes e_j + (x^j)^2 e_i\otimes e_i - x^i x^j (e_i\otimes e_j +
e_j\otimes e_i) \\
& = 2 \left[ r^2 \sum_i e_i\otimes e_i - \left(\sum_i x^i e_i\right)
\otimes \left(\sum_i
x^i e_i\right)\right].
\end{align*}
Noting that $\sum_i x^i e_i$ represent $r$ times the unit radial
vector field, we have our claim. 
\end{proof}
\begin{cor}
For a $(p,q)$ tensor field $V$, 
\begin{equation}
\abs{\ooo{\covD}V}_{\ooo{g},\tau} \lesssim_{p,q,d} \frac{1}{\jb{t}}
\abs{\jb{\allrot} V}_{\ooo{g},\tau} +
\abs{\tau^a\ooo\covD_{a} V}_{\ooo{g},\tau}.
\end{equation}
\end{cor}

Now, letting $X^a$ be a vector field tangent to $\spaceslice{t}$, we
see from \eqref{eq:gensyscurl} that
\[ (\tau^a X^b - \tau^b X^a)\ooo\covD_a\phi_{bc} = 0.\]
This implies that, together with the above corollary, that for $\phi$
verifying \eqref{eq:gensyscurl} we have
\[ \abs{\ooo\covD \phi}_{\ooo{g},\tau} \lesssim_{d} \frac{1}{\jb{t}}
\abs{\jb{\allrot} \phi}_{\ooo{g},\tau} + \abs{\tau^a\tau^b\tau^c
\ooo\covD_{a}\phi_{bc}}.\]
For this final remaining component, we need to use
\eqref{eq:gensysdiv}, which implies
\[ \abs{\tau^a\tau^b\tau^c \ooo\covD_{a}\phi_{bc}} \lesssim_{d}
\frac{1}{\abs{g^{ab}\tau_a\tau_b}} \left[ \jb{t}^{-1}
\abs{g}_{\ooo{g},\tau} \abs{\jb{\allrot}\phi}_{\ooo{g},\tau} +
\abs{F}_{\ooo{g},\tau} \right].\]
Combining our computations we have the following estimates.
\begin{prop}\label{prop:ellipticest}
Let $\phi$ be a symmetric 2-tensor solving \eqref{eq:gensysdiv} and
\eqref{eq:gensyscurl}, then we have
\begin{equation}\label{eq:ellipticest1}
\abs{\ooo\covD \phi}_{\ooo{g},\tau} \lesssim_{d}
\frac{1}{1 - \abs{g - \ooo{g}}_{\ooo{g},\tau}} 
\left[ \jb{t}^{-1} \left(1 + \abs{g - \ooo{g}}_{\ooo{g},\tau}\right)
\abs{\jb{\allrot}\phi}_{\ooo{g},\tau}
+ \abs{F}_{\ooo{g},\tau} \right].
\end{equation}
For the higher order norms we have 
\begin{multline}\label{eq:ellipticest2}
\abs{\jb{\allrot}^k \ooo\covD\phi}_{\ooo{g},\tau} \lesssim_{k,d}
\frac{1+ \abs{g - \ooo{g}}_{\ooo{g},\tau}}{\left(1 - \abs{g - \ooo{g}}_{\ooo{g},\tau}\right)^{k+1}} 
\Big[ \abs{\jb{\allrot}^k F} \\
+ \jb{t}^{-1} \left( 1 + \abs{\jb{\allrot}^{\lceil k/2\rceil}[g -
\ooo{g}]}_{\ooo{g},\tau}\right)^k \abs{\jb{\allrot}^{k+1}
\phi}_{\ooo{g},\tau} \\
+ \jb{t}^{-1} \left(1 + \abs{\jb{\allrot}^{\lceil k/2\rceil}[g -
\ooo{g}]}_{\ooo{g},\tau}\right)^{k-1} \abs{\jb{\allrot}^k [g -
\ooo{g}]}_{\ooo{g},\tau} \abs{\jb{\allrot}^{\lceil k/2\rceil}
\phi}_{\ooo{g},\tau} \Big].
\end{multline}
\end{prop}
\begin{proof}
The estimate \eqref{eq:ellipticest1} follows immediately from the
discussion before the statement of the proposition;
\eqref{eq:ellipticest2} is a consequence of \eqref{eq:ellipticest1}
applied to the system \eqref{eq:gensyscurlh} and
\eqref{eq:gensysdivh}, and using induction on $k$ after noting that 
the right hand side of
\eqref{eq:gensysdivh} contains a term of the form
$\jb{\allrot}^{k-1}\ooo\covD\phi$ with \emph{fewer} angular
derivatives. 
\end{proof}

\subsection{Sobolev estimates}
In order to obtain $L^\infty$ estimate from the $L^2$ based energy
quantities, we need some form of \emph{uniform} Sobolev estimates.
This we obtain simply from the standard Sobolev inequalities on the
standard sphere, using that $\spaceslice{t}$ are isometric to spheres
with radii $\jb{t}$. 

\begin{lem}\label{lem:sobolev}
Let $V$ be a $(p,q)$-tensor field, we have that for $k > d/2$ 
\[ \norm[\infty]{V(t)} \lesssim_{d,p,q} \jb{t}^{-d/2}
\norm[2]{\jb{\allrot}^kV(t)}.\]
\end{lem}
\begin{proof}
The Sobolev inequality on a fixed compact Riemannian manifold (such as the
unit sphere) for tensor fields along the manifold is standard (see, e.g.\
\cite{Taylor2011}). To obtain it for our (space-time) tensor field we
partially scalarise the normal components by examining the contraction
of $V$ against $\tau$. But noting that $\tau$ commutes with
$\rotvf{ij}$ we see that we can still write the expression as in the
compact form above. It suffices to obtain the factor $\jb{t}^{-d/2}$.
But this follows from scaling, and noting that $\D*{\ooo{A}} =
\jb{t}^d \areasph$. 
\end{proof}

\tocaddline
\section{The case of ``small data'': perturbations of
\texorpdfstring{$\dS$}{de Sitter}}
We are now ready to attack the quasilinear system
\eqref{eq:mainsystem} for $\phi$ small. Note that here the
source term $F$ of \eqref{eq:gensysdiv} vanishes identically. The
result that we will prove is:
\begin{thm}[Small data case] \label{thm:mainSmallData}
For every positive integer\footnote{The lower bound for $N$ here is
not sharp. One can improve the bound if we keep more of the structure
of \eqref{eq:gensysdivh} instead of the rough pigeonholed estimate
given in Proposition \ref{prop:higherbasicenergy}.} $N > d+3$, 
there exists a real constant $\epsilon_0 > 0$ 
depending on the dimension $d$ and the number $N$, for which the following
holds:
if $\phi$ solves \eqref{eq:mainsystem} on $\bulkregion{t_1}{t_2}$
for any $t_2 > t_1 > 0$, and for some $t_0\in (t_1,t_2)$ we
have 
\[ \energy[g,\jb{\allrot}^N\phi](t_0) < \epsilon_0, \]
and 
\[ \norm[\infty]{g(t_0) -\ooo{g}(t_0)} < \frac{1}{8(d+1)},\]
then $\phi$ can be extended to a (classical) solution of
\eqref{eq:mainsystem} on $\bulkregion{t_1}{\infty}$ such that
\begin{equation}\label{eq:mainenergybound}
 \sup_{t\in (t_0,\infty)} \energy[g,\jb{\allrot}^N\phi](t) < 2
\epsilon_0\end{equation}
and 
\[ \sup_{t\in (t_0,\infty)}
\jb{t}\norm[\infty]{\jb{\allrot}^{N - \lceil d/2\rceil}\phi(t)} <
\infty.\] 
\end{thm}
Note that by \eqref{eq:ellipticest2} the $L^\infty$ estimate implies
the following version for the covariant derivative
\[ \sup_{t\in(t_0,\infty)} \jb{t}^{k+1} \norm[\infty]{\ooo\covD^k
\phi} < \infty \]
where $k \leq N - \lceil d/2\rceil$; that is to say, each additional
derivative gains one factor of $t$ decay.  

\subsection{Estimates of \texorpdfstring{$\psi$ and
$\trace_{\ooo{g}}\phi$}{the inverse field and the}}
We note here some immediate consequences of \eqref{eq:defpsi}.
Firstly, we have that $\psi + \phi\psi = -\phi$, which implies that
\[ \abs{\psi}_{\ooo{g},\tau} \left( 1 -
\abs{\phi}_{\ooo{g},\tau}\right) \leq \abs{\phi}_{\ooo{g},\tau}.\]
This gives us 
\begin{lem}\label{lem:psifromphi}
Whenever $\abs{\phi}_{\ooo{g},\tau} < \frac12$, we have
$\abs{\psi}_{\ooo{g},\tau} \lesssim \abs{\phi}_{\ooo{g},\tau}$. 
\end{lem}

Next, from \eqref{eq:traceid} we obtain
\begin{lem}\label{lem:tracefromphi}
Whenever $\abs{\phi}_{\ooo{g},\tau} < \frac12$, we have
$\abs{\trace_{\ooo{g}}\phi}_{\ooo{g},\tau} \lesssim
\abs{\phi}_{\ooo{g},\tau}^2$. 
\end{lem}
Noting that rotational vector fields commute with $\trace_{\ooo{g}}$,
we can take higher derivatives of \eqref{eq:traceid} and obtain, under
the assumption\footnote{This condition can be significantly relaxed,
but the version quoted here suffices for our purposes.} that 
$\abs{\jb{\allrot}^k\phi}_{\ooo{g},\tau} < \frac12$,
\begin{equation}\label{eq:tracefromphiwder}
\abs{\jb{\allrot}^k \trace_{\ooo{g}}\phi}_{\ooo{g},\tau} \lesssim_k
\abs{\jb{\allrot}^k \phi}_{\ooo{g},\tau} 
\abs{\jb{\allrot}^{\lceil k/2\rceil}\phi}_{\ooo{g},\tau}.
\end{equation}

\subsection{Proof of Theorem \ref*{thm:mainSmallData}}
Recall that \emph{local} well-posedness for the \cpmc problem is
relatively straight-forward, as the equations can be cast locally as
quasilinear wave equations.\footnote{In principle one can also prove local
well-posedness of \eqref{eq:mainsystem} directly using the energy method 
based on the estimates discussed in this and the next section, see e.g.\ the 
general techniques discussed in \cite{CouHil1962, Kato1975, HuKaMa1976}. We will
not pursue this line of argument in here.} Thus it suffices for us to
prove the \emph{a priori} energy bound \eqref{eq:mainenergybound}. We
do so using a bootstrap/continuity argument.

Let us now assume that the solution $\phi$ exists on
$\bulkregion{t_1}{T}$ with the bound
\begin{equation}\label{eq:bootstrapenergy}
\sup_{t\in (t_0,T)} \energy[g,\jb{\allrot}^N\phi](t) < 4 \epsilon_0,
\end{equation}
and for convenience
\begin{equation}\label{eq:bootstrapg}
\sup_{t\in (t_0,T)} \norm[\infty]{g(t) - \ooo{g}(t)} <
\frac{1}{6(d+1)}.
\end{equation}
To close the bootstrap it suffices to show that, for $\epsilon_0$ sufficiently 
small, we can
improve the estimates \eqref{eq:bootstrapenergy} and
\eqref{eq:bootstrapg}. 

First, we note that under assumption \eqref{eq:bootstrapg} we have by
Proposition \ref{prop:fundposstress} that 
\[ \energy[g,\jb{\allrot}^N\phi](t) \gtrsim_{d} \tweight
\norm[2]{\jb{\allrot}^N\phi(t)}^2.\]
For $N \geq d/2 + 1$ we have, by the Sobolev Lemma \ref{lem:sobolev} that
\begin{equation}\label{eq:energytoinfty1}
 \norm[\infty]{\jb{\allrot}^{N - \lceil d/2 \rceil}\phi}^2
\lesssim_{N,d} \jb{t}^{-d} \tweight^{-1} \energy[g,\jb{\allrot}^N\phi](t).
\end{equation}
\begin{lem}[$L^\infty$ decay]\label{lem:linftydecay}
Assuming \eqref{eq:bootstrapg} and \eqref{eq:bootstrapenergy}, we have
that
\[ \sup_{t\in(t_0,T)} \jb{t} \norm[\infty]{\jb{\allrot}^{N - \lceil
d/2\rceil} \phi(t)} \lesssim_{N,d} \sqrt{\epsilon_0}. \]
Hence for $\epsilon_0$ sufficiently small, we have the improved
version of \eqref{eq:bootstrapg}:
\[ \sup_{t\in (t_0,T)} \norm[\infty]{g(t) - \ooo{g}(t)} <
\frac{1}{8(d+1)}.\]
\end{lem}
\begin{proof}
This first estimate follows from \eqref{eq:energytoinfty1} together
with \eqref{eq:deftweight}. The
second estimate uses the fact that $g - \ooo{g} = O(\psi,\psi^2)$ and
Lemma \ref{lem:psifromphi}. 
\end{proof}

It remains to improve the energy bound \eqref{eq:bootstrapenergy}. We
do so by studying the higher order energy estimate Proposition
\ref{prop:higherbasicenergy}. We observe that from the definition
\eqref{eq:defgfromphi} of $g$, 
\[ \abs{\jb{\allrot}^k[g - \ooo{g}]}_{\ooo{g},\tau} \lesssim_{k,d}
\abs{\jb{\allrot}^k \psi}_{\ooo{g},\tau}\left(1 + \abs{\allrot^{\lceil
k/2\rceil}\psi}_{\ooo{g},\tau}\right).\]
Putting that together with our elliptic estimate Proposition
\ref{prop:ellipticest} and with our pointwise estimates
Lemma \ref{lem:psifromphi} and \eqref{eq:tracefromphiwder}, we have,
finally, 
\[
\energy[g,\jb{\allrot}^k\phi](t) -
\energy[g,\jb{\allrot}^k\phi](t_0)  \lesssim_{k,d}
\int_{t_0}^t \tweight \norm[2]{\jb{\allrot}^k\phi(s)}^2
\norm[\infty]{\jb{\allrot}^{\lceil k/2\rceil+1}\phi(s)} \jb{s}^{-1}
\D{s}.
\]
For $N > d + 3$ we have
\[ N - \lceil d / 2\rceil > \lfloor d/2\rfloor + 3 \geq \lceil
N/2\rceil + 1\]
and so we can apply Lemma \ref{lem:linftydecay} and the bound of $L^2$
by energy to get
\begin{equation}\label{eq:lastmainproofest} \energy[g,\jb{\allrot}^N\phi](t) -
\energy[g,\jb{\allrot}^N\phi](t_0)  \lesssim_{N,d}
\sqrt{\epsilon_0} \int_{t_0}^t \energy[g,\jb{\allrot}^N\phi](s) 
 \jb{s}^{-2}
\D{s}.
\end{equation}
So applying Gronwall's inequality and using that $\jb{s}^{-2}$ is integrable, 
we can pick $\epsilon_0$
sufficiently small, depending on $N$ and $d$ \emph{but independently of $T$} such that
\[ \energy[g,\jb{\allrot}^N\phi](t) < 2
\energy[g,\jb{\allrot}^N\phi](t_0) \]
for all $t \in (t_0, T)$. This complete the proof of Theorem \ref{thm:mainSmallData}. 

\subsection{Improved decay for temporal components}
The components $\phi_{ab}\tau^b$ in fact enjoys slightly better decay rates
than indicated in Theorem \ref{thm:mainSmallData}. Reading off the
\emph{a priori} estimates above, we get the following:
\begin{cor}\label{cor:improveddecay}
Under the assumptions of Theorem \ref{thm:mainSmallData}, 
\[ \int_{t_0}^\infty \norm[\infty]{\jb{\allrot}^{N-\lceil d/2\rceil}
\phi \cdot \tau(s)}^2 s \D{s} \lesssim_{d} 
\energy[g,\jb{\allrot}^N\phi](t_0).\]
\end{cor}
\begin{proof}
In the proof of Theorem \ref{thm:mainSmallData}, the integrated energy
term on the left hand side of Proposition \ref{prop:higherbasicenergy}
is just dropped due to its being positive. \emph{A posteriori} from
the proof we can re-insert the integrated energy into 
\eqref{eq:lastmainproofest} to obtain
\[ \int_{\bulkregion{t_0}{\infty}} \frac{t}{\jb{t}} \tweight \abs{\jb{\allrot}^{N}
\phi\cdot \tau}_{\ooo{g},\tau}^2 \D{\ooo{V}} \lesssim
\energy[g,\jb{\allrot}^N\phi](t_0).\] 
Observe then from the Sobolev Lemma \ref{lem:sobolev} we have
\begin{align*}
\int_{t_0}^t \norm[\infty]{\jb{\allrot}^{N-\lceil
d/2\rceil}\phi\cdot\tau(s)}^2 s \D{s}  & \lesssim_d \int_{t_0}^t
\norm[2]{\jb{\allrot}^{N} \phi\cdot \tau(s)}^2 s \jb{s}^{- d} \D{s}
\\
& \lesssim_{d} \int_{\bulkregion{t_0}{t}} \frac{t}{\jb{t}} \tweight
\abs{\jb{\allrot}^N \phi \cdot \tau}_{\ooo{g},\tau}^2 \D{\ooo{V}}  \\
& \lesssim_{d}
\energy[g,\jb{\allrot}^N\phi](t_0). \qedhere
\end{align*}
\end{proof}
In spherical symmetry this bulk term partially cancels the term with
the bad-sign in \eqref{eq:strwdeform} and removes the need for the
$\tweight$-renormalisation. In fact this leads to the predicted decay
rates stated in the paragraph before Theorem
\ref{thm:sphsymstableigm}. 

\tocaddline
\section{The main theorem}
\begin{thm}[Stability of expanding, spherically symmetric solutions]
\label{thm:bigmaintheorem}
Let $(\eta,\zeta):\bulkregion{t_1}{\infty} \to\Real^2$ represent, with $\zeta > -
\frac{1}{d+1}$, and $t_1 > 0$, a spherically symmetric solution as
described in Section \ref{sec:sphsymrevisit}. Given an integer $N > d
+ 3$, there exists a real constant $\epsilon_0 > 0$ depending on
$\eta,\zeta,N,$ and $d$, such that the following holds:
whenever $\phi$ is a solution to \eqref{eq:mainsystem} on
$\bulkregion{t_1}{t_2}$ such that for some $t_0 \in (t_1,t_2)$ we
have 
\[ \energy[g,\jb{\allrot}^N(\phi_{ab} - \eta \tau_a\tau_b - \zeta
\ooo{g}_{ab})](t_0) < \epsilon_0 \]
and
\[ \norm[\infty]{g^{ab}+ \frac{1}{(1 + \zeta - \eta)^2}
\tau^a\tau^b - \frac{1}{(1+\zeta)^2} (\ooo{g}^{ab} +
\tau^a\tau^b)}(t_0) < \frac{1}{8(d+1)},\]
then we have that $\phi$ can be extended to a solution on
$\bulkregion{t_1}{\infty}$ such that 
\[ \sup_{t\in (t_0,\infty)} \jb{t} \norm[\infty]{\jb{\allrot}^{N -
\lceil d/2\rceil}\phi(t)} < \infty. \]
\end{thm}
Taking the notation of Section \ref{sec:sphsymrevisit}, we write
$\breve{\phi}_a^b = \eta \tau_a\tau^b + \zeta\delta_a^b$ for the
spherically symmetric expanding solution, and $\breve{g}^{ab}$ its
corresponding induced metric. By noting that for any rotation 
vector field
\[ \lieD_{\rotvf{ij}}\phi = \lieD_{\rotvf{ij}}(\phi - \breve\phi)\]
since the background solution is spherically symmetric, we see that
for higher order derivatives we can proceed analogously as in the
proof of Theorem \ref{thm:mainSmallData}. The main difficulty lies in
the $0$th order terms, which instead solves the system
\eqref{eq:largedatasystem}. What we will make use of is the decay
estimates stated before Theorem \ref{thm:sphsymstableigm}: an analysis
of \eqref{eq:igmrotsym} shows that both $\zeta$ and $\eta$ decay, and
we have
\[ \norm[\infty]{\breve\phi(t)}
\lesssim_{\delta,\norm[\infty]{\breve\phi(t_0)}} \jb{t}^{-d-1 +
\delta}.\]
This shows that its corresponding weighted energy
\[ \energy[\breve{g},\breve\phi](t)
\lesssim_{d,\delta,\norm[\infty]{\breve\phi(t_0)}} \jb{t}^{-2d +
2\delta}.\]
(Compare this to the situation of the almost conservation law in
\eqref{eq:basicenergy}: one should not be surprised because we were
somewhat wasteful in the derivation of \eqref{eq:basicenergy}, where
the terms with good signs on the right-hand-side of
\eqref{eq:strwdeform} are just thrown away, when in fact they provide
some weak form of integrated energy decay.)

We sketch here two arguments giving the proof of Theorem
\ref{thm:bigmaintheorem}. The basic ingredients are still energy
estimates and a bootstrap argument, which are largely similar to the
proof of Theorem \ref{thm:mainSmallData}: therefore we will just
highlight the differences between the proofs and that of Theorem
\ref{thm:mainSmallData}. 

\begin{proof}[Sketch of first proof of Theorem
\ref*{thm:bigmaintheorem}]
Using the faster energy decay of the spherically symmetric
backgrounds, we can approach Theorem \ref{thm:bigmaintheorem} using a
Cauchy stability argument. The basic idea is the following: instead of
bootstrapping on $L^\infty$ decay of the solution, we bootstrap on
$L^\infty$ boundedness, as one would do for a \emph{local
well-posedness} result. This shows that for sufficiently small
initial perturbations, the solution remain a small perturbation up to
some large finite time $T$. Using that the background has decayed, at
time $T$ now we are in a situation where Theorem
\ref{thm:mainSmallData} applies: it is crucial here that the order of
the quantifiers in the statement of Theorem \ref{thm:mainSmallData} is
as it is, such that $\epsilon_0$ is independent of the time $t_0$. 

Here we will study \eqref{eq:largedatasystem} for $\phi - \breve\phi$
and \eqref{eq:gensyscurlh} and \eqref{eq:gensysdivh} for the higher
order derivatives. 

First note that provided $\phi - \breve\phi$ is sufficiently small, we
can appeal to Proposition \ref{prop:fundposstress} to get coercivity
of the energy on a weighted $L^2$ norm. 
Examining Proposition \ref{prop:higherbasicenergy} we see that using 
\[ 
F_c = -2(\delta^a_e + \breve\psi^a_e)(\psi - \breve\psi)^{be}
\massterm_{abc} + \massterm_{abc}(\psi - \breve\psi)^{be}
(\psi-\breve\psi)^a_e, \]
we have, in fact, that provided $\phi - \breve\phi$ is sufficiently
small, the estimate
\[ \energy[g,\jb{\allrot}^k(\phi-\breve\phi)](t) -
\energy[g,\jb{\allrot}^k(\phi-\breve\phi)](t_0) \lesssim \int_{t_0}^t
\energy[g,\jb{\allrot}^k(\phi - \breve\phi)](s) \jb{s}^{-1} \D{s}\]
where the implicit constant depends on $k,d$ as well as the background
solution $\eta,\zeta$ and an assumed $L^\infty$ bootstrap bound on
$\abs{\jb{\allrot}^{\lceil k/2\rceil + 1}(\phi -
\breve\phi)}_{\ooo{g},\tau}$. Here we note that we do not need to do
anything special to control the trace term $\trace_{\ooo{g}} \phi -
\breve{\phi}$, since we do not need decaying coefficients!
From Gronwall's inequality we get 
that the energy of $\phi - \breve\phi$
grows at most linearly in $t$; hence by taking initial perturbations
arbitrarily small, we can make the bootstrap bound be satisfied for
arbitrarily long (finite) times. This proves Cauchy stability in a small
neighbourhood of the spherically symmetric solution $\breve\phi$. 

Now fix $T$ sufficiently large that $\breve\phi$ has sufficiently
decayed. By choosing our initial $\epsilon_0$ small we can
guarantee that 
\[ \energy[g,\jb{\allrot}^N\phi](T) \leq \energy[g,\jb{\allrot}^N(\phi
- \breve\phi)](T) + \energy[g,\phi](T) \]
is sufficiently small so we can apply Theorem \ref{thm:mainSmallData}.
\end{proof}

\begin{proof}[Sketch of second proof of Theorem
\ref*{thm:bigmaintheorem}]
One can also approach the proof of Theorem
\ref{thm:bigmaintheorem} by directly studying the system
\eqref{eq:largedatasystem} \`a la the proof of Theorem
\ref{thm:mainSmallData}. The fact that the coefficients
$\massterm_{abc}$ decays like $\jb{t}^{-(d+1)+\epsilon}$ means that
the contribution of the inhomogeneity $F_c$ to the energy estimate
Proposition \ref{prop:higherbasicenergy} is relatively harmless (with
the $\jb{t}^{-1}$ weight carried by the volume form $\D*{\ooo{V}}$
this becomes integrable in time). The trace term is also treatable:
the trace identity \eqref{eq:traceid} now implies the schematic
decomposition
\[ \trace_{\ooo{g}}(\phi - \breve\phi) = \breve\phi (\psi-\breve\psi)
+ \breve\psi (\phi - \breve\phi) + (\phi - \breve\phi)(\psi
-\breve\psi)\]
which consists of a quadratic term (which is ``higher order'' and we
can control using $L^\infty$ decay) and two linear terms which have
good decay in the coefficients (again, $\breve\phi$ decays like
$\zeta$). Using also that $\lieD_\Omega [g - \ooo{g}] = \lieD_\Omega
[g - \breve{g}]$ for rotation vector fields, we see that the only term
on the right hand side of our energy estimate in Proposition
\ref{prop:higherbasicenergy} that we may have difficulty controlling
is the first term which requires estimating $\ooo\covD g$, or rather,
by triangle inequality, $\ooo\covD \breve{g}$. This term, however,
also decays using the decay of the background $\breve{\phi}$. 

The bootstrap step in this argument is slightly more delicate,
however, using that we have essentially ``linear'' terms appearing in
the energy estimate. Basically the energy estimate outlined above
shows that, under the assumption that the energy
$\energy[g,\jb{\allrot}^N(\phi - \breve\phi)](t)$ remains sufficiently
small, say $< \epsilon_1$, we can prove that
$\energy[g,\jb{\allrot}^N(\phi-\breve\phi)](t) \leq C
\energy[g,\jb{\allrot}^N(\phi-\breve\phi)](t_0)$ for some really large
constant $C$. Hence we need to pick $\epsilon_0 < C^{-1} \epsilon_1$
for our initial data in order to close the bootstrap. (Note that in
the small data case for every $\delta > 0$, we can choose sufficiently
small $\epsilon_0$ such that the almost conservation law 
\[ \energy[g,\jb{\allrot}^N\phi](t) \leq (1+\delta)
\energy[g,\jb{\allrot}^N\phi](t_0) \]
holds for all $t > t_0$. In the ``large data'' regime, the
$(1+\delta)$ bound is not tenable, and the best we can do is some
fixed large constant $C$ \emph{depending on the chosen background
around which we perturb}.) 
\end{proof}

\begin{rmk}\label{rmk:improveddecay}
It then follows that Corollary \ref{cor:improveddecay} also holds
under the assumptions of Theorem \ref{thm:bigmaintheorem}.
\end{rmk}

\subsection{Geometric implications}\label{sec:geometricimplications}
We conclude this section with some discussion of the geometric
implications of the Theorems \ref{thm:mainSmallData} and
\ref{thm:bigmaintheorem}, and in particular justify the
interpretation stated in Theorem \ref{thm:maintheoremgeometric}.

\subsubsection{Extrinsic geometry}
We first set some notations. Let $\tau$ the unit time-like vector
field on $\dS$ as before, and recall from Section \ref{sec:IGMgauge} 
the inverse Gauss map $\Phi$ from $\dS$ to our solution manifold 
$M$ and its differential $A$. First parametrise $\Real^{1,d+1}$ in
radial coordinates $(t,r,\omega)$, and we can define $u \eqdef t-r$. 
Outside of $\{r = 0\}$ the system $(t,u,\omega)$ defines a coordinate
chart on Minkowski space. In this coordinate system we have that
\[ \dS = \{ u(u - 2 t)  = 1 \}, \qquad \bulkregion{0}{\infty} = \{ u =
t - \jb{t}, t > 0\}\]
and
\[ \tau = \jb{t} \partial_t +  (\jb{t} - t) \partial_u. \]
Observe that 
\[ \jb{t} - t = \frac{1}{\jb{t} + t}.\]
Now, by Corollary \ref{cor:expansiontolightcone}, we have that for
each of the spherically symmetric background solutions in the
statements of Theorems \ref{thm:mainSmallData} and
\ref{thm:bigmaintheorem}, there exists some $u_\infty$ such that the
solution is asymptotic to $\{u = u_\infty\}$. In the case of $\dS$, we
have $u_\infty = 0$. 

Now let $(\tilde{u}, \tilde{t},\tilde{\omega})$ be the surface defined
by $\Phi$, written as functions of $(t,\omega)$. We have that 
\begin{equation}
\begin{split}
\partial_t \tilde{u} &= \left( 1 - \frac{t}{\jb{t}} \right)
(1+\phi_{\tau\tau}),\\
\partial_t \tilde{t} &= 1 + \phi_{\tau\tau}.
\end{split}
\end{equation}
Using that $\phi_{\tau\tau} \lesssim \jb{t}^{-1}$, we have that
$\partial_t \tilde{u} \lesssim \jb{t}^{-2}$ and is integrable, while
$\partial_t \tilde{t}$ just misses being integrable. Thus we have that
\begin{lem}\label{lem:boundedshift}
Under the assumptions in Theorem \ref{thm:mainSmallData} or
\ref{thm:bigmaintheorem}, there exists a function 
$\tilde{u}_\infty:\Sphere^d \to\Real$ such that for
each fixed $\omega$ 
\[ \lim_{t\to\infty} \abs{\tilde{u}(t,\omega) -
\tilde{u}_\infty(\omega)} = 0.\]
Furthermore we have $\sup_{\omega\in\Sphere^d} 
\abs{u_\infty - \tilde{u}_\infty(\omega)}^2 \lesssim
\epsilon_0$. 
\end{lem}

Examining the definitions we have that
\[ \abs{\partial_\omega \partial_t\tilde{u}(t,\omega)} \leq
\frac{\jb{t} - t}{\jb{t}} \abs{\allrot \phi_{\tau\tau}(t)}.\]
Integrating $t$ over $(0,\infty)$ and $\omega$ over some curve on
$\Sphere^d$, we have that 
\begin{lem}\label{lem:uinftycont}
The function $\tilde{u}_\infty$ described in Lemma
\ref{lem:boundedshift} is Lipschitz continuous. 
\end{lem}
The two Lemmata \ref{lem:boundedshift} and \ref{lem:uinftycont} should
be compared with Theorem \ref{thm:expansionstability}. 

\subsubsection{Causal geometry}
To complete our justification of Theorem
\ref{thm:maintheoremgeometric}, we first discuss the causal geometry
of expanding solutions. Let $(t,\omega(t))$ be a curve in
$\Real\times\Sphere^d$, its image
$(\tilde{t}(t),\tilde{u}(t),\tilde{\omega}(t))$ is time-like if and
only if
\[ - \partial_t\tilde{t} \partial_t\tilde{u} + (\partial_t\tilde{u})^2
+ \jb{t}^2(\partial_t\tilde{\omega})^2 < 0 .\]
Using the uniform decay rates on $\phi$ this requires
\[ \abs{\partial_t\tilde{\omega}} \lesssim \jb{t}^{-2}\]
which implies that there exists $\tilde{\omega}_\infty$ such that 
\[ \abs{\tilde{\omega}(t) - \tilde{\omega}_\infty} \lesssim
\jb{t}^{-1}.\]
Noting that $\partial_t \omega \approx \partial_t\tilde{\omega}$ by
the uniform decay of $\phi$, we have that there exists some
$\omega_\infty$ toward which $\omega(t)$ converges at rate
$\jb{t}^{-1}$ under our time-like
assumptions. 

Now let $(t,\omega_0(t))$ and $(t,\omega_1(t))$ be two curves, whose
images $\tilde{\omega}_0$ and $\tilde{\omega}_1$ converges to the same
$\tilde{\omega}_\infty$. By the above $\abs{\omega_0 - \omega_1}
\lesssim \jb{t}^{-1}$. Now let $\gamma:(0,1)\to\Sphere^d$ be a geodesic
curve connecting $\gamma_0$ and $\gamma_1$. Integrating
$\partial_\omega \tilde{u}$ along $\gamma$ we obtain that
$\abs{\tilde{u}_0 - \tilde{u}_1}\lesssim \jb{t}^{-1}$. This shows 
\begin{prop}
There exists some $\tilde{u}_{\bar\infty}: \Sphere^d \to\Real$ such
that whenever 
$(\tilde{t}(t),\tilde{u}(t), \tilde{\omega}(t))$ is a smooth time-like
curve on our solutions manifold, we have
\[ \lim_{t\to\infty} \tilde{u}(t) =
\tilde{u}_{\bar\infty}(\lim_{t\to\infty} \tilde{\omega}(t)).\]
\end{prop}

The above proposition makes precise the locally spatially convergence
discussed in Theorem \ref{thm:maintheoremgeometric}. One can rephrase
the convergence in the language of the causal boundary of the expanding
solutions: for an expanding solution to the \cpmc problem its \emph{timelike infinity} $i_+$ can be identified with a quotient of
the set of time-like curves where two curves are viewed as equivalent
if they are asymptotically parallel. The mapping
$\tilde{u}_{\bar\infty}$ can be regarded as being defined on $i_+$,
assigning to each point of future time-like infinity a corresponding
$u$-shift. Inside the causal past of a point $\tilde{\omega}_\infty
\in i_+$, the perturbed solution converges to a shift of $\dS$ by the
parameter $\tilde{u}_{\bar\infty}(\tilde{\omega}_\infty)$.

\subsubsection{Intrinsic geometry}
Immediately from the $L^\infty$ decay estimates of Theorems
\ref{thm:mainSmallData} and \ref{thm:bigmaintheorem}, we see that the
induced metrics $g$ on $M$ and $\ooo{g}$ on $\dS$ converge.
Furthermore, we have that the connection coefficients $\ooo\covD g =
O(\jb{t}^{-2})$, which is integrable. Thus the future time-like
geodesic completeness of $\dS$ implies the same of $M$, and that
the curves $\tilde{\Psi}(\cdot,\omega)$ for any fixed $\omega$ is
asymptotically geodesic in $M$, as $\Psi(\cdot,\omega)$ is geodesic in
$\dS$.

\tocaddline
\section{Discussion and open problems}
The theme of the present manuscript is one that is familiar in general
relativity, especially in the study of cosmological space-times. More
precisely, what we have is that expanding space-times with effectively
a positive cosmological constant (such as the de Sitter metric and
many of the FLRW solutions) have improved stability properties coming
from the exponential\footnote{Again, in proper time, which in our case
is something like $\sinh^{-1}(t)$.} decay induced by the space-time
expansion. One consequence is that the vector field method is
particularly simple to implement: it is only necessary in our case to
consider the multiplier field $\tau$ and the commutator family
$\allrot$. Of course, part of the simplification comes from the choice
of the inverse-Gauss-map gauge: this gave us a \emph{canonical} choice
of the vector fields with \emph{favourable built-in weights}. Compare
this to the case of, e.g.\ \cite{Speck2012}, where the appropriate
geometric renormalisation needs to be inserted in ``by hand'' to
factor in the different scaling properties of spatial and temporal
derivatives. 

Theorem \ref{thm:bigmaintheorem} above settles the question of
future (and also past, using a simple time reflection) asymptotic
stability for \emph{expanding spherically symmetric} solutions of the
\cpmc problem. Of course, this still leaves open two venues of
investigation: the stability properties of the cylindrical and
asymptotically cylindrical solutions, as well as the stability of the
singularity formation in the case of collapse (see Section
\ref{sec:classification}). As we have seen already in Theorems
\ref{thm:unstablecylinder} and \ref{thm:collapsestability}, the
stability properties of the corresponding ODE in the spherically
symmetric case are completely understood. How much of this carries to
the non-spherically symmetric case is unknown. We make several remarks
here:
\begin{itemize}
\item It is clear that the inverse-Gauss-map gauge will play no role
(in the current formulation) in the analysis of the \emph{cylindrical}
solution, due to the Gauss map being non invertible for that solution.
For the asymptotically cylindrical solutions the situation is less
clear, but one will have to contend with $\zeta$ approaching $-
\frac{1}{d+1}$ and hence $\eta$ blowing up asymptotically. (This
blow-up manifest for both $\phi$ and $\psi$ in fact.) This blow-up is
of course expected since we are essentially compactifying in time:
future time infinity corresponds to the slice $\spaceslice{0}$ under
the inverse-Gauss-map gauge. 
\item For the collapse cases, the inverse-Gauss-map gauge is
well-defined, but its role in the analysis is also not clear. Most
importantly is the fact that the collapse limit has a different causal
structure when compared with the case of asymptotic expansion: whereas 
in the expanding case
we have the presence of cosmological horizons, in the collapse case
the causal past of the singularity contains the entire manifold. 
\item Aside from the stability of collapse, it may be interesting to
also classify the different possible geometries near singularity
formation. 
\end{itemize}

\appendix
\tocaddline
\section{Review of pseudo-Riemannian geometry}
We gather here some facts concerning pseudo-Riemannian geometry,
partly to set notations and conventions, and partly to review and 
recall the main concepts, as the similarity and differences between the
pseudo-Riemannian/Lorentzian cases and the Riemannian cases may not be
familiar to all. Most of the material here are discussed in more
detail in \cite{ONeil1983}.

\subsection{Some linear algebra}
Let $V$ be a real vector space, and let $g:V\times V\to \Real$ be a
non-degenerate, symmetric, bilinear form. This form can be
equivalently viewed as an isomorphism from $V\to V^*$, which we call
$\flat$, and induces
another non-degenerate, symmetric, bilinear form which we shall write
as $(g^{-1}): V^*\times V^*\to\Real$. By definition $(g^{-1})$ can
also be treated as a mapping $V^*\to V$ which we write also as
$\sharp$ and we have that $g\circ
g^{-1} = \Id_{V^*}$ and $g^{-1} \circ g = \Id_{V}$ when interpreted as
linear mappings. 

Let $A:V\to V$ be a linear map. We say that $A$ is \emph{self-adjoint
relative to the form $g$} if for every pair $v,w\in V$ we have that 
$g(v,Aw) = g(Av,w)$, or that $g\circ A$ represent a symmetric bilinear
form. In the following we write
\begin{itemize}
\item $A^{-1}$ to be the inverse map to $A$ (we assume that $A$ is
invertible). 
\item $A^*:V^*\to V^*$ to be the dual map of $A$ given by
$(A^*\eta)(v) =
\eta(Av)$.
\end{itemize}
\begin{prop}\label{prop:linalg}
If $A$ is self-adjoint relative to $g$, then $A^{-1}$ is self-adjoint
relative to $g$, and $A^*$ is self-adjoint relative to $g^{-1}$.
\end{prop}
\begin{rmk}
In the case that $g$ is positive definite, the Proposition can also be
obtained as a consequence of the spectral theorem for
self-adjoint operators on finite dimensional vector spaces. The
spectral theorem in the case where $g$ is pseudo-Euclidean is a bit
more complicated, and here we obtain the result by elementary
calculations.
\end{rmk}
\begin{proof}
The first statement is evident:
\[ g(A^{-1}v,w) = g(A^{-1}v,A A^{-1}w) = g(A A^{-1}v,A^{-1}w) =
g(v,A^{-1}w).\]
For the second statement, we have that by nondegeneracy of $g$ we have
that if $\eta = v^\flat$ then $v = \eta^\sharp$, therefore 
\[ (A^*\eta)(w) = \eta(Aw) = g(v,Aw) = g(Av,w) \]
showing that 
\[ A^*(v^\flat) = (Av)^\flat. \]
This immediately implies that
\[ g^{-1}(A^*\eta,\zeta) = g((A^*\eta)^\sharp,\zeta^\sharp) =
g(A\eta^\sharp,\zeta^\sharp) = g(\eta^\sharp,A\zeta^\sharp) =
g^{-1}(\eta,A^*\zeta) \]
as desired. 
\end{proof}
As an immediate corollary we have that this implies that $(A^{-1})^* =
(A^*)^{-1}$ is self-adjoint relative to $g^{-1}$. 

\subsection{Mean curvature for non-degenerate
submanifolds}\label{app:defmc}
Let $(M,g)$ be a pseudo-Riemannian manifold and $N\subset M$ be a
submanifold with positive codimension. We make the non-degeneracy
assumption that $g$ induces on $N$ a pseudo-Riemannian metric $h$
(which is sometimes also called the \emph{first fundamental form} of
the embedding). As the metric $h$ is non-degenerate, at the point
$p\in N$ we can split $T_pM = T_p N \oplus T^\perp_p N$ where the two
subspaces are orthogonal. Let $\proj^\perp$ denote the projection
operator to $T^\perp N$.

The \emph{second fundamental form} of the embedding is a section of
$T^\perp N \otimes (T^*N)^2$, and is defined by
\begin{equation}
\sff(X,Y) = \proj^\perp \covD_X Y,
\end{equation} 
where $X,Y$ are vector fields along $N$. Observe that since the
Levi-Civita connection $\covD$ is torsion-free, we have that 
\[ \proj^\perp ( \covD_X Y - \covD_Y X) = \proj^\perp [X,Y] = 0,\]
the second equality being due to Frobenius' theorem. This implies that
the second fundamental form is symmetric: $\sff(X,Y) = \sff(Y,X)$. 

The \emph{mean curvature vector} is defined to be the\footnote{Some authors
define it with an additional factor of $1/\dim(N)$, based on the
motivation by the hypersurface case where the associated mean
curvature scalar would be the actual \emph{average} of the principal
curvatures (eigenvalues of the second fundamental form). This 
normalisation factor is unimportant in the following analysis: we 
drop it to simplify aesthetically certain computations.}
$h$-trace of $\sff$, 
\begin{equation}
\mcv = \trace_h \sff = \sum_{i =
1}^{\dim N} \sff(e_i,e_i),
\end{equation}
where $\{e_1, \dots, e_{\dim(N)}\}$ is an orthonormal frame for $TN$. Observe
that $\mcv$ is a field of normal vectors along the submanifold $N$. 

Suppose now that $N$ is an orientable nondegenerate hypersurface in $M$; by
orientability we can choose a unit normal vector field to $N$, which
we denote by $\vec{n}$. Then \emph{relative to this orientation} the
\emph{mean curvature scalar} is the quantity 
\begin{equation}
\mcs \eqdef g(\mcv,\vec{n});
\end{equation}
thus
while the \emph{magnitude} of the mean curvature scalar is independent
of the orientation, the \emph{sign} is not. In the title of this
article we implicitly follow the usual convention where the normal
vector $\vec{n}$ is ``directed inward''.

For oriented hypersurfaces, a related concept is that of the
\emph{shape operator}. Let $\vec{n}$ again be the chosen unit normal
vector field. Observe that since
\[ g(\covD_X\vec{n},\vec{n}) = \frac12 \covD_X [g(\vec{n},\vec{n})] = 0
\]
we have that $\covD_X\vec{n}$ is tangent to $N$ for any vector field
$X$ tangent to $N$. The \emph{shape operator} is defined to be the
section of $TN\otimes T^*N$ given by 
\begin{equation}
\shapeop(X) = \covD_X \vec{n}.
\end{equation} 
Note that its definition again depends on the chosen orientation. In
the case of the hypersurface there is a simple relation between the
shape operator and the second fundamental form. Let $X,Y$ be vector
fields tangent to $N$ then we have
\[ g(\shapeop(X),Y) = g(\covD_X \vec{n},Y) = \covD_X [g(\vec{n},Y)] -
g(\vec{n},\covD_X Y) = - g(\vec{n},\sff(X,Y)).\]
The symmetry of the second fundamental form then implies that
$\shapeop$ is
self-adjoint relative to $g$. Since $\shapeop(X)$ is $N$-tangent, we also
then have that $\shapeop$ is self-adjoint relative to $h$. 

Finally, we remark here the scaling properties of the various objects
defined here. Let $(M,g)$ and $(M',g')$ be pseudo-Riemannian manifolds
and $(N,h,\sff)$ and $(N',h',\sff')$ nondegenerate submanifolds of $M,M'$
respectively, with their induced first and second fundamental forms.
Suppose $F:M\to M'$ is an diffeomorphism which restricts to a
diffeomorphism $F|_N: N\to N'$. Suppose additionally that the
pull-back metrics satisfy
\[ F^* g' = \lambda^2 g \]
for some positive constant $\lambda$. Then a direct computation yields
that
\begin{subequations}
\begin{align}
F^* h' &= \lambda^2 h, \\
F^* \sff' &= \sff \label{eq:sffscaling}
\end{align}
(remember that the second fundamental form is a section of $T^\perp
N\otimes (T^*N)^2$). This implies that the mean curvature vector
scales like
\begin{equation}
F^* \mcv' = \frac{1}{\lambda^2} \mcv
\end{equation}
while relative to a chosen orientation, the mean curvature scalar
scales like
\begin{equation}
F^*\mcs' = \frac{1}{\lambda} \mcs.
\end{equation}
\end{subequations}

\subsection{Pseudo-Euclidean spaces, hyperquadrics, and the Gauss map}
\label{app:gaussmap}
Now let $M$ be $\Real^{m,q}$ equipped with
the pseudo-Euclidean quadratic form $g$. A family of distinguished
hypersurfaces are the \emph{hyperquadrics} 
\[ \Sphere^{m,q,\rho} \eqdef \set{ x\in \Real^{m,q}}{g(x,x) = \rho} \]
where $\rho\in \Real\setminus \{0\}$ is a parameter. Observe that
$\Sphere^{m,q,\rho}$ is a non-degenerate hypersurface with dimension
$m+q-1$, and the induced metric has $m$ time-like directions if $\rho
> 0$ and $m-1$ time-like directions if $\rho < 0$. 

Now, the quadratic form $g$ is invariant under the indefinite
orthogonal group $O(m,q)$; these actions give rise to isometries of
$\Sphere^{m,q,\rho}$. As the dimension of $O(m,q)$ is
$(m+q)(m+q-1)/2$, the hyperquadrics are maximally symmetric. One
easily sees that the vector field $\nu = -\sum_{i = 1}^{m+q}
x^i \partial_{x^i}$ is a normal vector field to the hyperquadrics
with $g(\nu,\nu) = \rho$ along $\Sphere^{m,q,\rho}$. So letting
$\vec{n} = \frac{1}{\sqrt{\abs{g(\nu,\nu)}}} \nu$, the associated
shape operator is $\shapeop = -\frac{1}{\sqrt{\abs\rho}}
\Id$ and hence the mean curvature scalar (with the orientation given by
$\vec{n}$) of $\Sphere^{m,q,\rho}$ is the constant
$\mcs = \frac{m+q-1}{\sqrt{\abs\rho}}$. This of course is compatible with the
fact that $\Sphere^{m,q,\rho_1}$ and $\Sphere^{m,q,\rho_2}$ with
$\rho_1\rho_2 > 0$ are related by a scaling symmetry. 

\begin{exa}
When $m = 0$, the only admissible $\rho$ are positive, and
$\Sphere^{0,q,\rho}$ are just the $q-1$ dimensional round spheres with
radius $\sqrt{\rho}$. 
\end{exa}

\begin{exa}
When $m = 1$, for $\rho < 0$, the normal vector $\nu$ is time-like,
and $\Sphere^{1,q,\rho}$ is a \emph{Riemannian} manifold isometric to
a hyperbolic space of dimension $q$. For $\rho > 0$, the normal vector
$\nu$ is space-like and $\Sphere^{1,q,\rho}$ is \emph{Lorentzian} and
is isometric to a \emph{de Sitter} space; it is also known as the
\emph{pseudo-sphere}. We will denote by $\dS$ the
manifold $\Sphere^{1,d+1,1}\subset \Real^{1,d+1}$. 

Incidentally the \emph{anti de Sitter} spaces are isometric to the
hyperquadrics $\Sphere^{2,q,\rho}$ with $\rho < 0$ and are analogously
called the \emph{pseudo-hyperbolic spaces}. 
\end{exa}

Since $M$ has a
vector space structure we can canonically identify $T_p M$ with $M$
for every $p\in M$. Now let $N$ be
an orientable nondegenerate hypersurface. Denote again by $\vec{n}$
a choice of the unit normal vector field along $N$, so that
$g(\vec{n},\vec{n}) =
\pm 1$ (the sign depends on whether $\vec{n}$ is time-like or
space-like). The canonical identification of $T_p M$ with $M$ allows
us to associate to each $\vec{n}$ a point, which by abuse of notation
we will also call $\vec{n}$, of $M = \Real^{m,q}$. Consider the
mapping
\begin{equation}\label{eq:defgaussmap}
G(p) = - \vec{n}(p) \qquad p\in N.
\end{equation}
Since $\vec{n}$ is unit, we have that $G: N \to \Sphere^{m,q,\pm 1}$,
the sign depending on whether $\vec{n}$ is time-like or space-like.
This map sending a hypersurface to a standard hyperquadric via the
unit normal vector field is the \emph{Gauss map}, generalising to the
pseudo-Euclidean case the familiar Gauss map for surfaces in
$\Real^3$. 

\begin{rmk}
In \eqref{eq:defgaussmap} we took \emph{minus} the declared unit
normal vector. This is so that when used with our convention that the
normal vectors are inward pointing, the Gauss map reduces to the
identity map for the hyperquadrics $\Sphere^{m,q,\pm 1}$. 
\end{rmk}

The derivative of the Gauss map $\D*{G}$ maps $T_p N$ to $T_{G(p)}
\Sphere^{m,q,\pm 1}$; both tangent spaces are orthogonal to
$\vec{n}(p)$, after the identification of both $T_p M$ and $T_{G(p)}M$
with $M$ itself. This allows us to naturally identify
$T_{G(p)}\Sphere^{m,q,\pm 1}$ with $T_p N$ and hence identify $\D*{G}$
with $- \shapeop$, where $\shapeop$ is the shape operator relative to $\vec{n}$.
This recovers for us, in the setting of hypersurfaces in
pseudo-Euclidean spaces, the familiar relation between the second fundamental
form and the Gauss map for surfaces in $\Real^3$. 

\subsection{The Codazzi equations}
As already seen above in the case of the shape operator and Gauss map
descriptions of the second fundamental form, the second fundamental
form can be schematically
written as the first derivative of a smooth quantity. Now, from
calculus we expect second derivatives to commute, up to lower-order
curvature terms: this gives certain integrability criteria that the
second fundamental form of a submanifold must satisfy. These are the
Codazzi equations. Let $M$ be a pseudo-Riemannian manifold with metric
$g$ and Levi-Civita connection $D$, and let $N$ be a
nondegenerate submanifold with induced metric $h$ with induced
Levi-Civita connection $\covD$, and second
fundamental form $\sff$, we denote, for $W,X,Y$ vector fields along
$N$, 
\[ (\covD^\perp_W \sff)(X,Y) = \proj^\perp D_W(\sff(X,Y)) - \sff(\covD_W
X,Y) - \sff(\covD_W Y,X).\]
Then the \emph{Codazzi equations} read
\begin{equation}\label{eq:codazziorig}
\proj^\perp \Riem[^{(M)}](X,Y)W + (\covD^\perp_X\sff)(Y,W) -
(\covD^\perp_Y\sff)(X,W) = 0
\end{equation}
where $\Riem[^{(M)}]$ is the Riemann curvature tensor of the ambient
manifold $M$. 

Specialising now to the case of a hypersurface in pseudo-Euclidean
space, the ambient curvature vanishes identically, and
\eqref{eq:codazziorig} simplifies to 
\begin{equation}\label{eq:codahyperorig}
(\covD_X \shapeop)(Y) - (\covD_Y \shapeop)(X) = 0
\end{equation}
for the shape operator $\shapeop$ and any tangent vector fields $X,Y$. Now 
supposing our hypersurface $N$ has constant mean curvature, we can take 
the trace of \eqref{eq:codahyperorig} to obtain (in index notation)
\begin{equation}\label{eq:codahyperdual}
\covD_a \shapeop_b^a = 0.
\end{equation}

\subsection{Linearisation of mean curvature}\label{app:linearisationmc}
In the codimension-1 case the following computation is well-known (e.g.\
\cite{Choque1976}); here we start with the generalisation to the case of 
higher codimensions. Let $(M,g)$ be
a pseudo-Riemannian manifold and $(\ooo{N},\ooo{h})$ be an embedded
pseudo-Riemannian manifold (in particular $\ooo{h}$ is not
degenerate). Assume that $M$ has dimension $m$ and $\ooo{N}$
dimension $n$. Then locally in a small neighbourhood $M$ can be
described by the normal bundle $\ooo{N} \times \Real^{m-n}$. A
concrete local diffeomorphism can be obtained by the normal
exponential map on the normal bundle of $\ooo{N}$. 

This gives us a local coordinate system. Let $\phi : \ooo{N} \to
\Real^{m-n}$, this gives us another submanifold of $M$ that is
homotopic to $\ooo{N}$. What is its second fundamental form? We let
$x^1, \dots,x^n$ be a local coordinate system on $\ooo{N}$, and let
$x^{n+1},\dots,x^{m}$ be coordinates for $\Real^{m-n}$. What we
need to compute is the normal projection of $\covD_{\phi_*\partial_i}
\phi_*\partial_j$. We can write
\[ \phi_* \partial_i = \partial_i + \partial_i \phi^\mu \partial_\mu
\]
where Greek indices run from $n+1,\ldots,m$ for the vertical
directions and Latin indices run from $1, \ldots,n$ for the horizontal
directions. So we have that
\begin{multline} 
\covD_{\phi_*\partial_i} \phi_*\partial_j = \Gamma_{ij}^k\partial_k
+ \Gamma_{ij}^\nu \partial_\nu + \Gamma_{\mu j}^k \phi_{,i}^\mu
\partial_k + \Gamma_{\mu j}^\nu \phi_{,i}^\mu \partial_\nu 
+ \phi^{\nu}_{,ij} \partial_\nu \\ + \phi^{\nu}_{,j\mu} \phi^\mu_{,i}
\partial_\nu + \phi^\mu_{,j} \Gamma_{i \mu}^k\partial_k 
+ \phi^\mu_{,j} \Gamma_{i\mu}^\nu \partial_\nu + \phi^\mu_{,i}
\phi^{\rho}_{,j} \Gamma_{\mu\rho}^\nu \partial_\nu +
\phi^\mu_{,i}\phi^{\rho}_{,j} \Gamma_{\mu\rho}^k \partial_k.
\end{multline}
Here, as usual, $\Gamma^{\pholder}_{\pholder\pholder}$ denote the
Christoffel symbol of the metric $g$ relative to the coordinates $x^1,
\dots, x^m$. Note that the Christoffel symbol is evaluated at the
point $(x^1,\dots,x^n,\phi(x^1,\dots,x^n))$ and so implicitly depends
on $\phi$. 
The mean curvature vector is the normal projection of the trace
of the above expression. Note that when $\phi = 0$ this mean curvature
reduces to $\ooo{h}^{ij} \Gamma_{ij}^\nu \partial_\nu $. Keeping only
terms that are linear in $\phi$ gives us the formal linearisation of
the mean curvature (which we will denote by $\delta \mcv^\nu \partial_\nu$).
A direct computation yields that
\[ \delta \mcv^\nu = \ooo{h}^{ij}\left( 2\Gamma^\nu_{\mu j} \phi^\mu_{,i} +
\phi^\nu_{,ij} - \phi^\nu_{,k} \Gamma^k_{ij} +  \phi^\mu
\partial_\mu \Gamma_{ij}^\nu \right) + \phi^\mu \partial_\mu \ooo{h}^{ij}
\Gamma_{ij}^\nu  \]
where we used that
\[ \partial_k = \phi_*\partial_k - \phi^\nu_{,k} \partial_\nu \]
and
\[ \Gamma_{ij}^\nu(\phi) = \Gamma_{ij}^\nu(0) + \partial_\mu
\Gamma_{ij}^\nu \phi^\mu + O(\phi^2) .\]
Now, treating $\phi^\mu$ as a section of the normal bundle, we have
that
\[ \covD_{\partial_i} \phi^A = \partial_i \phi^A +
\Gamma_{i\nu}^A\phi^\nu \]
and
\[ \covD^2_{\partial_i\partial_j} \phi^A = \phi^A_{,ij} +
\partial_j \Gamma_{i\nu}^A \phi^\nu + \Gamma_{i\nu}^A
\phi^\nu_{,j} + \Gamma^A_{j\nu} \phi^{\nu}_{,j} +
\Gamma^{A}_{jB}\Gamma^{B}_{i\nu} \phi^\nu - \Gamma_{ij}^k
\phi^A_{,k} - \Gamma_{ji}^B \Gamma_{B\nu}^A \phi^\nu \]
where $A,B$ stand for both horizontal and vertical directions, with
naturally $\phi^i = 0$.

This implies that the formal linearisation of the mean curvature is
given by
\[ \delta \mcv^\nu = \ooo{h}^{ij} \covD^2_{ij} \phi^\nu + \ooo{h}^{ij}
\Riem[_{\mu ij}^{\nu}]\phi^\mu + \partial_\mu \ooo{h}^{ij} \Gamma_{ij}^\nu
\phi^\mu \]
here the convention for $\Riem$ is that $g^{AB}\Riem[_{CAB}^D] =
\Ricci_C^D$. The derivative $\partial_\mu \ooo{h}_{ij} = 2 g_{\mu\sigma}
\Gamma^\sigma_{ij}$ by assumption of orthogonality and implies finally
\begin{equation}\label{eq:linearMCper}
\delta \mcv^\nu = \ooo{h}^{ij} \covD^2_{ij} \phi^\nu + \ooo{h}^{ij}
\Riem[_{\mu ij}^\nu] \phi^\mu + 2 g_{\mu\rho} \phi^\mu
\Gamma^\rho_{ij} \Gamma^\nu_{kl} h^{ik}h^{jl}~. \end{equation}

In the case of a codimension-1 orientable hypersurface, we can write
$\phi^\nu\partial_\nu =\phi \vec{n}$ where $\vec{n}$ is a field of
unit normal vectors and contract, this gives us that the linearisation
of the mean curvature \emph{scalar} satisfies
\begin{equation}\label{eq:linearMC1}
\delta \mcs = \triangle_{\ooo{h}} \phi + \Ricci[_\nu^\nu] + \phi
\shapeop:\shapeop 
\end{equation}
where $\shapeop$ is the shape operator and the notation
$\shapeop:\shapeop$ is a
shorthand for $\shapeop^i_j \shapeop^j_i$. We lost a factor of two in
the last part because
\[ 0 = \covD g(\vec{n},\vec{n}) = 2 g(\vec{n}, \covD\vec{n}) \implies g(\covD^2\vec{n},\vec{n}) = - g(\covD
\vec{n},\covD\vec{n}). \]

Now let us specialise to the case where the ambient space $M$ is
$\Real^{1,d+1}$ and $\ooo{N}$ is Lorentzian. Since Minkowski space is
flat we can drop the Ricci term and write (switching $\triangle$ to
$\Box$ since the Laplace-Beltrami operator is now a wave operator)
\[ \delta \mcs = \Box_{\ooo{h}} \phi + \shapeop:\shapeop \phi.\]
In the case that $\ooo{N}$ is the hyperquadric $\dS$, we further have
$\shapeop = -\delta$ which gives 
\begin{equation}\label{eq:linearMCdS}
\delta \mcs = \Box_{\dS} \phi + (d+1) \phi.
\end{equation}

\tocaddline
\bibliographystyle{amsalpha}
\bibliography{../../../bib_files/jabrefmaster.bib}

\providecommand{\bysame}{\leavevmode\hbox to3em{\hrulefill}\thinspace}
\providecommand{\MR}{\relax\ifhmode\unskip\space\fi MR }
\providecommand{\MRhref}[2]{%
  \href{http://www.ams.org/mathscinet-getitem?mr=#1}{#2}
}
\providecommand{\href}[2]{#2}
\begin{thebibliography}{DKSW13}

\bibitem[AC79a]{AurChr1979}
A.~Aurilia and D.~Christodoulou, \emph{Theory of strings and membranes in an
  external field. {I}. {G}eneral formulation}, J. Math. Phys. \textbf{20}
  (1979), no.~7, 1446--1452. \MR{538719 (84a:35184)}

\bibitem[AC79b]{AurChr1979a}
\bysame, \emph{Theory of strings and membranes in an external field. {II}.
  {T}he string}, J. Math. Phys. \textbf{20} (1979), no.~8, 1692--1699.
  \MR{543904 (80h:83026)}

\bibitem[Ali01a]{Alinha2001}
S.~Alinhac, \emph{The null condition for quasilinear wave equations in two
  space dimensions {I}}, Invent. Math. \textbf{145} (2001), no.~3, 597--618.
  \MR{1856402 (2002i:35127)}

\bibitem[Ali01b]{Alinha2001a}
\bysame, \emph{The null condition for quasilinear wave equations in two space
  dimensions. {II}}, Amer. J. Math. \textbf{123} (2001), no.~6, 1071--1101.
  \MR{1867312 (2003e:35193)}

\bibitem[Ali03]{Alinha2003}
Serge Alinhac, \emph{An example of blowup at infinity for a quasilinear wave
  equation}, Ast\'erisque (2003), no.~284, 1--91, Autour de l'analyse
  microlocale. \MR{2003417 (2005a:35197)}

\bibitem[Bre02]{Brendl2002}
Simon Brendle, \emph{Hypersurfaces in {M}inkowski space with vanishing mean
  curvature}, Comm. Pure Appl. Math. \textbf{55} (2002), no.~10, 1249--1279.
  \MR{1912097 (2003j:58043)}

\bibitem[CB76]{Choque1976}
Y.~Choquet-Bruhat, \emph{Maximal submanifolds and submanifolds with constant
  mean extrinsic curvature of a {L}orentzian manifold}, Ann. Scuola Norm. Sup.
  Pisa Cl. Sci. (4) \textbf{3} (1976), no.~3, 361--376. \MR{0423405 (54
  \#11384)}

\bibitem[CH62]{CouHil1962}
R.~Courant and D.~Hilbert, \emph{Methods of mathematical physics. {V}ol. {II}:
  {P}artial differential equations}, (Vol. II by R. Courant.), Interscience
  Publishers (a division of John Wiley \& Sons), New York-London, 1962.
  \MR{0140802 (25 \#4216)}

\bibitem[Chr86]{Christ1986}
Demetrios Christodoulou, \emph{Global solutions of nonlinear hyperbolic
  equations for small initial data}, Comm. Pure Appl. Math. \textbf{39} (1986),
  no.~2, 267--282. \MR{820070 (87c:35111)}

\bibitem[Chr00]{Christ2000}
\bysame, \emph{The action principle and partial differential equations}, Annals
  of Mathematics Studies, Princeton University Press, 2000.

\bibitem[Chr09]{Christ2009}
\bysame, \emph{The formation of black holes in general relativity}, EMS
  monographs in mathematics, European Mathematical Society, 2009.

\bibitem[CK93]{ChrKla1993}
Demetrios Christodoulou and Sergiu Klainerman, \emph{The global nonlinear
  stability of the {Minkowski} space}, Princeton University Press, 1993.

\bibitem[DKSW13]{DoKrSW2013}
Roland Donninger, Joachim Krieger, Jeremie Szeftel, and Willie Wai-Yeung Wong,
  \emph{Codimension one stability of the catenoid under the vanishing mean
  curvature flow in {M}inkowski space}, Pre-print (2013), Available
  arXiv:1310.5606.

\bibitem[FB52]{Four1952}
Yvonne Four{\`e}s-Bruhat, \emph{Th{\'e}or{\`e}me d'existence pour certains
  syst{\`e}mes d'{\'e}quations aux d{\'e}riv{\'e}es partielles non
  lin{\'e}aires}, Acta Math. \textbf{88} (1952), no.~1, 141 -- 225.

\bibitem[FCS80]{FisSch1980}
Doris Fischer-Colbrie and Richard Schoen, \emph{The structure of complete
  stable minimal surfaces in {$3$}-manifolds of nonnegative scalar curvature},
  Comm. Pure Appl. Math. \textbf{33} (1980), no.~2, 199--211. \MR{562550
  (81i:53044)}

\bibitem[Fri86]{Friedr1986}
Helmut Friedrich, \emph{Existence and structure of past asymptotically simple
  solutions of {E}instein's field equations with positive cosmological
  constant}, J. Geom. Phys. \textbf{3} (1986), no.~1, 101--117. \MR{855572
  (88c:83006)}

\bibitem[HKM76]{HuKaMa1976}
Thomas J.~R. Hughes, Tosio Kato, and Jerrold~E. Marsden, \emph{Well-posed
  quasi-linear second-order hyperbolic systems with applications to nonlinear
  elastodynamics and general relativity}, Arch. Rational Mech. Anal.
  \textbf{63} (1976), no.~3, 273--294 (1977). \MR{MR0420024 (54 \#8041)}

\bibitem[Hop13]{Hoppe2013}
Jens Hoppe, \emph{Relativistic membranes}, J. Phys. A \textbf{46} (2013),
  no.~2, 023001, 30. \MR{3005897}

\bibitem[HS13]{HadSpe2013}
Mahir Hadzic and Jared Speck, \emph{The global future stability of the flrw
  solutions to the dust-einstein system with a positive cosmological constant},
  Pre-print (2013).

\bibitem[Kat75]{Kato1975}
Tosio Kato, \emph{The {C}auchy problem for quasi-linear symmetric hyperbolic
  systems}, Arch. Ration. Mech. Anal. \textbf{58} (1975), no.~3, 181--205.

\bibitem[Kla80]{Klaine1980}
Sergiu Klainerman, \emph{Global existence for nonlinear wave equations}, Comm.
  Pure Appl. Math. \textbf{33} (1980), no.~1, 43--101. \MR{544044 (81b:35050)}

\bibitem[Kla86]{Klaine1986}
S.~Klainerman, \emph{The null condition and global existence to nonlinear wave
  equations}, Nonlinear systems of partial differential equations in applied
  mathematics, {P}art 1 ({S}anta {F}e, {N}.{M}., 1984), Lectures in Appl.
  Math., vol.~23, Amer. Math. Soc., Providence, RI, 1986, pp.~293--326.
  \MR{837683 (87h:35217)}

\bibitem[KM95]{KlaMac1995}
S.~Klainerman and M.~Machedon, \emph{Finite energy solutions of the
  {Y}ang-{M}ills equations in {$R^{3+1}$}}, Ann. of Math. (2) \textbf{142}
  (1995), no.~1, 39--119.

\bibitem[Lin04]{Lindbl2004}
Hans Lindblad, \emph{A remark on global existence for small initial data of the
  minimal surface equation in {M}inkowskian space time}, Proc. Amer. Math. Soc.
  \textbf{132} (2004), no.~4, 1095--1102 (electronic). \MR{2045426
  (2005a:35203)}

\bibitem[Lin08]{Lindbl2008}
\bysame, \emph{Global solutions of quasilinear wave equations}, Amer. J. Math.
  \textbf{130} (2008), no.~1, 115--157. \MR{2382144 (2009b:58062)}

\bibitem[LR05]{LinRod2005}
Hans Lindblad and Igor Rodnianski, \emph{Global existence for the {E}instein
  vacuum equations in wave coordinates}, Comm. Math. Phys. \textbf{256} (2005),
  no.~1, 43 -- 110.

\bibitem[LR10]{LinRod2010}
\bysame, \emph{The global stability of {M}inkowski space-time in harmonic
  gauge}, Ann. of Math. (2) \textbf{171} (2010), no.~3, 1401--1477. \MR{2680391
  (2011k:58042)}

\bibitem[MRS10]{MeRaSz2010}
Frank Merle, Pierre Rapha{\"e}l, and Jeremie Szeftel, \emph{Stable self-similar
  blow-up dynamics for slightly {$L^2$} super-critical {NLS} equations}, Geom.
  Funct. Anal. \textbf{20} (2010), no.~4, 1028--1071. \MR{2729284
  (2011m:35294)}

\bibitem[MSBV14]{MeSaVa2014}
Richard Melrose, Ant\^onio S\'a~Barreto, and And\'as Vasy, \emph{Asymptotics of
  solutions of the wave equation on {de Sitter-Schwarzschild} space}, Comm.
  Partial Differential Equations \textbf{39} (2014), no.~3, 512--529.

\bibitem[NT13]{NguTia2013}
Luc Nguyen and Gang Tian, \emph{On smoothness of timelike maximal cylinders in
  three-dimensional vacuum spacetimes}, Classical Quantum Gravity \textbf{30}
  (2013), no.~16, 165010, 26. \MR{3094875}

\bibitem[O'N83]{ONeil1983}
Barrett O'Neill, \emph{{Semi-Riemannian} geometry: with applications to
  relativity}, Academic Press, 1983.

\bibitem[Rin08]{Ringst2008}
Hans Ringstr{\"o}m, \emph{Future stability of the {E}instein-non-linear scalar
  field system}, Invent. Math. \textbf{173} (2008), no.~1, 123--208.
  \MR{2403395 (2009c:53097)}

\bibitem[Rin09]{Ringst2009}
\bysame, \emph{The {C}auchy problem in general relativity}, ESI Lectures in
  Mathematics and Physics, European Mathematical Society (EMS), Z\"urich, 2009,
  see also the erratum ``Existence of a maximal globally hyperbolic
  development'', available http://www.math.kth.se/~hansr/mghd.pdf. \MR{2527641
  (2010j:83001)}

\bibitem[RR12]{RapRod2012}
Pierre Rapha{\"e}l and Igor Rodnianski, \emph{Stable blow up dynamics for the
  critical co-rotational wave maps and equivariant {Y}ang-{M}ills problems},
  Publ. Math. Inst. Hautes \'Etudes Sci. (2012), 1--122. \MR{2929728}

\bibitem[RS13]{RodSpe2013}
Igor Rodnianski and Jared Speck, \emph{The nonlinear future stability of the
  {FLRW} family of solutions to the irrotational {E}uler-{E}instein system with
  a positive cosmological constant}, J. Eur. Math. Soc. (JEMS) \textbf{15}
  (2013), no.~6, 2369--2462. \MR{3120746}

\bibitem[Spe12]{Speck2012}
Jared Speck, \emph{The nonlinear future stability of the {FLRW} family of
  solutions to the {E}uler-{E}instein system with a positive cosmological
  constant}, Selecta Math. (N.S.) \textbf{18} (2012), no.~3, 633--715.
  \MR{2960029}

\bibitem[Spe13]{Speck2013}
\bysame, \emph{The stabilizing effect of spacetime expansion on relativistic
  fluids with sharp results for the radiation equation of state}, Arch. Ration.
  Mech. Anal. \textbf{210} (2013), no.~2, 535--579. \MR{3101792}

\bibitem[Tay11]{Taylor2011}
Michael~E. Taylor, \emph{Partial differential equations {III}. {N}onlinear
  equations}, second ed., Applied Mathematical Sciences, vol. 117, Springer,
  New York, 2011. \MR{2744149 (2011m:35003)}

\bibitem[Wei85]{Weinst1985}
Michael~I. Weinstein, \emph{Modulational stability of ground states of
  nonlinear {S}chr\"odinger equations}, SIAM J. Math. Anal. \textbf{16} (1985),
  no.~3, 472--491. \MR{783974 (86i:35130)}

\bibitem[Wei86]{Weinst1986}
Michael~I. Weinstein, \emph{Lyapunov stability of ground states of nonlinear
  dispersive evolution equations}, Comm. Pure Appl. Math. \textbf{39} (1986),
  no.~1, 51 -- 67.

\bibitem[Won11]{Wong2011}
Willie Wai-Yeung Wong, \emph{Regular hyperbolicity, dominant energy condition
  and causality for {L}agrangian theory of maps}, Classical Quant. Grav.
  \textbf{28} (2011), no.~21, 215008.

\end{thebibliography}

\end{document}